\documentclass[letterpaper,12pt]{article}
\usepackage{amsmath,amssymb,amsthm,graphicx,xcolor,mathrsfs}
\usepackage{setspace}
\usepackage{bbm, bm}
\usepackage{url}
\usepackage[colorlinks,linkcolor=black, citecolor=blue,urlcolor=blue]{hyperref}
\usepackage{natbib}
\usepackage{varwidth}
\usepackage[ruled, linesnumbered, lined, Algorithm]{algorithm}
\usepackage{multirow}
\usepackage{cancel}
\usepackage[normalem]{ulem}
\usepackage{subfig, float}
\usepackage{rotating}
\usepackage{array}
\allowdisplaybreaks
\usepackage{fancyvrb} 
\usepackage{tikz}
\usepackage{verbatim}
\usepackage{subfig}
  
\usepackage{algpseudocode}
 
\newcommand*\circled[1]{\tikz[baseline=(char.base)]{
		\node[shape=circle, draw, inner sep=1pt, 
	 	minimum height=12pt] (char) {#1};}}
\newcommand{\PreserveBackslash}[1]{\let\temp=\\#1\let\\=\temp}
\newcolumntype{C}[1]{>{\PreserveBackslash\centering}p{#1}}
\renewcommand{\d}{\mathrm{d}}

\usepackage[latin1]{inputenc}

\numberwithin{equation}{section}
\theoremstyle{plain}
\newtheorem{Theorem}{Theorem}
\numberwithin{Theorem}{section}

\newtheorem{corollary}{Corollary}
\numberwithin{corollary}{section}
\newtheorem{proposition}{Proposition}
\numberwithin{proposition}{section}

\newtheorem{Lemma}{Lemma}
\numberwithin{Lemma}{section}
{
	\theoremstyle{definition}

	\numberwithin{example}{section}
	
	\numberwithin{fact}{section}
	\newtheorem{remark}{Remark}
}

\newcommand{\Pro}{\mathbb{P}}
\newcommand{\1}{\mathbbm{1}}

\def\bth{{\boldsymbol{\theta}}}

\def\Pr{\mathbb P}

\DeclareMathOperator*{\argmin}{arg\,min}

\def\R{{\mathbb R}}

\def\det{{\rm det}}
\def\Var{{\rm Var}}

\def\E{\mathbb{E}}

\def\spacingset#1{\renewcommand{\baselinestretch}%
	{#1}\small\normalsize} \spacingset{1}

\begin{document}

\title{Optimal Differentially Private Ranking \\from Pairwise Comparisons}
\author{T. Tony Cai, Abhinav Chakraborty, and Yichen Wang \thanks{T. T. Cai and A. Chakraborty are with the Department of Statistics and Data Science, The Wharton School, University of Pennsylvania. The research of T.T. Cai was supported in part by NSF Grant DMS-2015259 and NIH grant R01-GM129781. Y. Wang is an independent researcher.}
}

\date{\today}

\maketitle

\begin{abstract}
Data privacy is a central concern in many applications involving ranking from incomplete and noisy pairwise comparisons, such as recommendation systems, educational assessments, and opinion surveys on sensitive topics. In this work, we propose differentially private algorithms for ranking based on pairwise comparisons. Specifically, we develop and analyze ranking methods under two privacy notions: edge differential privacy, which protects the confidentiality of individual comparison outcomes, and individual differential privacy, which safeguards potentially many comparisons contributed by a single individual. Our algorithms--including a perturbed maximum likelihood estimator and a noisy count-based method--are shown to achieve minimax optimal rates of convergence under the respective privacy constraints. We further demonstrate the practical effectiveness of our methods through experiments on both simulated and real-world data.

\end{abstract}
\noindent{\bf Keywords:\/} differential privacy, ranking, Bradley-Terry-Luce model, minimax optimality

\spacingset{1.75}
\vspace{-5mm}
\section{Introduction} 
As personal data are more extensively collected and analyzed than ever, the importance of privacy protection in data analysis is also increasingly recognized. In this paper, we consider privacy-preserving methods for ranking from pairwise comparisons. In this ranking problem, the data analyst observes random and incomplete pairwise comparisons among items that follow some unknown ranking, with higher ranked items more likely (but not guaranteed) to prevail over lower ranked ones. The analyst then tries to infer the underlying ranking from the noisy comparison results. The extensive applied research on this topic attests to its practical relevance in many settings.
\begin{itemize}
	
	\item \textbf{Pairwise comparison in sensitive survey data:} Pairwise comparisons in surveys provide a structured approach for respondents to express their relative preferences across a wide range of options. For instance, a survey analyzed by \cite{weber2011novel} was conducted to gauge public sentiment toward immigration. The study involved 98 student participants, each of whom responded to at least one pairwise comparison. These comparisons were constructed from a set of four extreme statements regarding immigrants.
		
	\item \textbf{Pairwise comparison in recommendation systems.} Pairwise comparisons are widely used in recommendation systems that capture users' preferences between pairs of items such as movies, books, or other consumer products. For example, \cite{balakrishnan2012two} proposed a method in which customers are presented with a series of paired preference questions (e.g. ``Do you prefer item A over item B?''), allowing the system to infer personalized rankings based on comparative judgments.
	
	\item \textbf{Pairwise comparison in education.} Pairwise comparison can serve as an effective tool for educational assessment. For instance, \cite{heldsinger2010using} describes a study in which teachers employed a pairwise comparison procedure to grade student writing samples and construct a performance scale. The study found that teacher judgments were highly internally consistent and showed strong correlation with scores from a large-scale standardized testing program administered to the same group of students.	
\end{itemize}
Privacy is a critical concern in many applications involving ranking from pairwise comparisons. For example, in the educational assessment study by \cite{heldsinger2010using}, teachers'  preferences between pairs of student assignments should remain confidential. Similarly, in the survey conducted by \cite{weber2011novel}, respondents were asked to express preferences between pairs of political positions--data considered sensitive due to the controversial nature of the statements and the potential ramifications for individuals if their views were made public.

Motivated by the importance of protecting privacy in such settings, we develop \textit{statistically optimal} algorithms for ranking from pairwise comparisons under \textit{differential privacy (DP) constraints}.  Differential privacy \cite{dwork2006calibrating, dwork2006our} is the most widely adopted framework for privacy-preserving data analysis, offering strong formal guarantees that the output of an algorithm reveals provably little about any individual data point. At the same time, it supports the design and implementation of efficient algorithms. In this paper, we propose and analyze DP algorithms for ranking from pairwise comparisons and show that they achieve statistical optimality under the DP constraint: no other differentially private algorithm can attain a faster rate of convergence to the true ranking.

\subsection{Main Results and Our Contribution}

{\bf Edge DP and individual DP}. While the technical definitions are deferred to Sections \ref{sec: edge dp problem formulation} and \ref{sec: individual dp problem formulation}, we briefly provide an intuitive overview of differential privacy (DP). A differentially private algorithm guarantees that its outputs on two adjacent datasets are statistically similar. Two datasets are considered adjacent if they differ by exactly one unit of data. In essence, a DP algorithm ensures that its output does not depend significantly on any single data point, thereby preventing inference about the presence or absence of any individual unit in the input.

However, as discussed in \cite{NBERc15017}, the definition of adjacency--and thus the unit of data--depends on the specific context and data analysis task. For instance, in a census, a unit might be a person or a household, while in business analytics, it could be a single transaction. In the context of ranking from pairwise comparisons, two natural notions of adjacency arise, leading to two corresponding definitions of differential privacy:
\begin{itemize}
	\item Edge DP: Two datasets are adjacent if they differ by a single pairwise comparison outcome. 
	\item Individual DP: Two datasets are adjacent if they differ in all the comparisons contributed by a single individual.
\end{itemize}
 To the best of our knowledge, this paper is the first to simultaneously consider and rigorously analyze both definitions of DP in the context of ranking from pairwise comparisons.
 
{\bf Optimal parametric estimation with differential privacy}. Under the parametric Bradley-Terry-Luce (BTL) model for pairwise comparisons \cite{bradley1952rank, luce1959individual}, we introduce in Section \ref{sec: ranking upper bound} a perturbed maximum likelihood estimator of the following form:
 \begin{align*}
 	\widetilde \bth = \argmin_{\bth \in \R^n} \mathcal L(\bth; y) + \frac{\gamma}{2}\|\bth\|_2^2 + \bm w^\top \bth, \quad \bm w = (w_1, w_2, \cdots, w_n) \stackrel{\text{i.i.d.}}{\sim }\text{Laplace}(\lambda),
 \end{align*}
where $\mathcal L(\bth, y)$ is the likelihood function, $\bth$ is the vector of latent parameters which determine the items' ranks, and $y$ is the data set of pairwise comparisons. 

We then show that the estimator respectively satisfies edge DP and individual DP for two different choices of $(\gamma, \lambda)$, in Sections \ref{sec: ranking upper bound} and \ref{sec: individual parametric ranking}. It is further shown by Theorems \ref{thm: ranking lower bound} and \ref{thm: individual-privacy-lower} that the estimator's respective rates of convergence under edge DP and individual DP are optimal.
 
{\bf Optimal nonparametric ranking with differential privacy}. In the absence of parametric assumptions, we show that ranking items based on noisy counts of wins--with appropriately calibrated noise distributions--satisfies either edge differential privacy or individual differential privacy, depending on the setting. Moreover, this simple approach is statistically optimal: it accurately ranks items across the broadest possible regime of sample sizes and privacy levels.  Specifically, in Theorems \ref{thm: ranking nonparametric upper bound} and \ref{thm: nonparametric ranking lower bound}, we show that ranking by noisy counts succeeds at identifying the top $k$ items whenever the strengths of the $k$-th and $(k+1)$-th items are separated by a certain threshold, and no other edge DP algorithm will succeed below this threshold. Theorems \ref{thm: individual ranking nonparametric upper bound} and \ref{thm: individual nonparametric ranking lower bound} contain similar results for individual DP.

{\bf Entrywise analysis of DP algorithms}. The study of differentially private ranking also yields insights of broader methodological interest. In particular, the perturbed maximum likelihood estimator (MLE) achieves differential privacy by injecting a single dose of noise into the objective function. This design facilitates a leave-one-out analysis \cite{chen2019spectral}, enabling precise control of entrywise errors in optimization problems. While most prior work on differentially private optimization has focused on bounding $\ell_2$ errors, our framework extends the analysis to the entrywise setting. Furthermore, our adaptation of the score attack technique \cite{cai2023score}--originally developed for $\ell_2$  risk--to entrywise error analysis may have broader applicability in establishing entrywise lower bounds for differentially private algorithms more generally.

\vspace{-5mm}
\subsection{Related Work}

Some of the most historically significant contributions to the study of pairwise comparisons and ranking include \cite{thurstone1927law}, which pioneered the use of pairwise comparisons for measuring psychological values, and the seminal works of \cite{bradley1952rank} and \cite{luce1959individual}, which introduced the Bradley-Terry-Luce (BTL) model. Additionally, \cite{ford1957solution} was the first to formulate the ranking problem as a maximum likelihood estimation task.

In recent years, there has been growing interest in understanding the minimax rates of convergence for ranking from pairwise comparisons. Several works--including \cite{wauthier2013efficient, rajkumar2014statistical, negahban2017rank, shah2015estimation, chen2015spectral, chen2019spectral}--adopt parametric assumptions, often relying on the BTL model, to study the minimax $\ell_2$ or $\ell_\infty$  risk in estimating underlying parameters. Parallel to this, a nonparametric line of research focuses on identifying the top-ranked items \cite{shah2017simple, chen2017competitive} or estimating pairwise comparison probabilities under the assumption of stochastic transitivity \cite{shah2016stochastically, shah2019feeling, pananjady2020worst}.

On the other hand, the trade-off between differential privacy and statistical utility has also attracted substantial attention. A wide array of differentially private statistical methods have been proposed and analyzed--ranging from Gaussian mean estimation and linear regression \cite{cai2021cost}, to nonparametric density estimation \cite{wasserman2010statistical}, M-estimators \cite{lei2011differentially}, and PCA \cite{dwork2014analyze}. These methods are often grounded in foundational design paradigms such as the Laplace and Gaussian mechanisms \cite{dwork2006calibrating, dwork2014algorithmic}, and private convex optimization techniques \cite{chaudhuri2009privacy, chaudhuri2011differentially, bassily2014private, kifer2012private, bassily2019private}.

Specifically on differentially private ranking, existing works such as \cite{shang2014application, hay2017differentially, yan2020private, song2022distributed} focus on the related--but distinct--problem of rank aggregation, where the goal is to combine multiple full rankings into a single consensus ranking. \cite{qiao2021oneshot} proposes a one-shot algorithm for ranking from pairwise comparisons as an application of their private top-$k$ selection method; we compare this with our algorithms in Section \ref{sec: experiments}. \cite{cai2023score} studies a special case of our edge DP setting under $\ell_2$  loss, but does not directly address the ranking problem considered here.

Understanding the privacy-utility trade-off requires identifying the minimum possible loss in accuracy among all differentially private methods. To this end, several powerful lower-bounding techniques have been developed, including tracing attacks \cite{bun2014fingerprinting, dwork2015robust, steinke2017between, steinke2017tight, kamath2018privately, cai2021cost} and the score attack \cite{cai2020cost, cai2023score}, as well as differentially private versions of Le Cam's, Fano's, and Assouad's inequalities \cite{barber2014privacy, karwa2017finite, canonne2018structure, acharya2018differentially, acharya2021differentially}.

\vspace{-5mm}
\subsection{Organization}

The rest of the paper is organized as follows. Section \ref{sec: edge dp ranking} focuses on ranking under edge differential privacy (DP). Section \ref{sec: ranking upper bound} presents the privacy and optimality analysis of the parametric estimator, while Section \ref{sec: edge dp nonparametric ranking} addresses the privacy guarantees and optimality of ranking by noisy counts under edge DP. Section \ref{sec: individual-dp-ranking} turns to ranking under individual DP, with Sections \ref{sec: individual parametric ranking} and \ref{sec: individual-nonparametric} dedicated to the parametric and nonparametric analyses, respectively. Section \ref{sec: experiments} provides a numerical evaluation of our algorithms using both simulated and real datasets. Section \ref{sec: discussion} concludes the paper with a discussion of open questions and directions for future work. Additional technical details and omitted proofs are provided in the supplementary material \cite{cai2023optimal}.

\vspace{-5mm}
\subsection{Notation}

 For real-valued sequences $\{a_n\}, \{b_n\}$, we write $a_n \lesssim b_n$ if $a_n \leq cb_n$ for some universal constant $c \in (0, \infty)$, and $a_n \gtrsim b_n$ if $a_n \geq c'b_n$ for some universal constant $c' \in (0, \infty)$.  We say $a_n \asymp b_n$ if $a_n \lesssim b_n$ and $a_n \gtrsim b_n$. $c, C, c_0, c_1, c_2, \cdots, $ and so on refer to universal constants in the paper, with their specific values possibly varying from place to place. For a positive integer $n$, let $[n] = \{1, 2, 3, \cdots, n\}$.

\vspace{-5mm}
\section{Optimal Ranking under Edge Differential Privacy}\label{sec: edge dp ranking}
\vspace{-5mm}
\subsection{Problem Formulation under Edge DP}\label{sec: edge dp problem formulation}
{\bf The ranking problem}. There are $n$ distinct items, represented by indices in $[n] = \{1, 2, 3, \cdots, n\}$. Pairwise comparisons between items are observed randomly and independently, where each pair $(i, j)$, $1 \leq i < j \leq n$, is compared with a known probability $p \in (0, 1]$. This results in an Erd\H{o}s--R\'enyi random graph $\mathcal{G}$ with $n$ nodes and the observed comparisons constituting the edges. Every observed pair $(i, j)$ determines a unique winner, symbolized by the outcome $Y_{ij} \in \{0, 1\}$, satisfying $Y_{ij} + Y_{ji} = 1$. Consequently, for $i<j$, the random variable $Y_{ij}$ follows an independent Bernoulli distribution with parameter $\rho_{ij} \in [0, 1]$, and the requirement $Y_{ij} + Y_{ji} = 1$ implies $\rho_{ij} + \rho_{ji} = 1$. We assume $\rho_{ii} = 1/2$ for clarity.
 
Our objective is to rank the set of $n$ items based on the average true winning probabilities when compared to randomly selected counterparts. This average winning probability for each item $i \in [n]$ is formally denoted by $\tau_i = \frac{1}{n}\sum_{j \in [n]} \rho_{ij}$. We are interested in estimating the index set $\mathcal S_k$, where $\mathcal S_k = \{i \in [n]: \tau_{i} \text{ is among the top-$k$ largest of } \tau_1, \tau_2, \ldots, \tau_n\}$ for a predetermined $k \in [n]$.

{\bf Parametric and nonparametric models}. The ranking problem is studied under two models. The first one is a parametric model where $\rho_{ij} = F(\theta^*_i - \theta^*_j)$. Each item $i \in [n]$ is assigned a latent parameter $\theta^*_i$, and $F: \mathbb R \to [0, 1]$ is a predetermined link function. This model generalizes well-known the Bradley-Terry-Luce (BTL) model \cite{bradley1952rank, luce1959individual} for pairwise comparison, and recovers the BTL model when $F$ is the standard logistic CDF. With this parametric assumption, the ranking of $\tau_i$ is equivalent to the ranking of $\theta_i^*$, which further reduces to estimating real-valued parameters $\{\theta^*_i\}_{i \in [n]}$. 

The second model is nonparametric, in which we do not assume any parametric form for the $\rho_{ij}$ values, and instead aim to estimate the ranks directly. This nonparametric ranking problem is the focus of a more recent line of work \cite{shah2017simple, chen2017competitive, shah2017simple, shah2019feeling}. We define and study these two settings in Sections \ref{sec: ranking upper bound} and \ref{sec: edge dp nonparametric ranking} respectively.

{\bf Edge differential privacy}. Under these models, we study ranking algorithms satisfying $(\varepsilon, \delta)$ differential privacy ($(\varepsilon, \delta)$-DP). The formal definition of $(\varepsilon, \delta)$-DP requires that, for an algorithm $M$ taking values in some domain $\mathcal R$ and every measurable subset $A \subseteq \mathcal R$,
\begin{align*}
	\Pro(M(X) \in A) \leq e^\varepsilon \cdot \Pro(M(X') \in A) + \delta
\end{align*}
for any pair of data sets $X$ and $X'$ which differ by one element. A pair of data sets is called ``adjacent" if they differ by exactly one element. For example, if $X, X'$ are sets of real numbers, $X = \{x_1, x_2, \ldots, x_n\} \in \R^n $ and $X' = \{x'_1, x'_2, \ldots, x'_n\} \in \R^n$, then $X$ and $X'$ are adjacent if $|X \cap (X')^c| = |X^c \cap X'| = 1$.

When specialized to pairwise comparison data, one natural interpretation of``adjacency" is that two sets of comparison outcomes differ by a single comparison. Concretely, two sets of comparison outcomes $\bm Y = \{Y_{ij}\}_{(i, j) \in \mathcal G}$ and $\bm Y' = \{Y'_{ij}\}_{(i, j) \in \mathcal G'}$ are adjacent if they satisfy one of the two (disjoint) scenarios.
\begin{itemize}
	\item The comparison graphs are identical, $\mathcal G = \mathcal G'$, and there exists exactly one edge $(i^*, j^*) \in \mathcal G$ on which the comparison outcomes differ, $Y_{i^*j^*} \neq Y'_{i^*j^*}$. All other comparison outcomes are identical: $Y_{ij} = Y'_{ij}$ for $(i, j) \neq (i^*, j^*)$.
	\item The comparison graphs $\mathcal G$ and $\mathcal G'$ differ by exactly one edge: there exist $a^*, b^*, c^*, d^* \in [n]$ and $(a^*, b^*) \neq (c^*, d^*)$, such that
	\begin{align*}
		\mathcal G = \mathcal G \cap \mathcal G' + \{(a^*, b^*) \}, \mathcal G' = \mathcal G \cap \mathcal G' + \{(c^*, d^*) \}.
	\end{align*}
	The comparison outcomes  $\bm Y = \{Y_{ij}\}_{(i, j) \in \mathcal G}$ and $\bm Y' = \{Y'_{ij}\}_{(i, j) \in \mathcal G'}$ satisfy $Y_{ij} = Y'_{ij}$ for all $(i, j) \in \mathcal G \cap \mathcal G'$.
\end{itemize} 
This notion of adjacency and the corresponding definition of differential privacy is akin to ``edge differential privacy'' for graphs \cite{nissim2007smooth, kasiviswanathan2013analyzing}, and we adopt the same term for our case.

\vspace{-5mm}
\subsection{Parametric Estimation under Edge DP}
\label{sec: ranking upper bound}

We first study ranking from pairwise comparisons under parametric assumptions: each item $i \in [n]$ is associated with a latent parameter $\theta^*_i$, and the pairwise probability $\rho_{ij}$ is related to the latent parameters of items $i, j$ by a known increasing function $F: \R \to [0, 1]$, specifically $\rho_{ij} = F(\theta^*_i - \theta^*_j)$.  These assumptions conveniently reduce the problem of ranking $n$ items by their average winning probability against peers, $\tau_i = n^{-1}\sum_{j \in [n]} \rho_{ij}$, to the problem of estimating $\bth^* = \left(\theta^*_i\right)_{i \in [n]}$.

\subsubsection{The Edge DP Algorithm for Parametric Estimation}\label{sec: edge dp ranking upper bound}
For constructing a differentially private estimator of $\bth^*$, our approach is to minimize a randomly perturbed and $\ell_2$-penalized version of the negative log-likelihood function. For a vector $\bm v \in \R^n$, indices $i, j \in [n]$ and a given link function $F$, let $F_{ij}(\bm v) = F(v_i - v_j)$ and $F'_{ij}(\bm v) = F'(v_i - v_j)$. The negative log-likelihood function is given by
\begin{align}\label{eq: ranking likelihood equation}
	\mathcal L(\bth; y) = \sum_{(i, j) \in \mathcal G} -y_{ij}\log F_{ij}(\bth) - y_{ji}\log(1 - F_{ij}(\bth)).
\end{align}

\vspace{-5.5mm}
The estimator is defined by Algorithm \ref{alg: DP parametric ranking}.
\begin{algorithm}[!htbp]
	\caption{Differentially Private Ranking for parametric models} 
	\label{alg: DP parametric ranking}
	\textbf{Input}: Comparison data $(y_{ij})_{(i,j)\in \mathcal{G}}$, comparison graph $\mathcal{G}$, privacy parameter $\varepsilon$, regularity constants $\kappa_1,\kappa_2$ defined in~\eqref{eq: F assumption 1} and~\eqref{eq: F assumption 2}.
	\begin{algorithmic}[1]
		\State Set $\lambda \geq \frac{8\kappa_1}{\varepsilon}$ and $\gamma \geq \frac{4\kappa_2}{\varepsilon}$ and generate 
		$
		\bm w = (w_1, w_2, \cdots, w_n) \stackrel{\text{i.i.d.}}{\sim }\text{Laplace}(\lambda).
		$
		\State Solve for
		\begin{align}\label{eq: ranking MLE definition}
			\widetilde \bth = \argmin_{\bth \in \R^n} \mathcal L(\bth; y) + \frac{\gamma}{2}\|\bth\|_2^2 + \bm w^\top \bth, \quad \bm w = (w_1, w_2, \cdots, w_n) \stackrel{\text{i.i.d.}}{\sim }\text{Laplace}(\lambda).
		\end{align}
	\end{algorithmic} 
	\textbf{Output}: $\widetilde \bth$.
\end{algorithm}

Some regularity conditions on the function $F$ will be helpful throughout our analysis of $\widetilde \bth$. We collect them here for convenience.
\begin{enumerate}
	\item[(A0)] $F: \R \to [0, 1]$ is strictly increasing and satisfies $F(x) = 1 - F(-x)$ for every $x \in \mathbb \R$.
	\item[(A1)] \label{itm: F assumption 1} There is an absolute constant $\kappa_1 > 0$ such that 
	\begin{align}\label{eq: F assumption 1}
		\sup_{x \in \R} \left|\frac{F'(x)}{F(x)(1 - F(x))}\right| = \sup_{x \in \R} \frac{F'(x)}{F(x)(1 - F(x))} < \kappa_1.
	\end{align}
	\item[(A2)]  \label{itm: F assumption 2} $\frac{\partial^2}{\partial x^2} \left(-\log F(x)\right) > 0$ for every $x \in \R$, and there exists an absolute constant $\kappa_2 > 0$ such that
	\begin{align}\label{eq: F assumption 2}
		\frac{\partial^2}{\partial x^2} \left(-\log F(x)\right) < \kappa_2, \quad \min_{|x| \leq 4} \frac{\partial^2}{\partial x^2} \left(-\log F(x)\right) > \frac{1}{\kappa_2}.
	\end{align}
\end{enumerate}
In particular, choosing $F$ to be the standard logistic CDF satisfies these conditions and recovers the BTL model. 

Returning to the estimator \eqref{eq: ranking MLE definition}, the random perturbation $\bm w^\top \bth$ is an instance of objective perturbation methods in differentially private optimization \cite{chaudhuri2011differentially, kifer2012private}. Let $\mathcal R(\bth; y)$ denote the regularized log-likelihood part, $\mathcal R(\bth; y) = \mathcal L(\bth; y) + \frac{\gamma}{2}\|\bth\|_2^2$, then $\widetilde \bth$ amounts to the solution of a noisy stationary condition $\nabla \mathcal R(\widetilde \bth; y) = -\bm w$. The solution $\widetilde \bth = \widetilde \bth(y)$ is differentially private when
\begin{itemize}
	\item the scale parameter $\lambda$ of noise vector $\bm w$ is sufficiently large to obfuscate the change in $\nabla \mathcal R(\widetilde \bth)$ over adjacent data sets, and 
	\item the regularization coefficient $\gamma$ ensures strong convexity of the objective $\mathcal R(\bth)$, so that perturbation of the gradient is translated to perturbation of the solution $\widetilde \bth$.
\end{itemize}
The privacy guarantee is formalized by Proposition \ref{prop: ranking MLE privacy}.
\begin{proposition}\label{prop: ranking MLE privacy}
	Suppose conditions (A0), (A1) and (A2) hold. If $\lambda \geq 8\kappa_1/\varepsilon$ and $\gamma \geq 4\kappa_2/\varepsilon$, $\widetilde\bth$ as defined in Algorithm \ref{alg: DP parametric ranking} is $(\varepsilon, 0)$ differentially private.
\end{proposition}
Proposition \ref{prop: ranking MLE privacy} is proved in the supplement \cite{cai2023optimal}.

We have so far not considered the convergence of $\widetilde \bth$ to the truth $\bth^*$, and in particular choosing large values of $\lambda$ and $\gamma$ for differential privacy compromises the accuracy of the estimator $\widetilde \bth$. 
The optimal choice of $\lambda$ and $\gamma$, which balances privacy and utility, depends on the loss function. 

When $\gamma \asymp \sqrt{np\log n}$ and $F(x) = (1 + e^{-x})^{-1}$, the $\ell_2$-penalized MLE $\hat\bth = \argmin_{\bth \in \R^n} \mathcal L(\bth; y) + \frac{\gamma}{2}\|\bth\|_2^2$ is shown to be a minimax optimal estimator of $\bth^*$ by \cite{chen2019spectral}. By following a similar path as the leave-one-out analysis in \cite{chen2019spectral}, we can then characterize the entry-wise convergence of $\widetilde \bth$ in terms of the noise scale $\lambda$. As the parametric model $\rho_{ij} = F(\theta^*_i - \theta^*_j)$ is invariant to translations of $\bth^*$, we assume without the loss of generality that $\bth^*$ is centered: $\bm 1^\top \bth^* = 0$.

\begin{proposition}\label{prop: ranking MLE accuracy}
	If $\gamma = c_0\sqrt{np\log n}$ for some absolute constant $c_0$, $p \geq c_1\lambda \log n/n$ for some sufficiently large constant $c_1 > 0$, and $c_2 < \lambda < c_2\sqrt{\log n}$ for some sufficiently large constant $c_2 > 0$, it holds with probability at least $1 - O(n^{-5})$ that
	\begin{align} \label{eq: ranking MLE accuracy}
		\|\widetilde \bth - \bth^*\|_\infty \lesssim \sqrt{\frac{\log n}{np}} + \frac{\lambda\log n}{np}.
	\end{align}	
\end{proposition}
The proof is given in \cite{cai2023optimal}. Combining the privacy guarantee, Proposition \ref{prop: ranking MLE privacy}, with the rate of convergence, Proposition \ref{prop: ranking MLE accuracy} leads to the rate of convergence of our estimator $\widetilde \bth$.

\begin{Theorem}\label{thm: ranking upper bound}
	If $\gamma = c_0\sqrt{np\log n}$ for some absolute constant $c_0 > 0$, $p \geq c_1\log n/n\varepsilon$ for some absolute constant $c_1 > 0$, $\lambda = 8\kappa_1/\varepsilon$, and $c_2(\log n)^{-1/2} < \varepsilon < 1$ for some absolute constant $c_2 > 0$,  then the estimator $\widetilde \bth$ defined in \eqref{eq: ranking MLE definition} is $(\varepsilon, 0)$ edge differentially private, and it holds with probability at least $1 - O(n^{-5})$ that
	\begin{align}\label{eq: ranking upper bound}
		\|\widetilde \bth - \bth^*\|_\infty \lesssim \sqrt{\frac{\log n}{np}} + \frac{\log n}{np\varepsilon}.
	\end{align}
\end{Theorem}

In Theorem \ref{thm: ranking upper bound}, the assumed conditions ensure Propositions \ref{prop: ranking MLE privacy} and \ref{prop: ranking MLE accuracy} are applicable. The upper bound \eqref{eq: ranking upper bound} follows from  \eqref{eq: ranking MLE accuracy} in Proposition \ref{prop: ranking MLE accuracy} by plugging in $\lambda \asymp 1/\varepsilon$. 
The entry-wise error bound implies that the latent parameters $(\theta^*_i)_{i \in [n]}$ can be ranked correctly as long as the true $k$th and $(k+1)$th ranked items are sufficiently separated in their $\theta$ values for all $k \in [n-1]$, 
\begin{align}\label{eq: parametric ranking separation threshold}
	|\theta^*_{(k)} - \theta^*_{(k+1)}| \gtrsim \sqrt{\frac{\log n}{np}} + \frac{\log n}{np\varepsilon}. 
\end{align}
More formally, if $\widetilde{\mathcal S}_k$ is the index of the top $k$ values of the vector $\tilde \bth$ then we have the following result for the recovery of the true top-$k$ set $\mathcal{S}_k$.
\begin{corollary}\label{thm: parametric-ranking-upper-bound}
	Under conditions of Theorem \ref{thm: ranking upper bound}, if \eqref{eq: parametric ranking separation threshold} holds for a fixed $k$, then
	$$
	\Pr(\widetilde{\mathcal S}_k \neq \mathcal{S}_k) = O(n^{-5}).
	$$
\end{corollary}
In the separation condition \eqref{eq: parametric ranking separation threshold}, the $O\left(\frac{\log n}{np\varepsilon}\right)$ due to the differential privacy constraint can dominate the $O(\sqrt{\frac{\log n}{np}})$ term, which is optimal in the non-private case, when for example $\varepsilon \asymp (\log n)^{-1/2}$ and $p \ll \frac{\log^2 n}{n}$. The potentially severe cost of requiring differential privacy motivates the next section which studies the inevitable cost of privacy for entry-wise estimation of $\bth^*$.


\subsubsection{The Edge DP Lower Bound for Parametric Estimation}
\label{sec: ranking estimation lower bound}

For an arbitrary $(\varepsilon, \delta)$-DP estimator $M(\bm Y)$ of $\bth$, we would like to find a lower bound for the maximum $\ell_\infty$ risk $\sup_{\theta \in \Theta} \E\|M(\bm Y) - \bth\|_\infty$ over the parameter space $\Theta = \{\bth \in \R^n: \|\bth\|_\infty \leq 1\}$, which captures the inevitable cost of differential privacy for estimating $\bth$.

To this end, we consider an entry-wise version of the score attack method \cite{cai2023score}:
\begin{align*}
	\mathcal A^{(k)}(M(\bm Y), Y_{ij}) = \begin{cases}
		0 &(i, j) \not\in \mathcal G \text{ or } i, j \neq k, \\
		(M(\bm Y)_k - \theta_k)(y_{kj} - F_{kj}(\bth))\frac{F'_{kj}(\bth)}{F_{kj}(\bth)(1 - F_{kj}(\bth))} &(i, j) \in \mathcal G \text{ and } i = k, \\
		(M(\bm Y)_k - \theta_k)(y_{ik} - F_{ik}(\bth))\frac{F'_{ik}(\bth)}{F_{ik}(\bth)(1 - F_{ik}(\bth))} &(i, j) \in \mathcal G \text{ and } j = k.
	\end{cases}
\end{align*}
It is an entry-wise version of the score attack in the sense that summing $\mathcal A^{(k)}(M(\bm Y), Y_{ij})$ over $k \in [n]$ is exactly equal to the score attack for lower bounding the $\ell_2$ minimax risk. When the reference to $\bm Y$ and $M$ is clear, we denote $	\mathcal A^{(k)}(M(\bm Y), Y_{ij})$ by $A^{(k)}_{ij}$.

Our plan for lower bounding the $\ell_\infty$ risk consists of upper bounding $\sum_{1 \leq i < j \leq n} \E A^{(k)}_{ij}$ by the $\ell_\infty$ risk and lower bounding the same quantity by a non-negative amount. The results of these steps are condensed in Propositions \ref{prop: ranking attack soundness} and \ref{prop: ranking attack completeness} respectively.

\begin{proposition}\label{prop: ranking attack soundness}
	If $M$ is an $(\varepsilon, \delta)$-DP estimator with $0 < \varepsilon < 1$ and $p > 6\log n/n$, then for sufficiently large $n$, every $\bth \in \Theta$ and every $k \in [n]$, it holds that
	\begin{align}\label{eq: ranking attack soundness}
		\sum_{1 \leq i < j \leq n} \E_{\bm Y|\bth} A^{(k)}_{ij} \leq 4\kappa_1 np\varepsilon \cdot \E_{\bm Y|\bth} |M(\bm Y)_k - \theta_k| + 4\kappa_1 (n-1)\delta+ 2\kappa_1 n^{-1} .
	\end{align} 
\end{proposition}

Proposition \ref{prop: ranking attack soundness} is proved by considering $\widetilde{\bm Y}_{ij}$, an adjacent data set of $\bm Y$ obtained by replacing $Y_{ij}$ with an independent copy. By differential privacy of algorithm $M$, $\E A^{(k)}_{ij}$ should be close to $\E A^{(k)}(M(\widetilde{\bm Y}_{ij}), Y_{ij})$, which is seen to be exactly 0 by the statistical independence of $M(\widetilde{\bm Y}_{ij})$ and $Y_{ij}$. The full details can be found in the supplement.

In the opposing direction, instead of a pointwise lower bound of $\sum_{1 \leq i < j \leq n} \E_{\bm Y|\bth} A^{(k)}_{ij}$ at every $\bth \in \Theta$, we lower bound the sum over a particular prior distribution of $\bth$ over $\Theta$.

\begin{proposition}\label{prop: ranking attack completeness}
	Suppose $M$ is an estimator of $\bth$ such that $\sup_{\bth \in \Theta} \E\|M(\bm Y) - \bth\|_\infty < c$ for a sufficiently small constant $c > 0$. If each coordinate of $\bth$ has density $\pi(t) = \1(|t| < 1)(15/16)(1-t^2)^2$, then for every $k \in [n]$ there is some constant $C > 0$ such that\begin{align}\label{eq: ranking attack completeness}
		\sum_{1 \leq i < j \leq n} \E_\bth \E_{Y|\bth} A^{(k)}_{ij} > C.
	\end{align}
\end{proposition}

We defer the proof of Proposition \ref{prop: ranking attack completeness} to the supplement, and combine Propositions \ref{prop: ranking attack soundness} and \ref{prop: ranking attack completeness} to arrive at a minimax risk lower bound in $\ell_\infty$ norm for estimating $\bth$ with differential privacy.
\begin{Theorem}\label{thm: ranking lower bound}
	If $p > 6\log n/n$, $\varepsilon \gtrsim (\log n)^{-1}$, $0 < \varepsilon < 1$ and $\delta \lesssim n^{-1}$, then 
	\begin{align} \label{eq: ranking lower bound}
		\inf_{M \in M_{\varepsilon, \delta}}\sup_{\bth \in \Theta} \E_{\bm Y|\bth} \|M(\bm Y) - \bth\|_\infty \gtrsim \sqrt{\frac{\log n}{np}} + \frac{1}{np\varepsilon}.
	\end{align}
\end{Theorem}
The first term in the lower bound \eqref{eq: ranking lower bound} is exactly the non-private minimax rate proved in \cite{shah2017simple, chen2019spectral}. The second term is attributable to differential privacy.

The lower bound result given in Theorem \ref{thm: ranking lower bound} suggests that the perturbed MLE $\widetilde\bth$ is essentially optimal except possibly a room of improvement by $O(\log n)$, but there is no implication about differentially private ranking algorithms not based on estimating the latent parameters $\bth$. The next section considers differentially private ranking without relying on the parametric assumptions.

\vspace{-5mm}

\subsection{Nonparametric Estimation under Edge DP}\label{sec: edge dp nonparametric ranking}

If we drop the parametric assumption $\rho_{ij} = F(\theta^*_i - \theta^*_j)$, the estimand of interest shifts from $\bth^*$ to the index set of top-$k$ items $\mathcal S_k$ for $k \in [n-1]$ in terms of the average winning probability, $\tau_i = \frac{1}{n}\sum_{j \in [n]} \rho_{ij}$. In Section \ref{sec: nonparametric ranking upper bound}, we exhibit a differentially private estimator of $\mathcal S_k$ which exactly recovers $\mathcal S_k$ when $\tau_{(k)}$ and $\tau_{(k+1)}$ are sufficiently far apart,
\begin{align*}
	|\tau_{(k)} - \tau_{(k+1)}| \gtrsim \sqrt{\frac{\log n}{np}} + \frac{\log n}{np\varepsilon}. 
\end{align*}
It is not a coincidence that the requisite separation is identical to its parametric counterpart \eqref{eq: parametric ranking separation threshold}. We prove in Section \ref{sec: nonparametric lower bound} that this separation is the exact threshold for differentially private ranking in either parametric or nonparametric case.

Formally, consider the space of pairwise probability matrices
\begin{align*}
	\Theta(k,m,c) = \left\{\bm \rho \in [0,1]^{n \times n}: \bm \rho + \bm \rho^\top = 11^\top, \tau_{(k-m)} - \tau_{(k+m+1)} \geq c\left(\sqrt{\frac{\log n}{np}} + \frac{\log n}{np\varepsilon}\right)\right\},
\end{align*}
for $k \in [n-1]$ and $0 \leq m \leq \min(k-1, n-k-1)$. Let $d_H(\cdot, \cdot)$ denote the Hamming distance between sets, and an estimator $\widehat{\mathcal S_k}$ succeeds at recovering $\mathcal S_k$ within tolerance $m$ if
\begin{align*}
	\sup_{\rho \in \Theta(k, m, c)} \Pro\left(d_H(\widehat{\mathcal S_k}, \mathcal S_k) > 2m \right) =  o(1).
\end{align*}
Exact recovery of $\mathcal S_k$ corresponds to $m = 0$. By adopting a similar framework to that of \cite{shah2017simple}, we can directly compare the requisite threshold for top-$k$ ranking with or without differential privacy.

\subsubsection{The Edge DP Algorithm for Non-Parametric Estimation}
\label{sec: nonparametric ranking upper bound}

\cite{shah2017simple} shows that the Copeland counting algorithm, which simply ranks the $n$ items by their number of wins, exactly recovers the top-$k$ items when the $\tau$ values of the true $k$th and $(k+1)$th items are separated by at least $O\left(\sqrt{\frac{\log n}{np}}\right)$. Algorithm \ref{alg:DP-nonparametric-ranking} considers a differentially private version where the items are ranked by noisy numbers of wins.

\begin{algorithm}
	\caption{Differentially Private Ranking for nonparametric models} 
	\label{alg:DP-nonparametric-ranking}
	\textbf{Input}: Comparison data $(y_{ij})_{(i,j)\in \mathcal{G}}$, comparison graph $\mathcal{G}$, privacy parameter $\varepsilon$.
	\begin{algorithmic}[1]
		\State Set $N_i = \sum_{j \neq i,(i,j)\in \mathcal{G}} \1(Y_{ij} = 1)$ denote the number of comparisons won by item $i$.
		\State Generate 
		$$
		\bm W = (W_1, W_2, \cdots, W_n) \stackrel{\text{i.i.d.}}{\sim }\text{Laplace}\left(\frac{2}{\varepsilon}\right).
		$$
		\State Compute the top-$k$ set 
		\begin{align*}
			\widetilde {\mathcal S}_k = \{i \in [n]: N_i + W_i \text{ is among the top $k$ largest of }\{N_j + W_j\}_{j \in [n]}\}.
		\end{align*}
	\end{algorithmic} 
	\textbf{Output}: $\widetilde {\mathcal S}_k$.
\end{algorithm}

The estimator $\widetilde {\mathcal S}_k$ defined in Algorithm \ref{alg:DP-nonparametric-ranking} is $(\varepsilon, 0)$-DP by the Laplace mechanism \cite{dwork2006calibrating}: the vector $(N_1, N_2, \cdots, N_n)$ has $\ell_1$-sensitivity bounded by 2 over edge adjacent data sets, and the set $\widetilde {\mathcal S}_k$ is differentially private because it post-processes $\{N_j + W_j\}_{j \in [n]}$. 

$\widetilde {\mathcal S}_k$ recovers $\mathcal S_k$ within tolerance $m$ as long as $\tau_{(k-m)}, \tau_{(k+m+1)}$ are sufficiently separated.

\begin{Theorem}\label{thm: ranking nonparametric upper bound}
	For every $k \in [n-1]$ and any sufficiently large constant $C > 0$, 
	\begin{align}\label{eq: ranking nonparametric upper bound}
		\sup_{\bm \rho \in \Theta(k, m, C)} \Pro\left(d_H(\widetilde {\mathcal S}_k, \mathcal S_k) > 2m \right) < O(n^{-5}).
	\end{align}
\end{Theorem}
Specializing the theorem to $m = 0$ leads to the threshold for exact recovery.
\begin{corollary}
	For every $k \in [n-1]$, if the matrix of pairwise probabilities $\bm \rho$ is such that
	\begin{align*}
		|\tau_{(k)} - \tau_{(k+1)}| \geq C \left(\sqrt{\frac{\log n}{np}} + \frac{\log n}{np\varepsilon}\right)
	\end{align*}
	for a sufficiently large constant $C > 0$, we have $\Pro_{\bm \rho}\left(\widetilde {\mathcal S}_k \neq \mathcal S_k \right) < O(n^{-5})$.
\end{corollary}
As a further consequence, if $|\tau_{(k)} - \tau_{(k+1)}| 
\geq C\left( \sqrt{\frac{\log n}{np}} + \frac{\log n}{np\varepsilon}\right)$ for every $k$, then the union bound implies all $n$ items can be correctly ranked with probability at least $1 - O(n^{-4})$. The next section shows this threshold is optimal in the sense that no differentially private algorithm can succeed at recovering $\mathcal S_k$ when $|\tau_{(k)} - \tau_{(k+1)}| 
< c\left( \sqrt{\frac{\log n}{np}} + \frac{\log n}{np\varepsilon}\right)$ for a sufficiently small $c$.

\vspace{-2mm}

\subsubsection{The Edge DP Lower Bound for Non-parametric Estimation}\label{sec: nonparametric lower bound}
To establish the tightness of the threshold $\sqrt{\frac{\log n}{np}} + \frac{\log n}{np\varepsilon}$ for differentially private ranking, we shall prove that the supremum of $\Pro_{\bm \rho}\left(d_H(\widetilde {\mathcal S}_k, \mathcal S_k) > 2m \right)$ over the set of matrices
\begin{align*}
	\Theta(k,m,c) = \left\{\bm \rho \in [0,1]^{n \times n}: \bm \rho + \bm \rho^\top = 11^\top, \tau_{(k-m)} - \tau_{(k+m+1)} \geq c\left(\sqrt{\frac{\log n}{np}} + \frac{\log n}{np\varepsilon}\right)\right\},
\end{align*}
is bounded away from $0$ for sufficiently small $c$. In view of the lower bound, Theorem 2(b), in \cite{shah2017simple} where the supremum is taken over 
\begin{align*}
	\Theta^0(k,m, c) := \left\{\bm \rho \in [0,1]^{n \times n}: \bm \rho + \bm \rho^\top = 11^\top, \tau_{(k-m)} - \tau_{(k+m+1)} \geq c\sqrt{\frac{\log n}{np}}\right\},
\end{align*}
it suffices to show that the supremum of $\Pro_{\bm \rho}\left(d_H(\widetilde {\mathcal S}_k, \mathcal S_k) > 2m \right)$ over the set
\begin{align*}
	\widetilde\Theta(k, m, 2c) = \left\{\bm \rho \in [0,1]^{n \times n}: \bm \rho + \bm \rho^\top = 11^\top, \tau_{(k-m)} - \tau_{(k+m+1)} \geq 2c\frac{\log n}{np\varepsilon}\right\}
\end{align*}
is bounded away from 0, because 
$\Theta(k,m,c) \subseteq \Theta^0(k,m, 2c) \cup \widetilde\Theta(k, m, 2c)$ for $c > 0$.

For proving the lower bound over $\widetilde\Theta(k, m, 2c)$, the differentially private Fano's inequality \cite{barber2014privacy, acharya2021differentially} reduces the argument to choosing a number of different $\bm \rho$'s in $\widetilde\Theta(k, m, 2c)$ such that the distance among the distributions induced by the chosen $\bm \rho$'s is sufficiently small. We defer the details to the supplement \cite{cai2023optimal} and state the lower bound result below.

\begin{Theorem}\label{thm: nonparametric ranking lower bound}
	Suppose the tolerance $m$ is bounded by $2m \leq (1 + \nu_2)^{-1}\min\{n^{1-\nu_1}, k, n-k\}$, $\frac{\log n}{np\varepsilon} < c_0$, and $\delta < c_0\left(m\log n \cdot n^{10m}/\varepsilon\right)^{-1}$ for a sufficiently small constant $c_0$. There is a small constant $c(\nu_1, \nu_2)$ such that every $(\varepsilon, \delta)$-DP estimator $\widehat {\mathcal S}_k$ satisfies
	\begin{align}\label{eq: nonparametric ranking lower bound}
		\sup_{\bm \rho \in \widetilde\Theta(k, m, c)}\Pro_{\bm \rho}\left(d_H(\widehat {\mathcal S}_k, \mathcal S_k) > 2m \right) \geq \frac{1}{10}
	\end{align}
	whenever $c < c(\nu_1, \nu_2)$ and $n$ is sufficiently large. The inequality remains true if $\bm \rho = \left(\rho_{ij}\right)_{i, j \in [n]}$ is additionally restricted to the parametric model $\rho_{ij} = F(\theta^*_i - \theta^*_j)$, as long as $F$ satisfies regularity condition (A0) in Section \ref{sec: ranking upper bound}.
\end{Theorem}
In conjunction with Theorem \ref{thm: ranking nonparametric upper bound}, Theorem \ref{thm: nonparametric ranking lower bound} yields that $\widetilde{\mathcal S}_k$ is an optimal $(\varepsilon, \delta)$-DP estimator. Setting $m = 0$ in Theorem \ref{thm: nonparametric ranking lower bound} gives the lower bound for exactly recovering the top $k$ items $\mathcal S_k$. In the exact recovery case, the threshold for full ranking of $n$ items is when $|\tau_{(k)} - \tau_{(k+1)}| \geq C\left( \sqrt{\frac{\log n}{np}} + \frac{\log n}{np\varepsilon}\right)$ for every $k$. 
\begin{remark}\label{rem:top-k-to-l-infty-LB}
	Because the lower bound continues to hold when restricted to the parametric model, it in fact settles the $O(\log n)$ gap between the parametric upper bound Theorem \ref{thm: ranking upper bound} and the parametric lower bound Theorem \ref{thm: ranking lower bound}. If $\omega_{ij} = F(\theta^*_i - \theta^*_j)$ for some $F$ satisfying regularity conditions (A0) and (A1) in Section \ref{sec: ranking upper bound} and $\bth^* \in \Theta$, we have $|\tau_{(k)} - \tau_{(k+1)}| \asymp |\theta^*_{(k)} - \theta^*_{(k+1)}|$. The existence of an $(\varepsilon, \delta)$-DP estimator with a faster rate of convergence than $\widetilde{\bth}$ would contradict the lower bound above for recovering $\mathcal S_k$. Under the parametric assumptions, the perturbed MLE $\widetilde \bth$ is minimax optimal for estimating the parameters $\bth^*$. 
\end{remark}


\vspace{-8mm}
\section{Optimal Ranking under Individual Differential Privacy}
\label{sec: individual-dp-ranking}

\subsection{Problem Formulation under Individual DP} \label{sec: individual dp problem formulation}

In the previous section, we studied \emph{edge} differential privacy, where each observed comparison between a fixed pair of items was protected.  Here, we extend that framework to a more realistic setting in which each of \(m\) users performs multiple item comparisons.  Specifically, each user \(k\in[m]\) selects \(L\) unordered pairs \(\{i,j\}\subset[n]\) (with \(L\) a fixed constant) and reports the outcome of each comparison.  We continue to assume that all users share the same underlying preference structure: there is a matrix \(\rho=(\rho_{ij})_{1\le i<j\le n}\) with
$
\rho_{ij}+\rho_{ji}=1,\,\rho_{ii}=\tfrac12,
$
so that whenever any user compares items \(i\) and \(j\), item \(i\) is chosen (i.e., ``preferred'') with probability \(\rho_{ij}\) and item \(j\) with probability \(\rho_{ji}\).  By making \(\rho\) depend only on the item indices \(i,j\), we assume all users draw comparisons from the same pairwise preference distribution.
 
Concretely, for each user \(k=1,\dots,m\) and each draw \(l=1,\dots,L\):
\begin{enumerate}
	\item The user \(k\) selects an unordered pair \(\{i,j\}\subset[n]\) uniformly at random.
	\item They record \(Y_{ij}^{(k,l)}=+1\) if \(i\) is preferred, or \(Y_{ij}^{(k,l)}=-1\) if \(j\) is perferred.
\end{enumerate}
All comparisons are independent across users and draws.  The entire dataset is therefore
\begin{equation}\label{eq:full-data-set-idp}
\mathcal{D} \;=\; \bigl\{\,Y^{(k,l)}_{ij}: 1\le k\le m,\ 1\le l\le L,\ 1\le i<j\le n\bigr\},
\end{equation}
where each nonzero \(Y_{ij}^{(k,l)}\) follows
$
\Pr\bigl(Y_{ij}^{(k,l)}=+1\bigr)
\;=\;
\rho_{ij},$ and $
\Pr\bigl(Y_{ij}^{(k,l)}=-1\bigr)
\;=\;
\rho_{ji}
$.

Our objective remains to rank the \(n\) items by their ``average preference score'' $\tau_i$ as defined in Section \ref{sec: edge dp ranking} and, for a given \(k\in[n]\), recover the top $k$ set.

In contrast to edge DP where protecting a single pairwise outcome sufficed, we now consider \emph{individual} differential privacy: protecting all \(L\) comparisons contributed by any one user.

\medskip
\noindent
\textbf{Individual Differential Privacy}.  Let $\mathcal{D}$
be the full dataset of \(mL\) comparison results as defined in~\eqref{eq:full-data-set-idp}.  For any fixed user index \(k_0\), define a neighboring dataset \(\mathcal{D}'\) by replacing
$
\bigl\{Y^{(k_0,1)},\dots,Y^{(k_0,L)}\bigr\}
\;\longmapsto\;
\bigl\{\widetilde Y^{(k_0,1)},\dots,\widetilde Y^{(k_0,L)}\bigr\},
$
where each \(\widetilde Y^{(k_0,l)}\) is drawn independently using the same \(\rho\).  An algorithm \(M\) satisfied \((\varepsilon,0)\) individual DP if for every measurable event \(A\),
\[
\Pr\bigl[M(\mathcal{D})\in A\bigr]
\;\le\;
e^{\varepsilon}\,
\Pr\bigl[M(\mathcal{D}')\in A\bigr] + \delta.
\]
This ensures that changing all \(L\) comparisons of any single user has a limited effect (controlled by \(\varepsilon\)) on the output.

\medskip
\noindent
Throughout this section, we build on the edge DP constructions of Section \ref{sec: edge dp ranking}, adapting them so that the entire bundle of \(L\) comparisons from each user is protected.  Because each user's \(L\) draws can collectively influence counts or likelihoods, our noise scales must be adjusted accordingly (roughly multiplying by \(L\) compared to the edge DP case).  In the next subsection, we show how to privatize the maximum likelihood estimator for parametric scores, and then we turn to a counting-based approach for nonparametric top-\(k\) recovery under this stronger, individual-level privacy constraint.

\vspace{-5mm}
\subsection{Parametric Estimation under Individual DP}
\label{sec: individual parametric ranking}

We assume the same parametric model as in Section \ref{sec: edge dp ranking}: each item \(i\in[n]\) carries a latent score \(\theta^*_i\in\mathbb{R}\), and for any pair \((i,j)\), the probability that \(i\) is preferred over \(j\) is
$
\rho_{ij} \;=\; F\bigl(\theta^*_i - \theta^*_j\bigr),
$
where \(F:\mathbb{R}\to(0,1)\) is a known, strictly increasing function satisfying \(F(x)=1-F(-x)\).  Across all \(m\) users, each user \(k\) draws \(L\) item-pairs (with replacement), and records the outcome \(Y_{ij}^{(k,l)}\in\{+1,-1\}\) with probability $\rho_{ij}$ and $\rho_{ji}$ respectively, whenever they compare \((i,j)\). 
All comparisons are independent across users and draws.

Define $M_{ij} \;=\; \left|\bigl\{\,k\in[m],\,l\in[L]:\text{user }k\text{ compares }(i,j)\bigr\}\right|$ and
\[
\bar Y_{ij} \;=\; \frac{1}{M_{ij}}
\sum_{\substack{k,l:\\(i,j)\text{ compared}}}
\frac{1+Y_{ij}^{(k,l)}}{2}
\;\in[0,1].
\]
 In words, \(M_{ij}\) counts how many times \((i,j)\) was compared over all users, and \(\bar Y_{ij}\) is the empirical fraction of times \(i\) was chosen over \(j\).
Then, up to additive constants independent of \(\theta\), the negative log-likelihood over all \(mL\) observations is
\[
\mathcal{L}(\theta;Y)
\;=\;
\sum_{1\le i<j\le n}\,
M_{ij}\Bigl[
-\bar Y_{ij}\,\log F_{ij}(\theta)\;-\;(1-\bar Y_{ij})\,\log\bigl(1-F_{ij}(\theta)\bigr)
\Bigr],
\]
where \(F_{ij}(\theta)=F(\theta_i-\theta_j)\). 

\vspace{0.3em}
\noindent
\textbf{Individual DP estimator}.  To ensure individual DP, we apply \emph{objective perturbation} to the regularized likelihood.
\begin{equation}
\widetilde \bth = \argmin_{\bth \in \R^n} \mathcal L(\bth; y) + \frac{\gamma}{2}\|\bth\|_2^2 + \bm w^\top \bth, \quad \bm w = (w_1, w_2, \cdots, w_n) \stackrel{\text{i.i.d.}}{\sim }\text{Laplace}(\lambda).
\end{equation}
 Since each user contributes at most \(L\) comparisons per pair, our noise scales must scale with $L$. We have the following individual DP guarantee.
\begin{proposition}\label{prop: individual ranking MLE privacy}
	Suppose conditions (A0), (A1) and (A2) hold. If $\lambda \geq 8L \cdot \kappa_1/\varepsilon$ and $\gamma \geq 8L\kappa_2/\varepsilon$, $\widetilde\bth$ as defined in Algorithm \ref{alg: DP parametric ranking} is $(\varepsilon, 0)$ differentially private.
\end{proposition}
 With the noise scales set as above and under the usual sampling regime (\(m\) large enough so that \(n\log n/(m\,\varepsilon)=O(1)\)), a similar leave-one-out analysis as the edge DP case leads to the following accuracy result.

\begin{Theorem}
	\label{thm:idp-linf}
	Assume (A0)--(A2) and fix 
	\(\varepsilon\in\bigl(c_{2}(\log n)^{-1/2},\,1\bigr)\),
	with \(L\) a constant.  Choose 
	\(\gamma = 8\,L\,\kappa_{2}/\varepsilon\) 
	and 
	\(\lambda = 8\,L\,\kappa_{1}/\varepsilon\).  
	If \(m\) and \(n\) satisfy \(n\log n/(m\,\varepsilon)=O(1)\), then \(\widetilde{\bth}\) is \((\varepsilon,0)\)-DP, and it holds with probability at least $1-O(n^{-5})$ that
	\[
	\|\widetilde{\bth}-\bth^{*}\|_{\infty}
	\;\;\lesssim\;\;
	\sqrt{\frac{n\,\log n}{m}}
	\;+\;
	\frac{n\,\log n}{m\,\varepsilon}.
	\]
\end{Theorem}
In particular the first term matches the non-private minimax rate \(\sqrt{n\log n/m}\), and the second term quantifies the cost of individual DP.  

\noindent
\textbf{Minimax lower bound}.  
We can further show that no \((\varepsilon,\delta)\)-DP estimator can achieve a smaller worst-case \(\ell_{\infty}\) error.  By the same reduction argument from parametric estimation lower bound to top $k$ ranking lower bound made in Remark 1 in Section \ref{sec: edge dp ranking}, we shall prove a ranking lower bound in the nonparametric case, which then implies the following estimation lower bound.

\begin{Theorem}
	\label{thm: individual-privacy-lower}
	If \(\sqrt{(n\log n)/m} + (n\log n)/(m\,\varepsilon)<c_{0}\) and \(\delta\lesssim n^{-1-\omega}\) for some \(\omega>0\), then any \((\varepsilon,\delta)\)-DP estimator \(\widehat{\theta}\) must satisfy
$$
	\sup_{\|\bth\|_{\infty}\le1}
	\mathbb{E}\bigl\|\widehat{\bth}-\bth\bigr\|_{\infty}
	\;\gtrsim\;
	\sqrt{\frac{n\,\log n}{m}}
	\;+\;
	\frac{n\,\log n}{m\,\varepsilon}.
$$
\end{Theorem}
Hence $\widetilde{\bth}$ of Theorem \ref{thm:idp-linf} is minimax optimal up to constants.

\subsection{Nonparametric Ranking under Individual DP}
\label{sec: individual-nonparametric}

Next, we drop any parametric assumption on \(\rho_{ij}\) and directly recover the top-\(k\) set based on
$\tau_i = \frac{1}{n}\!\sum_{j=1}^n \rho_{ij}$.

\textbf{Count-based estimator}.  Let $N_i$ be the total wins of item \(i\).  Replacing one user's \(L\) comparisons can change each \(N_i\) by at most \(L\), and we therefore add i.i.d.\ \(\mathrm{Laplace}(L/\varepsilon)\) noise to each coordinate.

By the Laplace mechanism \cite{dwork2006calibrating}, the resulting noisy counts
\(\widetilde{N}_i\) in Algorithm \ref{alg:DP-nonparametric-ranking-indiv} are \((\varepsilon,0)\)-DP, and ranking items by \(\{\widetilde{N}_i\}\) preserves individual-level differential privacy.
\begin{algorithm}[!htbp]
	\caption{Individual-DP Top-\(k\) via Noisy Counts}
	\label{alg:DP-nonparametric-ranking-indiv}
	
	\textbf{Input:}  All comparisons \(\{Y^{(k,l)}\}\) and privacy parameter \(\varepsilon>0\).

	\begin{algorithmic}[1]
		\State For each \(i\in[n]\), compute
	$
		N_i \;=\; \sum_{j\neq i}\sum_{k=1}^m\sum_{l=1}^L \1\bigl(Y_{ij}^{(k,l)}=+1\bigr).
$
		\State Draw \(W_1,\dots,W_n\overset{\mathrm{i.i.d.}}{\sim}\mathrm{Laplace}\bigl(L/\varepsilon\bigr).\)
		\State Form \(\widetilde{N}_i = N_i + W_i\) for \(i=1,\dots,n\).
		\State Output \(\widetilde{\mathcal{S}}_k =\) the indices of the top \(k\) values among \(\{\widetilde{N}_i\}\).
	\end{algorithmic}
	
	\textbf{Output:}  \(\widetilde{\mathcal{S}}_k\).
\end{algorithm}

 Fix a tolerance \(u\in\{0,1,\dots,\min(k-1,n-k-1)\}\).  Define
\[
\Theta(k,u,c)
=\Bigl\{
\rho:\,[0,1]^{n\times n},\;\rho_{ij}+\rho_{ji}=1,\;
\tau_{(k-u)}-\tau_{(k+u+1)}
\;\ge\;c\Bigl(\sqrt{\tfrac{n\log n}{m}}+\tfrac{n\log n}{m\,\varepsilon}\Bigr)
\Bigr\},
\]
where \(\tau_{(1)}\ge\tau_{(2)}\ge\cdots\ge\tau_{(n)}\).  We show that the noisy counts succeed at recovering the top items with overwhelming probability.

\begin{Theorem}
	\label{thm: individual ranking nonparametric upper bound}
	There exists \(C>0\) such that, for any \(k\) and \(u\), if \(\rho\in\Theta(k,u,C)\), we have
	\[
	\Pr\bigl(d_H(\widetilde{\mathcal{S}}_k,\mathcal{S}_k)>2u\bigr)
	\;=\;O(n^{-5}).
	\]
	In particular, when \(u=0\), exact recovery holds with probability \(1-O(n^{-5})\) provided that
	\[
	\tau_{(k)}-\tau_{(k+1)}
	\;\ge\;
	C\Bigl(\sqrt{\tfrac{n\log n}{m}}+\tfrac{n\log n}{m\,\varepsilon}\Bigr).
	\]
\end{Theorem}

As a further consequence, if $|\tau_{(k)} - \tau_{(k+1)}| 
\geq C\left(\sqrt{\frac{n\log n}{m}} + \frac{n\log n}{m\varepsilon}\right)$ for every $k$, then the union bound implies all $n$ items can be correctly ranked with probability at least $1 - O(n^{-4})$. The next theorem shows this threshold is optimal in the sense that no differentially private algorithm can succeed at recovering $\mathcal S_k$ when $|\tau_{(k)} - \tau_{(k+1)}| 
< c\left(\sqrt{\frac{n\log n}{m}} + \frac{n\log n}{m\varepsilon}\right)$ for a sufficiently small constant $c$.

\medskip
\noindent
\textbf{Nonparametric ranking lower bound}.  By a DP Fano argument analogous to \cite{barber2014privacy,acharya2021differentially}, no \((\varepsilon,\delta)\)-DP procedure can reliably recover \(\mathcal{S}_k\) when the gap \(\tau_{(k)}-\tau_{(k+1)}\) is smaller.

\begin{Theorem} \label{thm: individual nonparametric ranking lower bound}
	Assume \(\sqrt{(n\log n)/m}+(n\log n)/(m\,\varepsilon)<c_0\) and 
	\(\delta< c_0\bigl(u\log n\,n^{10u}/\varepsilon\bigr)^{-1}\).  Then for sufficiently small \(c\) and large \(n\),
	\[
	\inf_{M\,\in\,\mathcal{M}_{\varepsilon,\delta}}\,
	\sup_{\rho\in\Theta(k,u,c)}
	\Pr\bigl(d_H(M(Y),\mathcal{S}_k)>2u\bigr)
	\;\ge\;\tfrac{1}{10}.
	\]
	In particular, exact recovery (\(u=0\)) fails if for some $k$, 
	\(\tau_{(k)}-\tau_{(k+1)}<c\bigl(\sqrt{(n\log n)/m}+(n\log n)/(m\,\varepsilon)\bigr)\).
\end{Theorem}

Together, Theorems \ref{thm: individual ranking nonparametric upper bound} and \ref{thm: individual nonparametric ranking lower bound} show that the requirement 
\(\sqrt{(n\log n)/m}+(n\log n)/(m\,\varepsilon)\)
is both necessary and sufficient for privately recovering the top-\(k\) set under individual DP. 
\vspace{-5mm}
\section{Numerical Experiments}
\label{sec: experiments}

We conduct three sets of experiments to evaluate the numerical performance of our algorithms and to guide their practical application.

In Sections \ref{sec: edge dp simulated data} and \ref{sec: individual dp simulated data}, we use simulated data to study the accuracy of edge DP and individual DP algorithms respectively. The numerical results qualitatively confirm the dependence of estimation and ranking errors on all parameters: the number of items $n$, the number of individuals $m$,  sampling probability $p$. Notably, the results support the most striking difference between the rates of convergence under edge DP and their counterparts in individual DP, namely that the former is a decreasing function of the number of items $n$ but the latter is increasing in $n$.

To complement these theoretical and simulation studies, Section \ref{sec: real data} applies the individual DP ranking algorithms to two real-world datasets: a student preference dataset from \cite{dittrich1998modelling} and an immigration attitude survey from \cite{weber2011novel}. Both datasets naturally align with the individual DP framework, as each involves individuals comparing multiple pairs of items. In such contexts, participants' opinions--on topics such as university preferences or immigration policy--are sensitive and merit privacy protection.

For practical applications of our ranking algorithms, we highlight three findings from the experiment results.

\begin{itemize}
	\item For recovering ranks, the non-parametric, counting based algorithm generally outperform the parametric estimation algorithm, in both edge and individual DP and across all settings of $n$, $p$, $m$, and $\varepsilon$. Both the parametric and nonparametric algorithms significantly outperform the one-shot algorithm in \cite{qiao2021oneshot} in our numerical experiments.
	
	\item The individual DP algorithms are accurate when the number of items is small but the number of comparisons is large. The edge DP algorithms, in contrast, perform well when both the number of items and the number of comparisons are large. With the caveat that the privacy protections offered by edge DP  and individual DP are quite different, this observation may nevertheless be relevant when there is flexibility in choosing which DP framework to apply to the data analysis task.
	
	\item In our experiments the estimation errors at $\varepsilon = 2.5$ or greater tend to be much closer to the non-private estimation errors than to the estimation errors at, say $\varepsilon = 1$, possibly suggesting that choosing $\varepsilon = 2.5$ is a promising choice for balancing the privacy and utility of our algorithms. While the $\varepsilon$ values in different contexts and applications are not inherently comparable \cite{NBERc15017}, a privacy parameter of $\varepsilon = 2.5$ is smaller than or comparable with the $\varepsilon$ values adopted by prominent adopters of DP, such as Apple, LinkedIn, or the US Census Bureau \cite{desfontainesblog20211001}.

\end{itemize}

\vspace{-5mm}
\subsection{Simulated Data Experiments under Edge DP}
\label{sec: edge dp simulated data}

\subsubsection{Setup}

\paragraph {Data Generation}
The pairwise comparison outcomes are sampled from the BTL model, with various values of the number of items $n$ and sampling probability $p$ as specified in the experiments below. We fix the value of $k = n/4$ and generate our $\theta_i$, up to centering, by
$$
e^{\theta_i} \sim 
\begin{cases}
	1 \quad &\mathrm{if} \quad i < k, \\
	\mathrm{Unif}(0.2,0.7)  \quad &\mathrm{if} \quad i \geq k.
\end{cases}	
$$

\paragraph{Evaluation Metrics for Parameter Estimation:}   
For evaluating the parametric estimation algorithms, we consider the $\ell_\infty$ and $\ell_2$ relative errors on logarithmic scale, $\log \left(\frac{\|\hat \theta-\theta^*\|_\infty}{\|\theta^*\|_\infty}\right)$ and $\log \left(\frac{\|\hat \theta-\theta^*\|_2}{\|\theta^*\|_2}\right)$, where $\hat\theta$ is the estimator and $\theta^*$ is the true parameter.

\paragraph{Evaluation Metric for Top-$k$ set recovery:} Under both the parametric and nonparametric models, we evaluate the performance of top-$k$ recovery by one minus the size of overlap between the estimator and the truth, 
$
1 - \frac{|\widehat{\mathcal{S}}_k \cap \mathcal{S}_k|}{k},
$
where $\mathcal{S}_k$ is the true top-$k$ set and $\widehat{\mathcal{S}}_k$ is an estimator.

\subsubsection{Experiments}

{\bf Experiment 1}.  We study the number of items $n$'s effect on the accuracy of our estimator (Figure \ref{fig:experiment-1}). The sampling probability $p$ is fixed at $1$, and we consider four  privacy levels $\varepsilon \in \{0.5, 1, 2.5, \infty\}$. All loss functions decrease as $n$ increases, demonstrating the consistency of our suggested approaches. It is noteworthy that, for top-$k$ set recovery,  the nonparametric Copeland algorithm outperforms the penalized-MLE  in both the private and non-private regimes. 

\begin{center}
	\begin{figure}[!htbp]
		\centering
		\subfloat{{\includegraphics[height=0.15\textheight, width=0.35\textwidth]{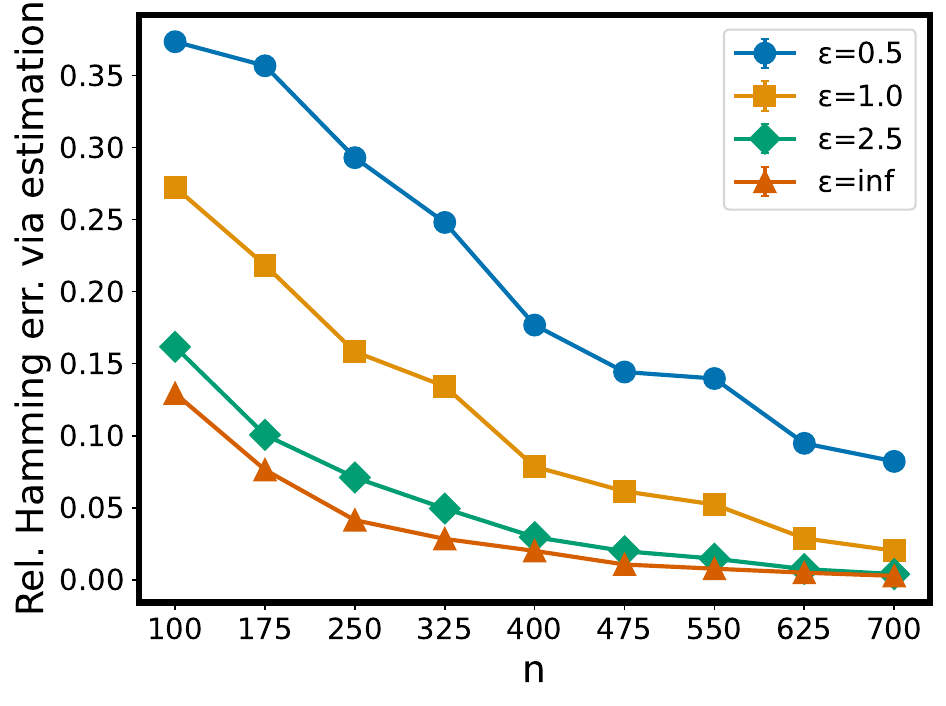}}}%
		\quad
		\subfloat{{\includegraphics[height=0.15\textheight, width=0.35\textwidth]{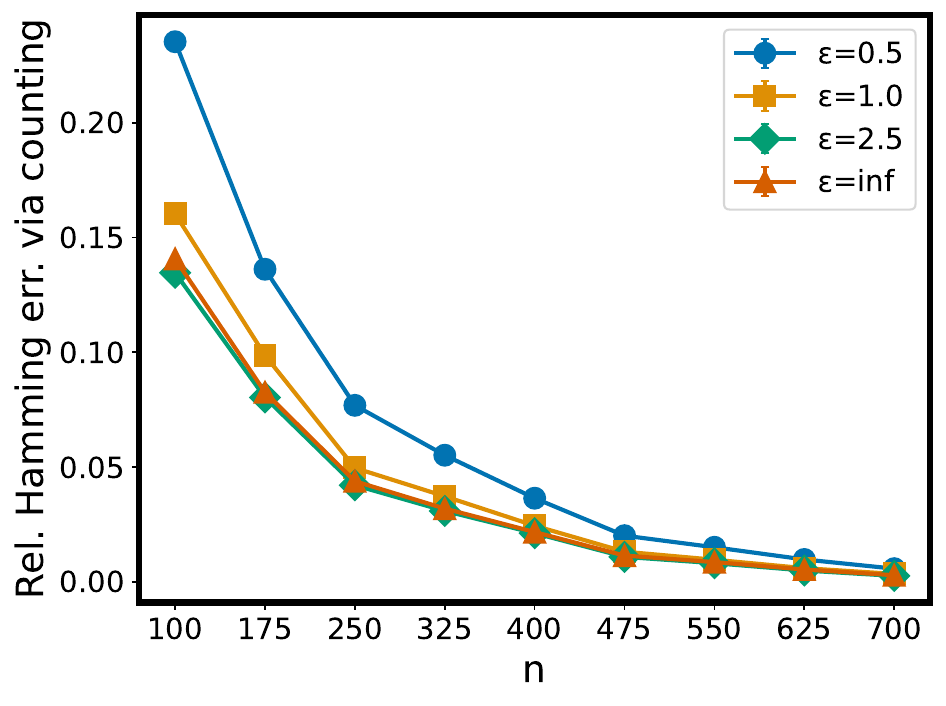}}}%
		
		\subfloat{{\includegraphics[height=0.15\textheight, width=0.35\textwidth]{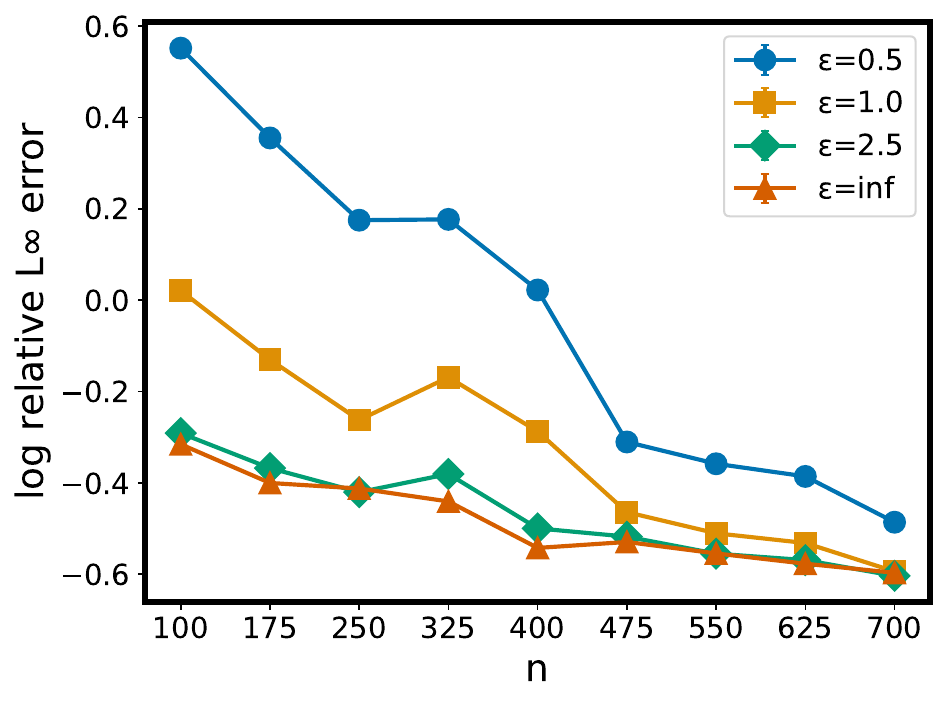}}}%
		\quad
		\subfloat{{\includegraphics[height=0.15\textheight, width=0.35\textwidth]{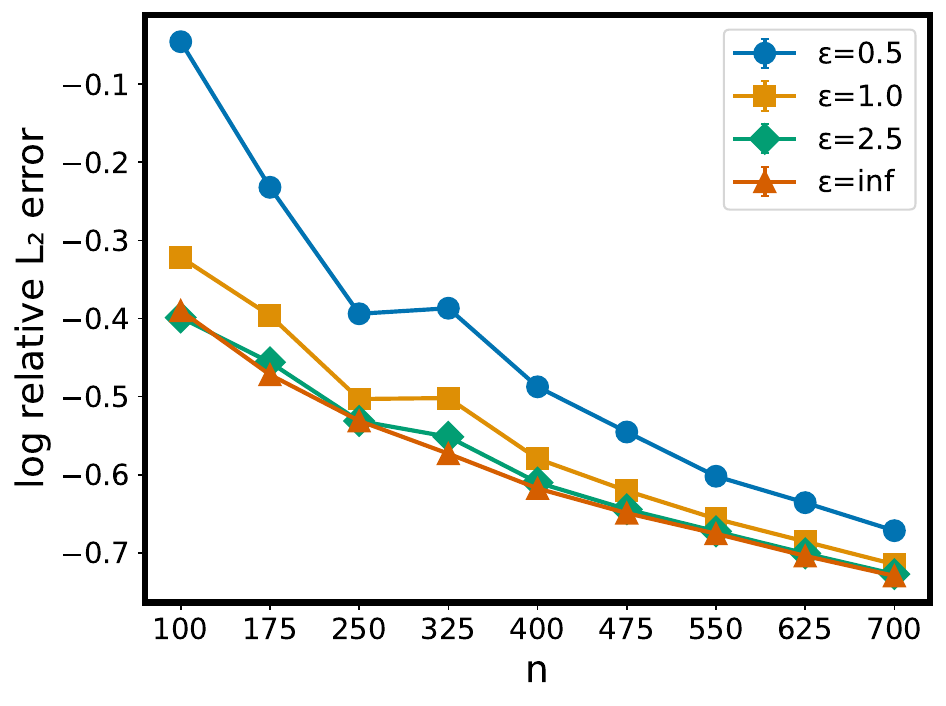}}}%
		\caption{Edge DP estimation errors versus the number of items $n$ at various privacy levels.}%
		\label{fig:experiment-1}%
	\end{figure}
\end{center}
\vspace{-1cm}

{\bf Experiment 2}. We investigate the effect of edge probability $p$ on the accuracy of the  proposed methods (Figure \ref{fig:experiment-2}). The sample size is fixed at $n=300$, and $\varepsilon$ varies across four different levels $\{0.5, 1, 2.5, \infty\}$. As $p$ increases,  we observe more pairwise comparisons, effectively increasing the sample size and leading to better performance.

	\begin{figure}[!htbp]
		\centering
		\subfloat{{\includegraphics[height=0.15\textheight, width=0.34\textwidth]{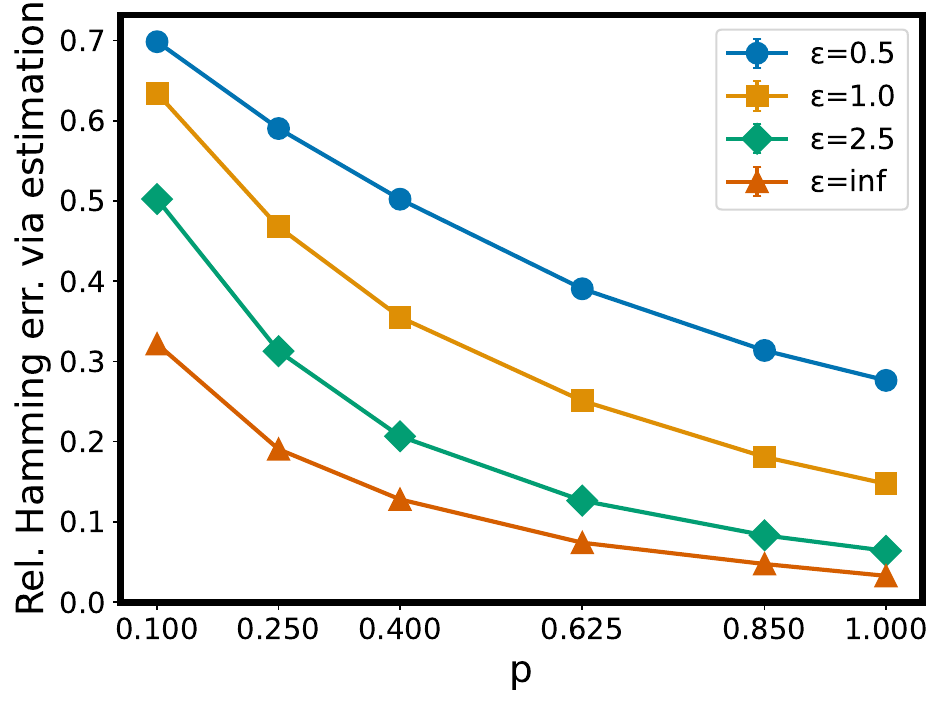}}}%
		\quad
		\subfloat{{\includegraphics[height=0.15\textheight, width=0.34\textwidth]{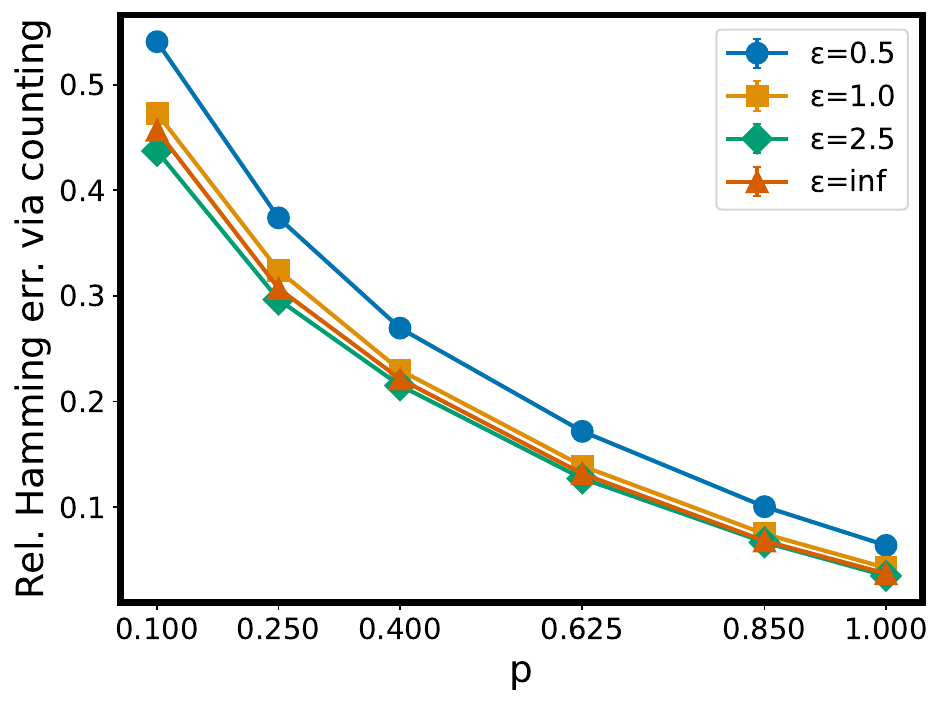}}}%
		
		\subfloat{{\includegraphics[height=0.15\textheight, width=0.34\textwidth]{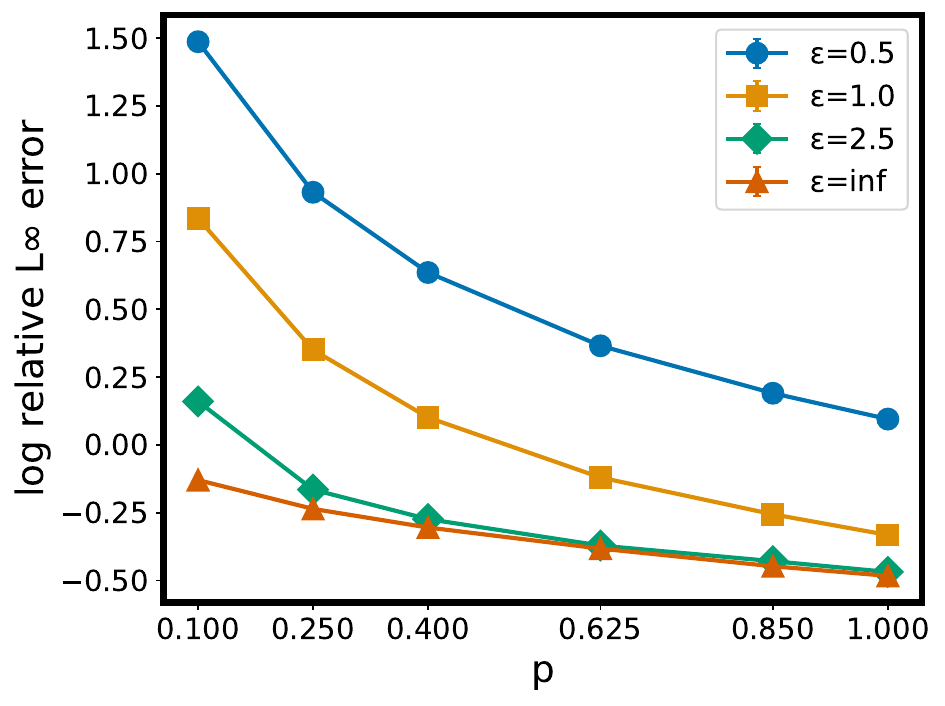}}}%
		\quad
		\subfloat{{\includegraphics[height=0.15\textheight, width=0.34\textwidth]{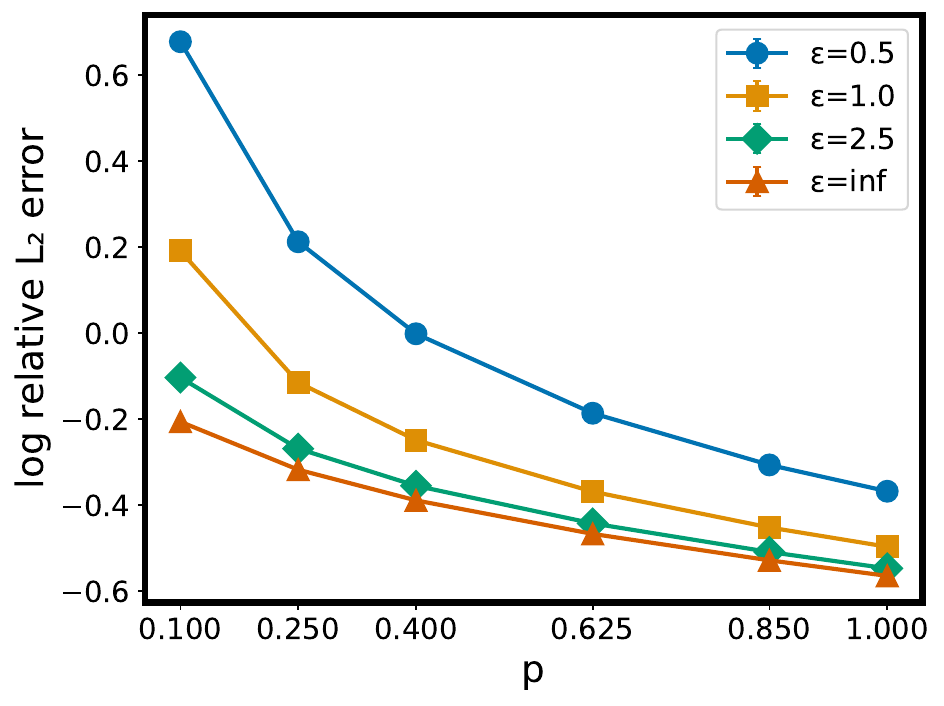}}}%
		\caption{Edge DP estimation errors versus the edge probability $p$ at various privacy levels.}%
		\label{fig:experiment-2}%
	\end{figure}

{\bf Experiment 3}. Here we investigate the effect of privacy parameter $\varepsilon$ on the accuracy of our methods (Figure \ref{fig:experiment-3}). The sample size is fixed at $n=300$, and the sampling probability $p$ varies across four levels $\{0.25,0.5,0.75,1\}$. Increasing $\varepsilon$ reduces the errors.

	\begin{figure}[!htbp]
		\centering
		\subfloat{{\includegraphics[height=0.15\textheight, width=0.34\textwidth]{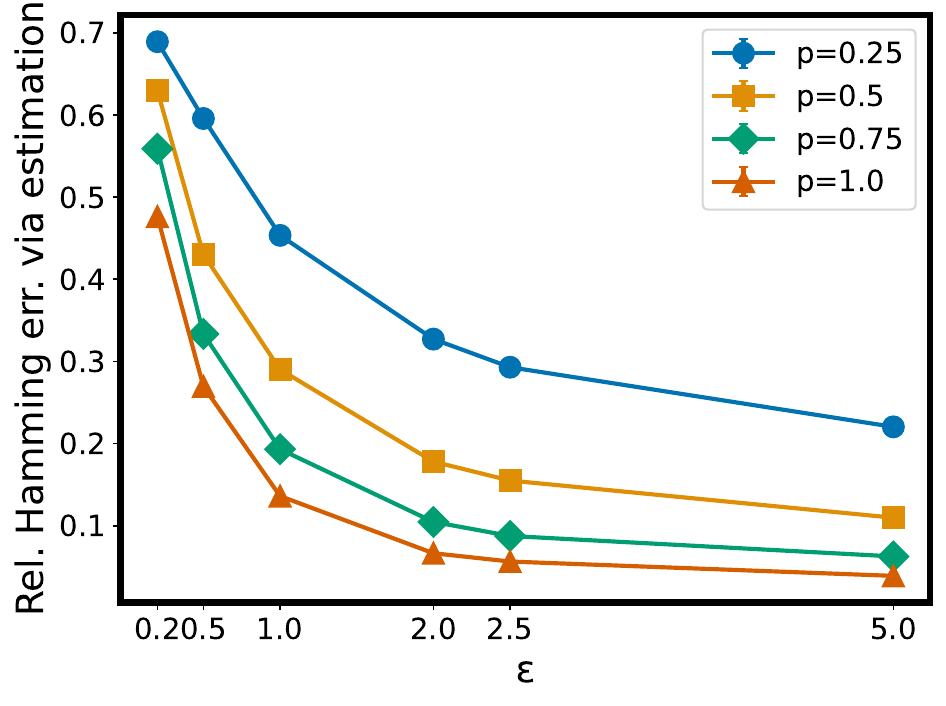}}}%
		\quad
		\subfloat{{\includegraphics[height=0.15\textheight, width=0.34\textwidth]{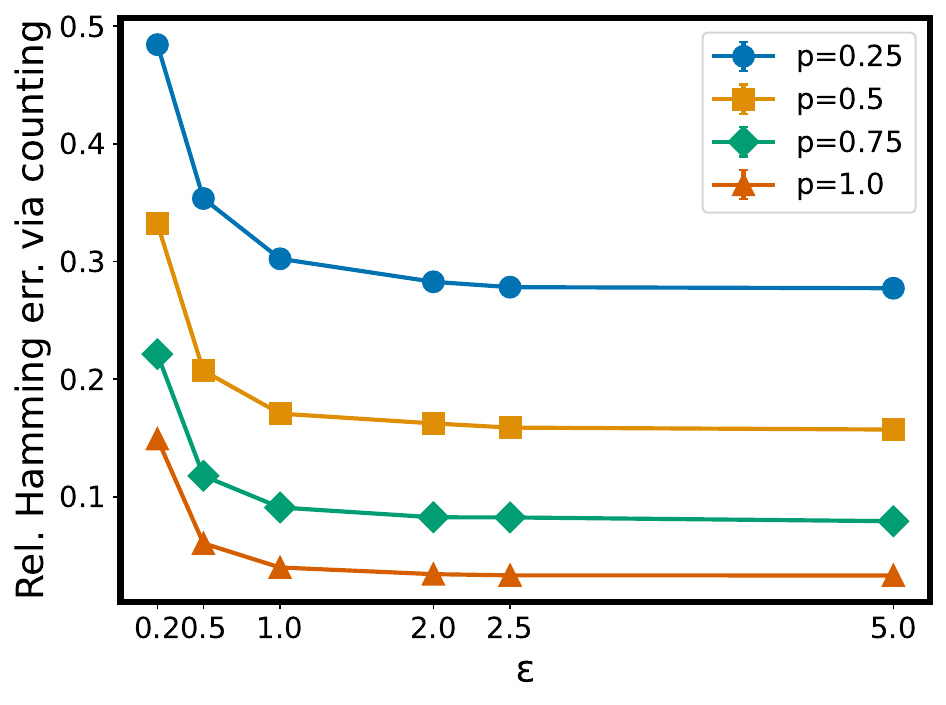}}}%
		
		\subfloat{{\includegraphics[height=0.15\textheight, width=0.34\textwidth]{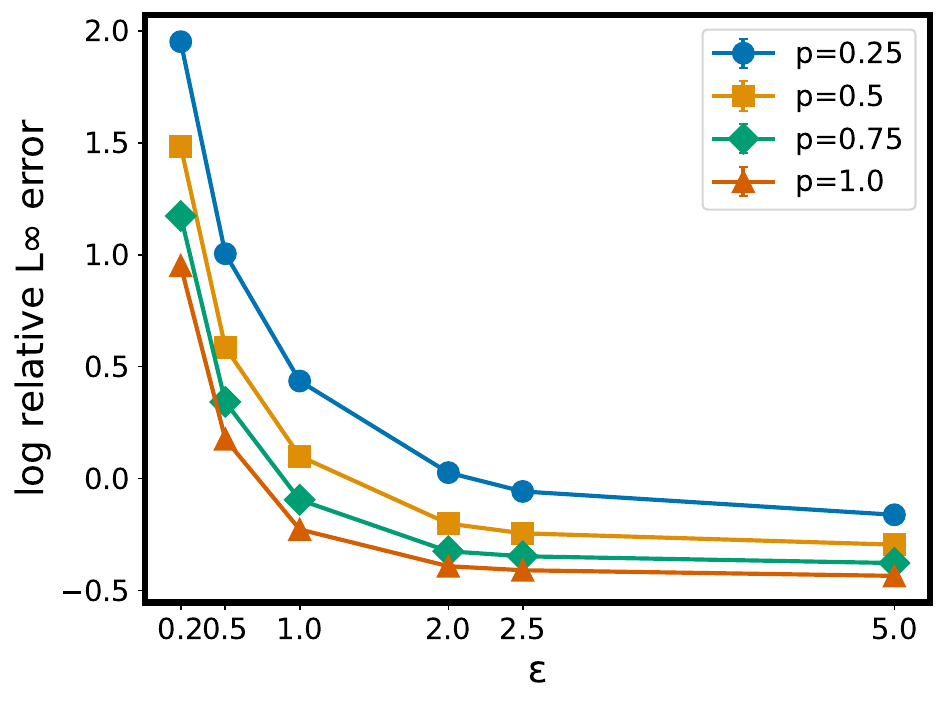}}}%
		\quad
		\subfloat{{\includegraphics[height=0.15\textheight, width=0.34\textwidth]{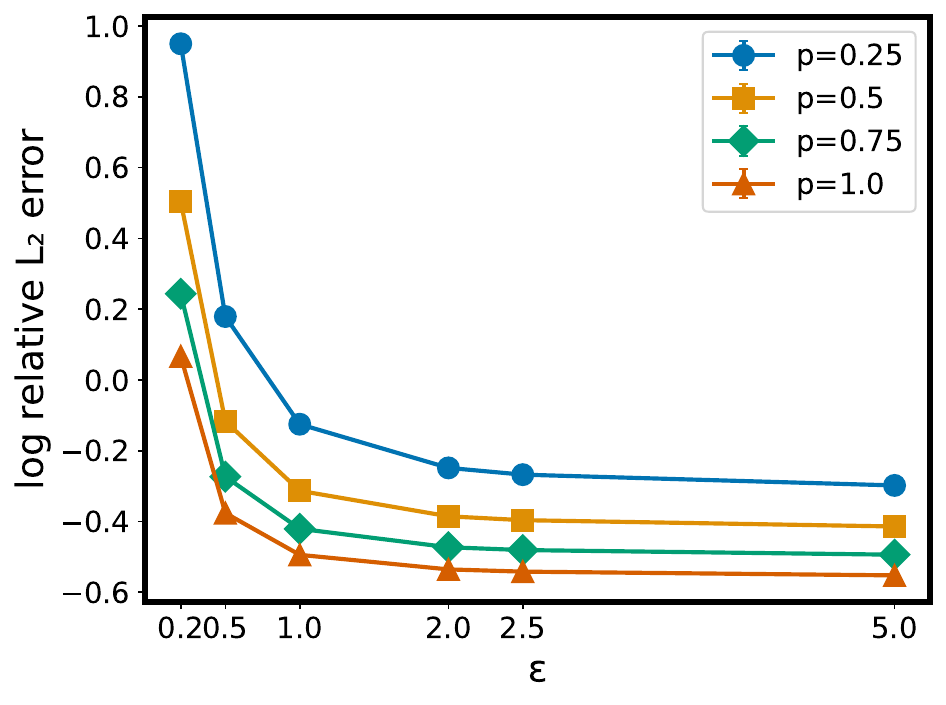}}}%
		\caption{Edge DP errors versus the privacy parameter $\varepsilon$ at various edge probability values.}%
		\label{fig:experiment-3}%
	\end{figure}

{\bf Experiment 4}. We compare the non-parametric and parametric algorithms with the one-shot algorithm proposed by \cite{qiao2021oneshot} at fixed $p=1$ and various $n$ and $\varepsilon$ values (Figure \ref{fig:experiment-4}). In terms of relative Hamming errors, our algorithms significantly outperform the one-shot algorithm.

\begin{figure}[!htbp]
	\centering
	\subfloat{{\includegraphics[height=0.15\textheight, width=0.34\textwidth]{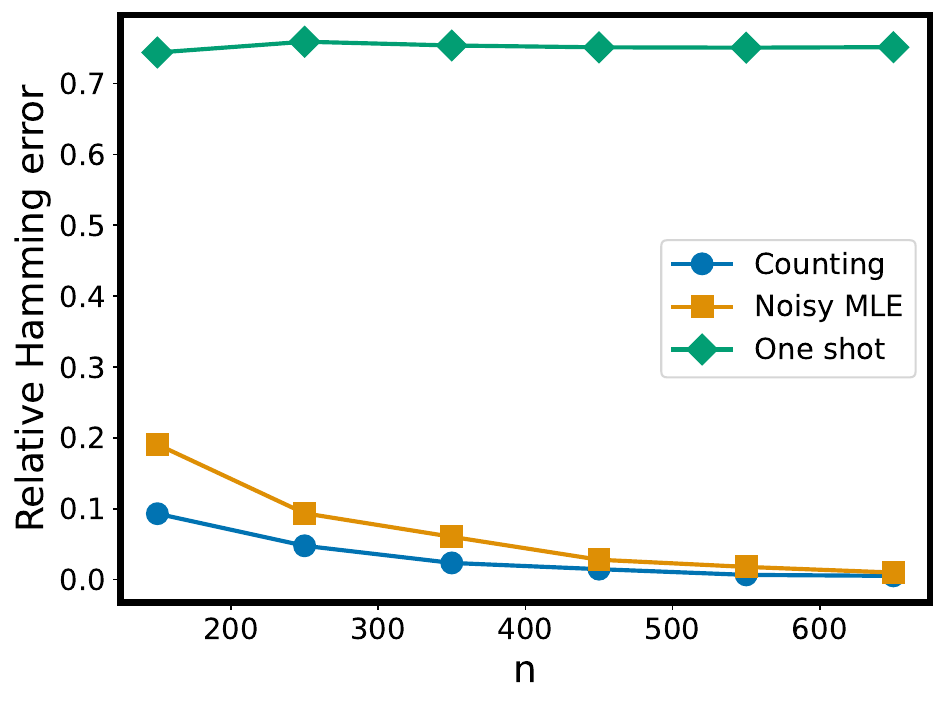}}}%
	\quad
	\subfloat{{\includegraphics[height=0.15\textheight, width=0.34\textwidth]{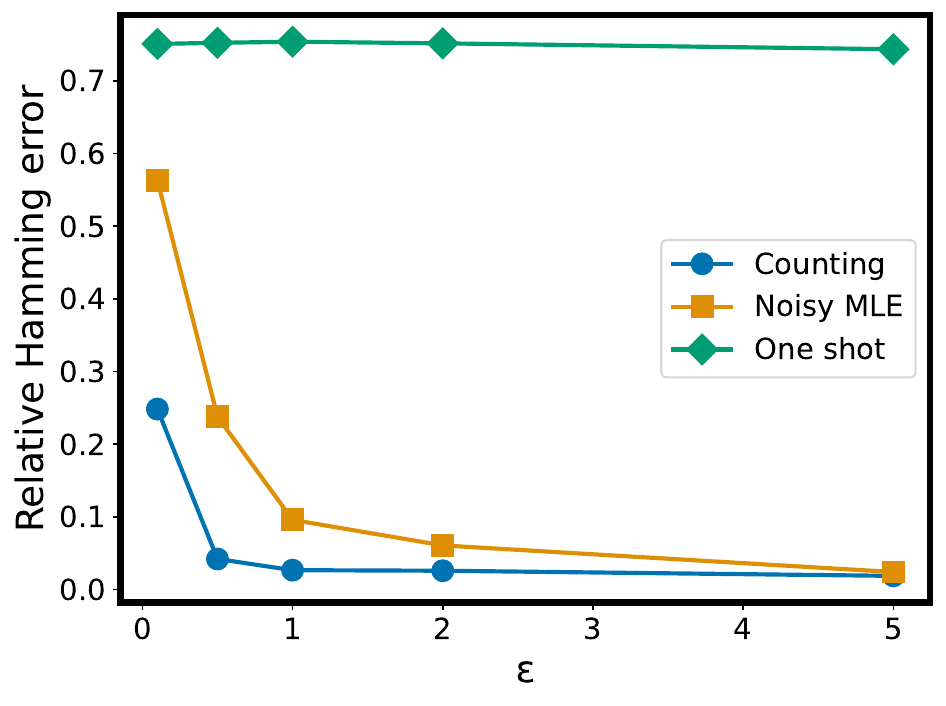}}}%
	
	\caption{Relative Hamming errors of our algorithms and the one-shot algorithm in \cite{qiao2021oneshot} at $p=1$ and various levels of $n$ and $\varepsilon$.}%
	\label{fig:experiment-4}%
\end{figure}
	
 \vspace{-5mm}
 \subsection{Simulated Data Experiments under Individual DP}   \label{sec: individual dp simulated data}

For individual DP simulations, we keep the generation of $\theta$ and evaluation metrics the same as in Section \ref{sec: edge dp simulated data}, and study the accuracy impact of varying the number of items $n$, the number of individuals $m$, as well as the privacy level $\varepsilon$. The number of comparisons contributed by each individual, consistent with the assumption of our theoretical analysis, is chosen to be an absolute constant $L = 5$.

Similar to the edge DP case, the theoretical dependence of estimation errors in $n$, $m$ and $\varepsilon$ is corroborated by the numerical results. In contrast to Experiment 1 in Section \ref{sec: edge dp simulated data}, the estimation errors under individual DP increases with the number of items $n$. The role of $m$ in individual DP is similar to that of edge probability $p$ in edge DP; the behavior of estimation errors versus $\varepsilon$ is also similar to the edge DP case.

{\bf Experiment 5}.  We study the number of items $n$'s effect on the estimation errors (Figure \ref{fig:experiment-5}). The number of individuals is fixed at $m = 1000$,  and we consider four  privacy levels $\varepsilon \in \{0.5, 1, 2.5, \infty\}$. 

\begin{center}
	\begin{figure}[!htbp]
		\centering
		\subfloat{{\includegraphics[height=0.15\textheight, width=0.34\textwidth]{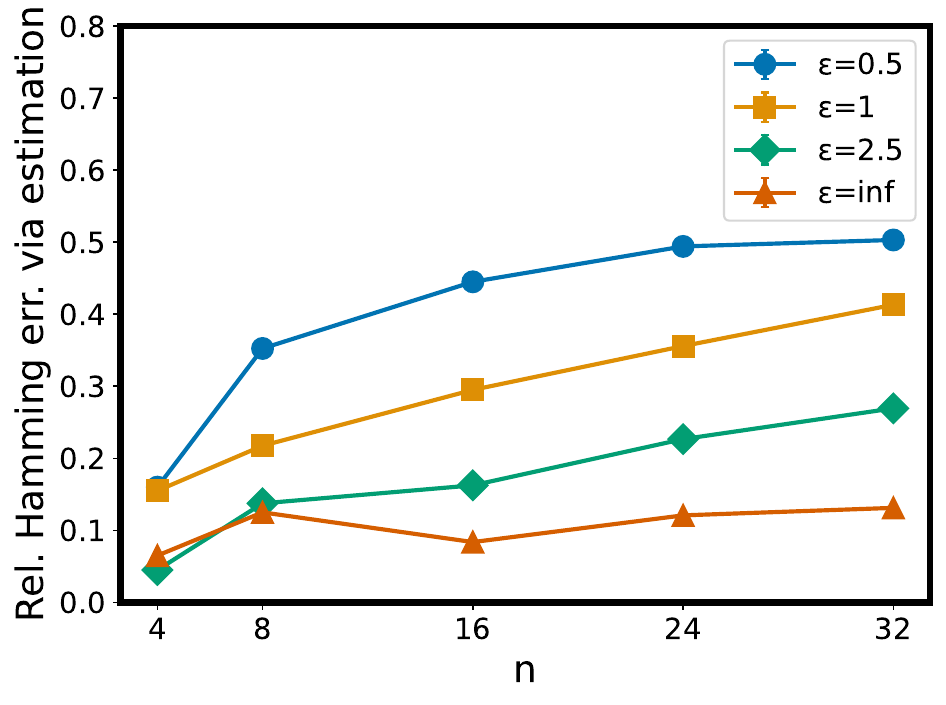}}}%
		\quad
		\subfloat{{\includegraphics[height=0.15\textheight, width=0.34\textwidth]{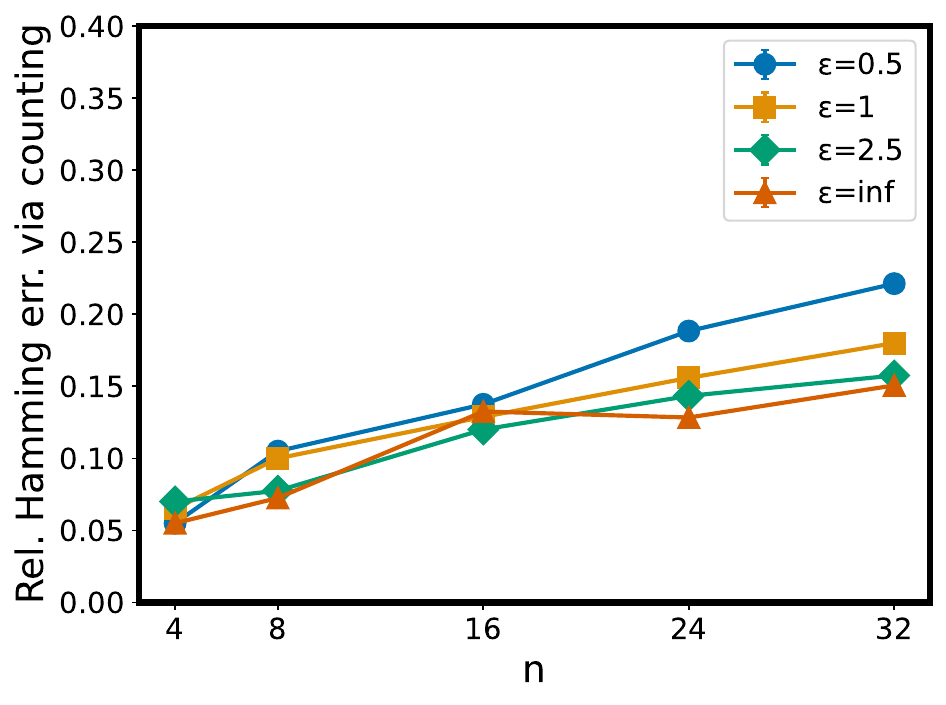}}}%
		
	\subfloat{{\includegraphics[height=0.15\textheight, width=0.34\textwidth]{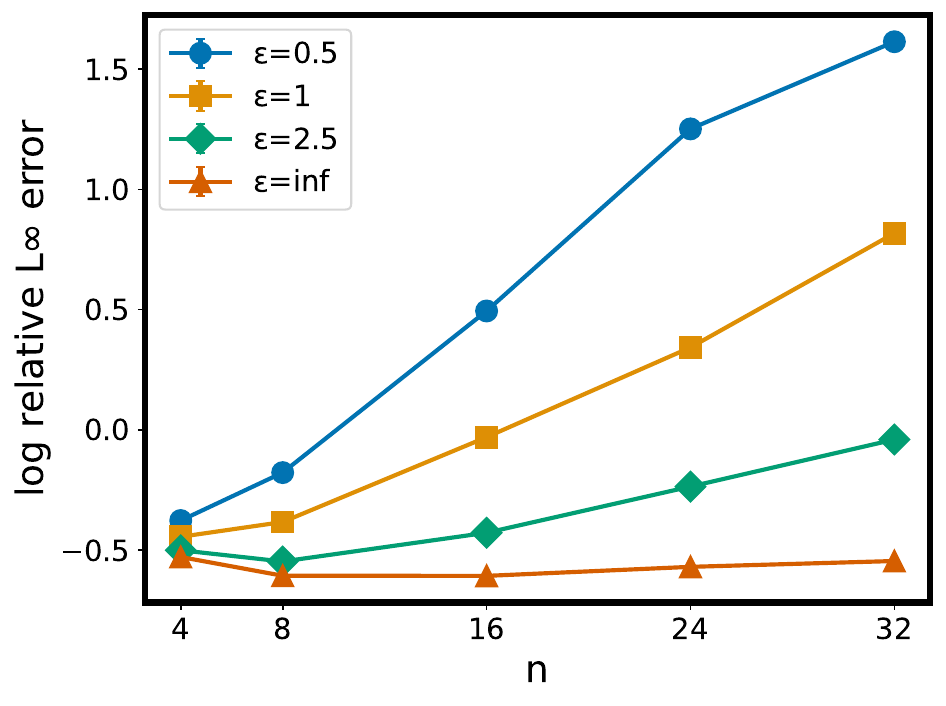}}}%
		\quad
		\subfloat{{\includegraphics[height=0.15\textheight, width=0.34\textwidth]{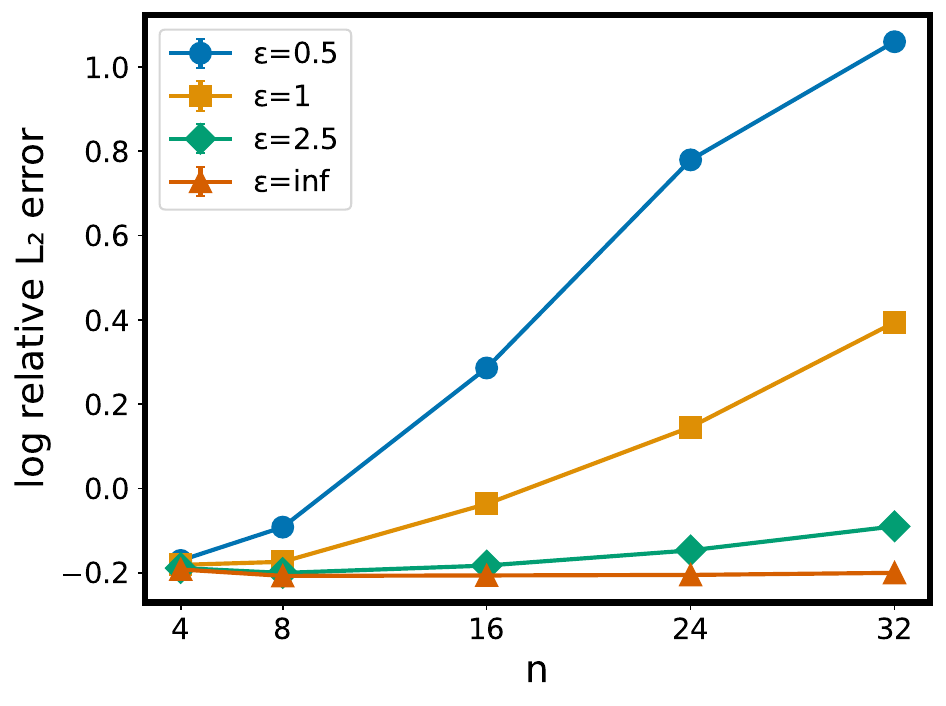}}}%
		\caption{Individual DP errors versus the number of items $n$ at various privacy levels.}%
		\label{fig:experiment-5}%
	\end{figure}
\end{center}

{\bf Experiment 6}. We study the effect of varying the number of individuals $m$ (Figure \ref{fig:experiment-6}). The sample size is fixed at $n=16$, and $\varepsilon$ varies across four different levels $\{0.5, 1, 2.5, \infty\}$. 

\begin{center}
	\begin{figure}[!htbp]
		\centering
		\subfloat{{\includegraphics[height=0.15\textheight, width=0.34\textwidth]{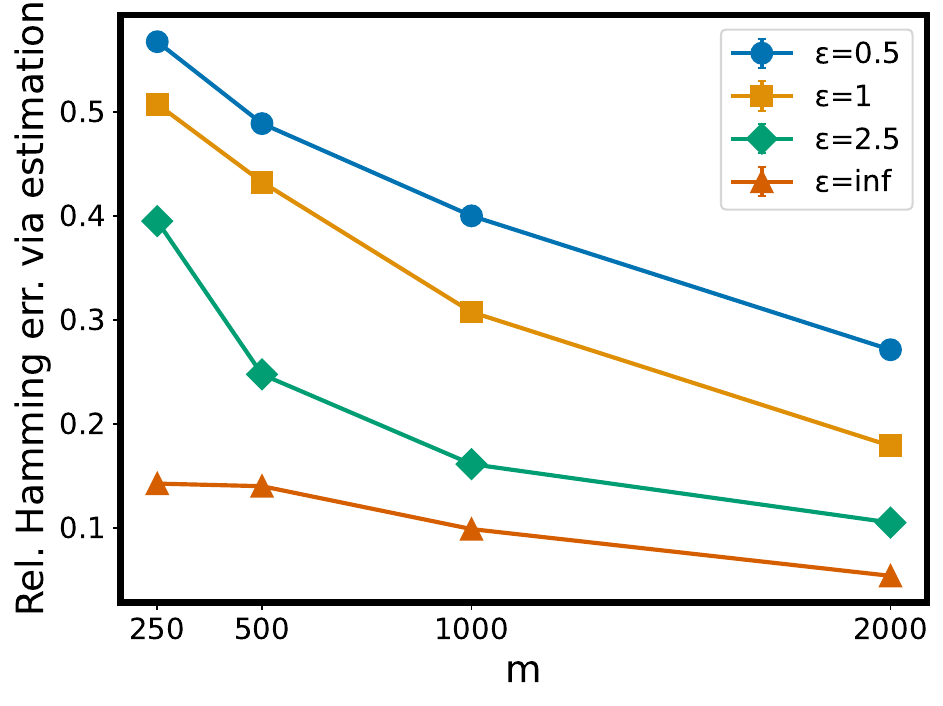}}}%
		\quad
		\subfloat{{\includegraphics[height=0.15\textheight, width=0.34\textwidth]{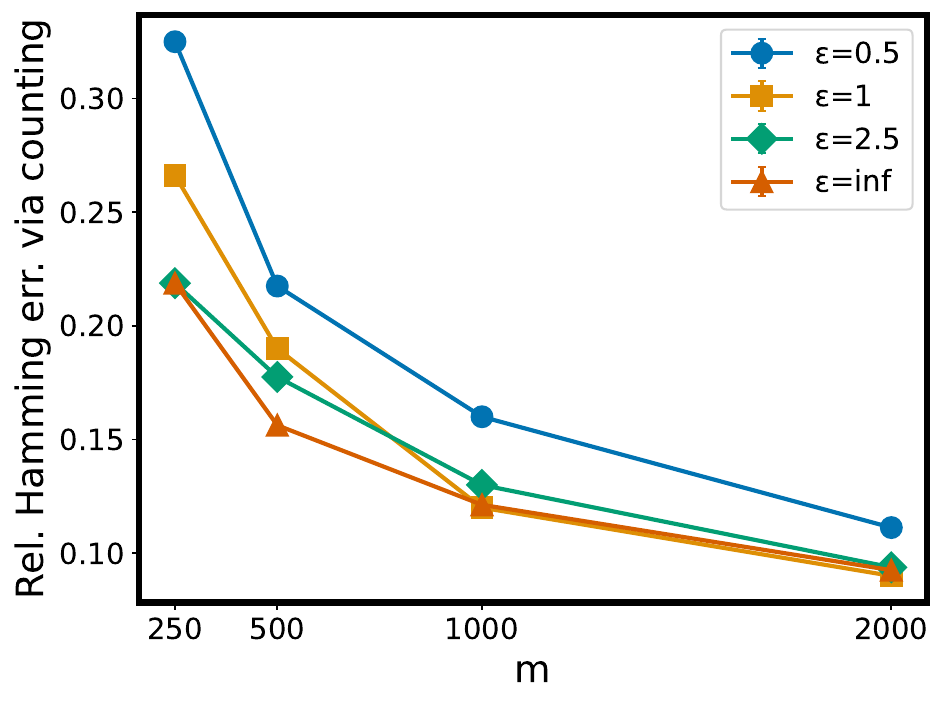}}}%
		
		\subfloat{{\includegraphics[height=0.15\textheight, width=0.34\textwidth]{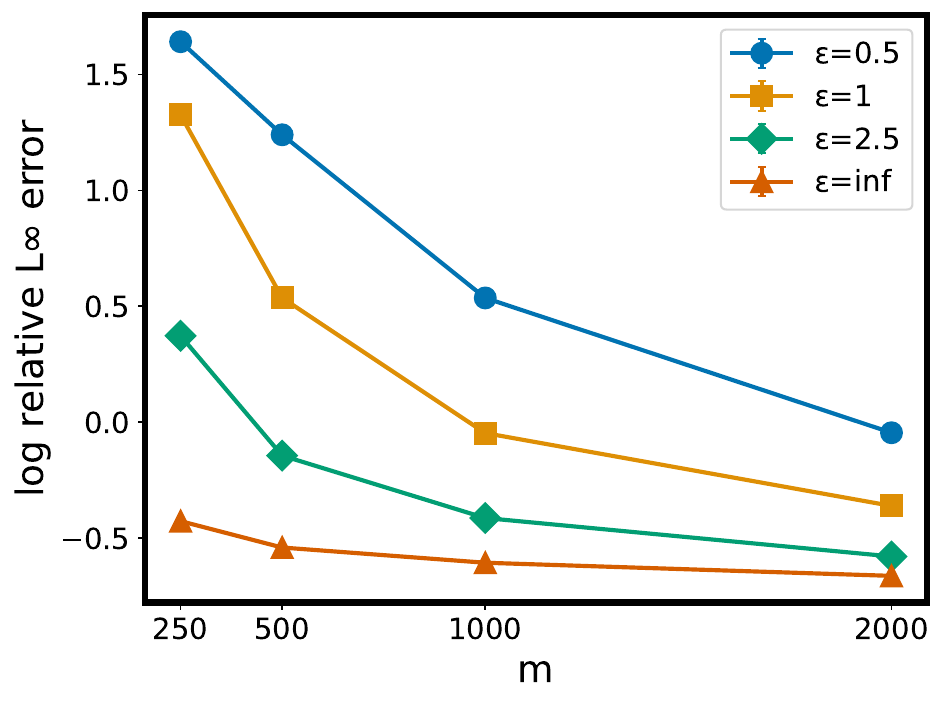}}}%
		\quad
		\subfloat{{\includegraphics[height=0.15\textheight, width=0.34\textwidth]{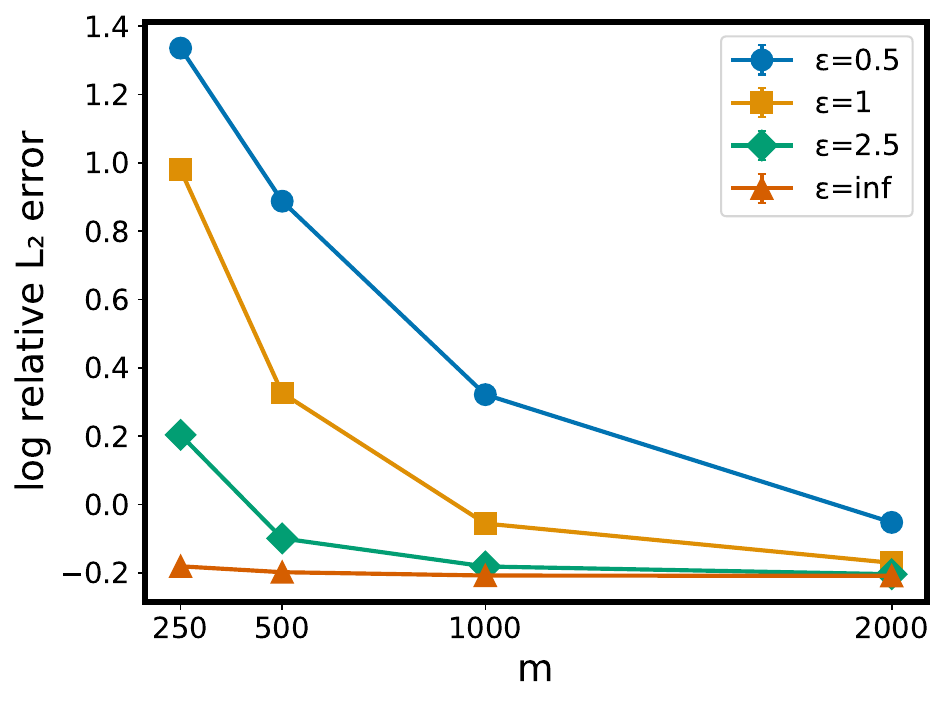}}}%
		\caption{Individual DP errors versus the number of individuals $m$ at various privacy levels.}%
		\label{fig:experiment-6}%
	\end{figure}
\end{center}

{\bf Experiment 7}. We investigate the effect of privacy parameter $\varepsilon$ on the accuracy (Figure \ref{fig:experiment-7}). The sample size is fixed at $n=16$, and the sampling probability $m$ varies across four levels $\{250,500,1000,2000\}$.

\begin{center}
	\begin{figure}[!htbp]
		\centering
		\subfloat{{\includegraphics[height=0.15\textheight, width=0.34\textwidth]{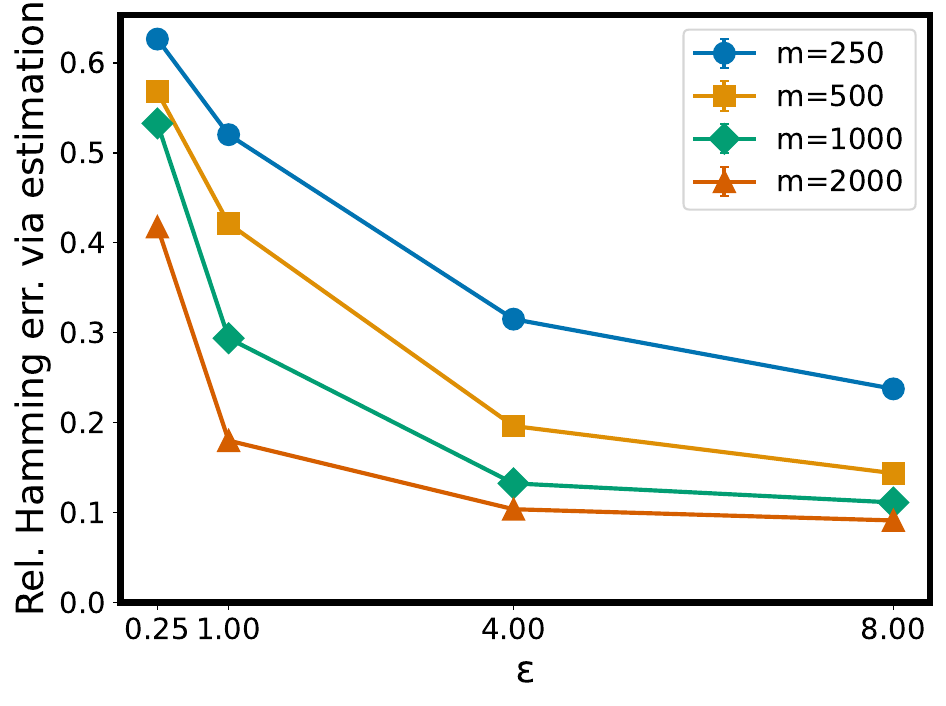}}}%
		\quad
		\subfloat{{\includegraphics[height=0.15\textheight, width=0.34\textwidth]{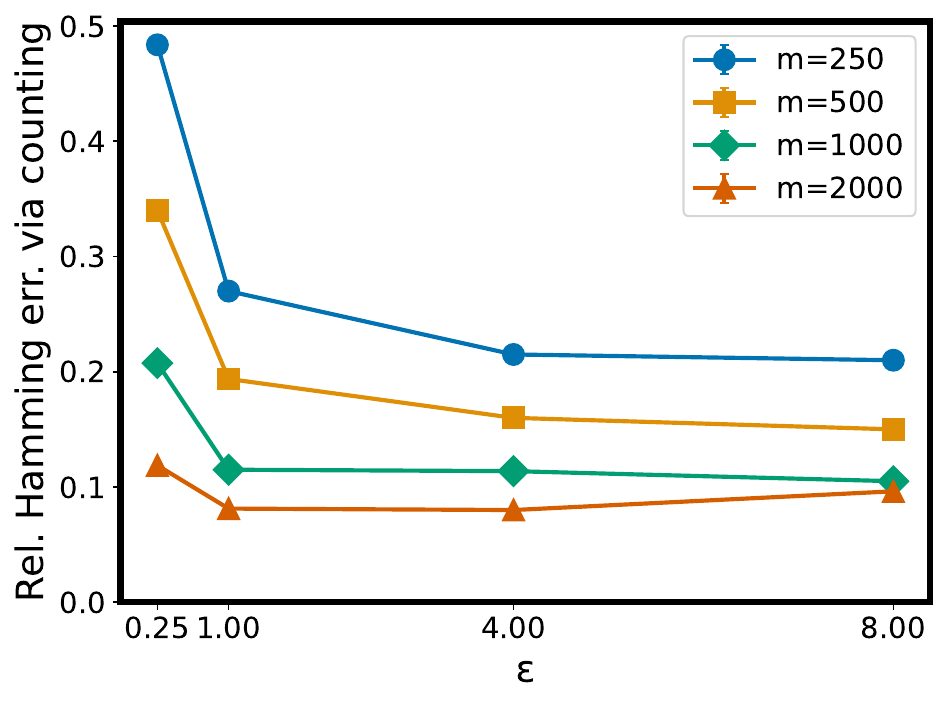}}}%
		
		\subfloat{{\includegraphics[height=0.15\textheight, width=0.34\textwidth]{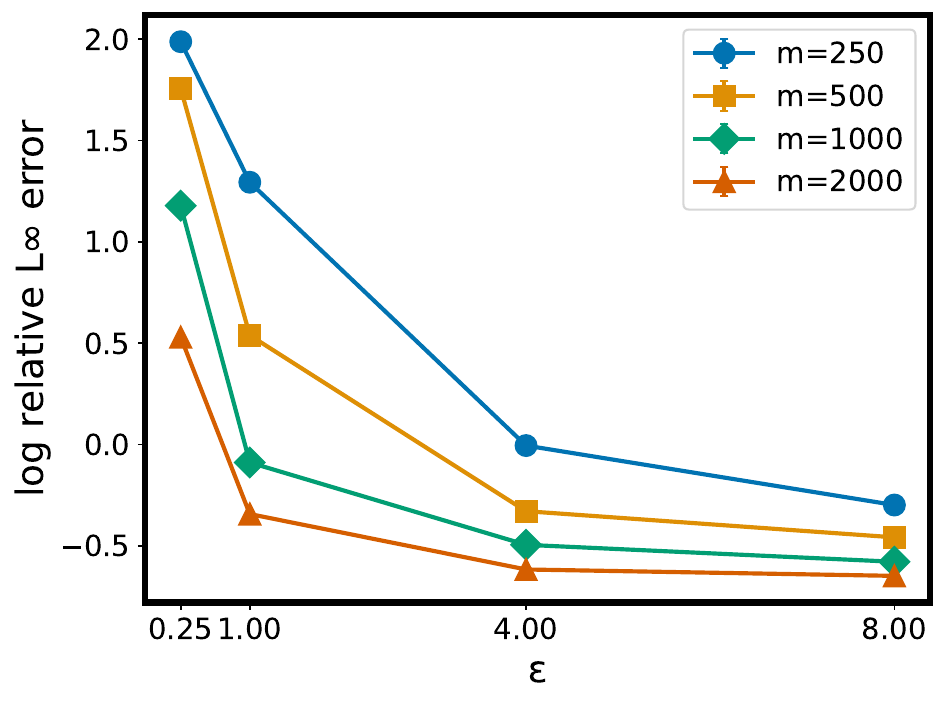}}}%
		\quad
		\subfloat{{\includegraphics[height=0.15\textheight, width=0.34\textwidth]{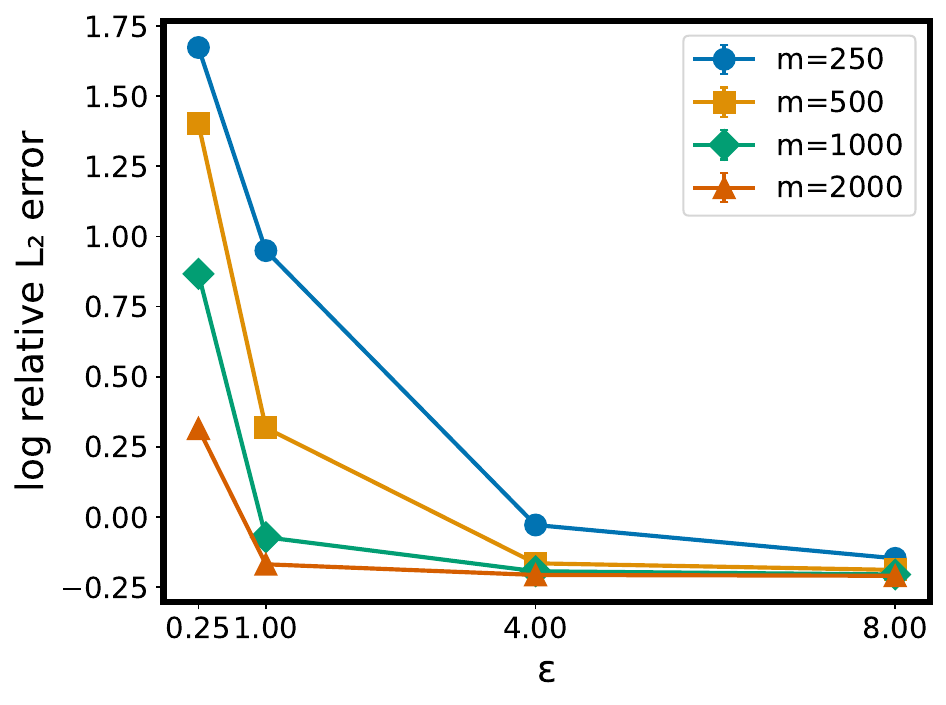}}}%
		\caption{Individual DP errors versus the privacy level $\varepsilon$ at various levels of $m$.}%
		\label{fig:experiment-7}%
	\end{figure}
\end{center}


\vspace{-17mm}
\subsection{Real Data Analysis under Individual DP} 
\label{sec: real data}
 
Now we move beyond simulations and study the effectiveness of our algorithms on real data sets. As real data sets do not have ``true'' rankings, we measure the accuracy of our algorithms by the average difference in ranks produced by the DP algorithm versus those by its non-DP counterpart, which quantifies the loss of accuracy attributable to differential privacy constraints in the practical task of ranking items.

\subsubsection{Data Sets}

\paragraph{University Preferences}

The university preference data set \cite{dittrich1998modelling} is collected in a survey conducted among students in the ``Community of European Management Schools" (CEMS) program by the Vienna University of Economics. The data set consists of observations from 303 students ($m=303$) and records their preference between pairs of European universities for their semester abroad. For each student, a total of 15 ($L=15$) pairwise comparisons between 6 universities ($n=6$) were asked for, and then an overall ranking of all universities was derived using the comparison outcomes. 

\paragraph{Student Attitudes on Immigration}

This dataset is collected in a survey conducted by \cite{weber2011novel} to understand public opinions on immigration. The survey collected responses from 98 students ($m=98$), each agreed to answer at least one paired comparison drawn from a pool of four extreme statements about immigrants ($n=4, L = \binom{4}{2} = 6$). 

\subsubsection{Results}

In both examples, the individual DP algorithms can produce ranks close to the non-DP estimated ranks. As the non-DP, noiseless algorithms are known to be optimal without differential privacy \cite{shah2017simple, chen2019spectral}, it is further implied that the DP ranks are also of high quality. 

Although our theoretical results do not encompass the mean absolute difference in ranks, this metric's behavior versus the privacy parameter $\varepsilon$ is as expected: lower $\varepsilon$, namely stronger privacy level, results in noisier ranks.

\begin{center}
	\begin{figure}[H]
		\centering
		\subfloat{{\includegraphics[height=0.2\textheight, width=0.34\textwidth]{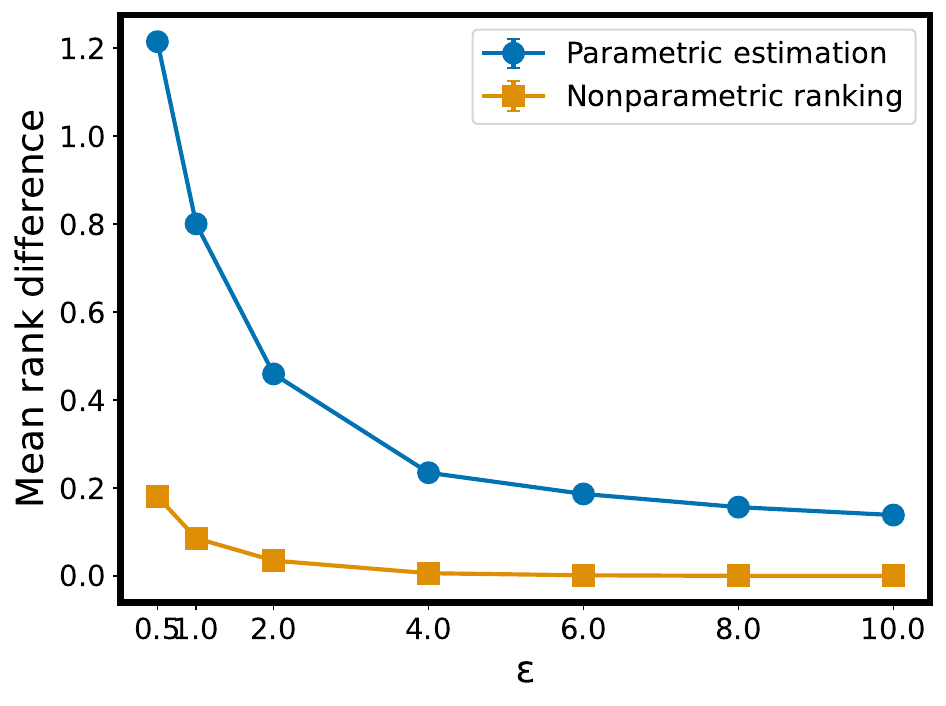}}}%
		\quad
		\subfloat{{\includegraphics[height=0.2\textheight, width=0.34\textwidth]{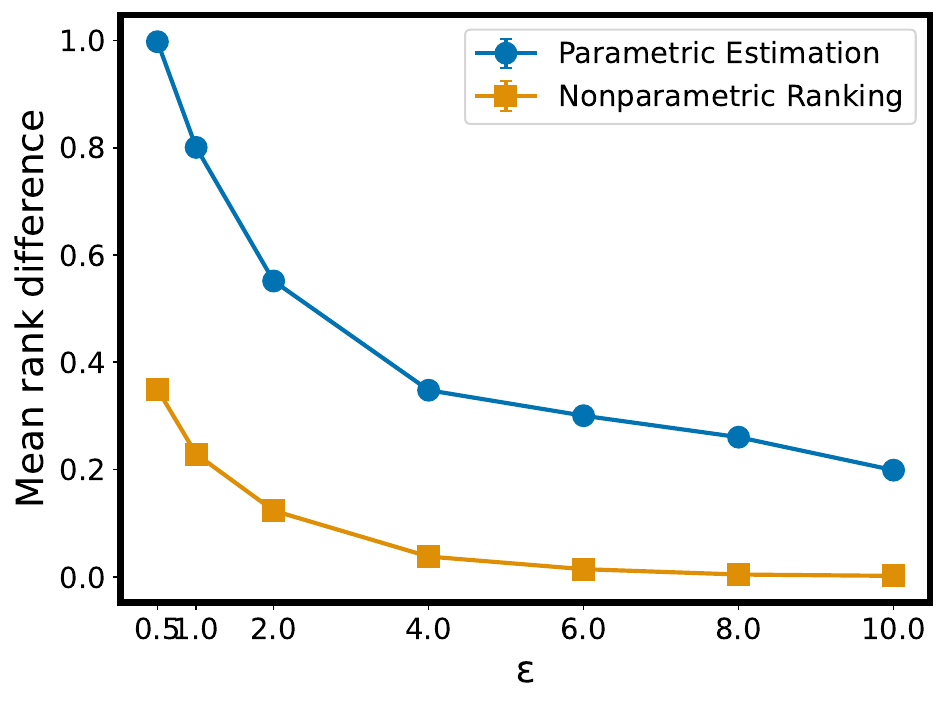}}}%
		
		\caption{Mean rank difference versus $\varepsilon$. Left: CEMS data. Right: immigration data.}%
		\label{fig:experiment-8}%
	\end{figure}
\end{center}
\vspace{-20mm}
\section{Discussion}\label{sec: discussion}

In this paper, we proposed differentially private algorithms for ranking based on pairwise comparisons, analyzed their rates of convergence, and established their optimality among all differentially private ranking procedures. Our theoretical results were supported by numerical experiments on both real and simulated datasets. We conclude by outlining several directions for future research.

{\bf Further theoretical analysis of individual DP}. A comparison between edge DP and individual DP reveals intriguing differences in how privacy constraints impact statistical accuracy. For example, under edge DP, the rate of convergence improves as the number of items $n$ increases, whereas under individual DP, it worsens. This divergence highlights fundamentally different statistical costs associated with the two privacy models. However, our theoretical treatment of individual DP remains less comprehensive than that of edge DP, largely due to the simplifying assumption that the number of comparisons $L$ contributed by each individual is fixed and known. Removing or relaxing this assumption may lead to a deeper understanding of individual DP in more realistic settings.

{\bf The cost of approximate versus pure differential privacy}.
Interestingly, the optimal $(\varepsilon, \delta)$-DP ranking algorithms proposed in this paper also satisfy the stronger $(\varepsilon, 0)$-DP, possibly implying that, for our ranking problem, the cost of ``pure'' differential privacy is not higher than that of 	``approximate'' differential privacy, when measured by the minimax rate of convergence and ignoring the exact constants in the minimax risk. This phenomenon stands in contrast with differentially private (Gaussian) mean estimation in high dimensions \cite{barber2014privacy, steinke2017between, cai2021cost}, where the optimal rate of convergence with $(\varepsilon, \delta)$-DP explicitly depends on $\delta$. It is an interesting theoretical question to understand the conditions under which approximate $(\varepsilon, \delta)$-DP is strictly less costly to statistical accuracy than pure $(\varepsilon, 0)$-DP, in terms of either the minimax rate or better constants at the same minimax rate.

{\bf Lower bound technique for entry-wise parametric estimation}. Under the parametric model, a single dose of noise added to the objective function simplifies the privacy analysis, and avoids potentially higher privacy cost incurred by iterative noise addition required by methods such as noisy gradient descent. The entry-wise error analysis of the perturbed MLE is potentially applicable in other statistical problems where entry-wise or $\ell_\infty$ errors are of primary interest. On the lower bound side, the entry-wise version of score attack in Section \ref{sec: ranking estimation lower bound} results in a $O(\log n)$ gap from the optimal lower bound in Section \ref{sec: nonparametric lower bound}. One would wonder if this method can be strengthened to eliminate such gaps.

\appendix

\section{Omitted Proofs in Section \ref{sec: edge dp ranking}}
\label{sec: omitted proofs in sec: edge dp ranking}

\subsection{Proof of Proposition \ref{prop: ranking MLE privacy}}
\label{sec: proof of prop: ranking MLE privacy}

	Throughout the proof it is useful to reference to the gradient of $\mathcal L(\bth; y)$,
	\begin{align}
		\nabla\mathcal L(\bth; y) &= \sum_{(i, j) \in \mathcal G} \left(-y_{ij}\frac{F'_{ij}(\bth)}{F_{ij}(\bth)} -y_{ji}\frac{-F'_{ij}(\bth)}{1 - F_{ij}(\bth)}\right) (\bm e_i - \bm e_j) \label{eq: likelihood gradient eq 1}\\
		&= \sum_{(i, j) \in \mathcal G} (F_{ij}(\bth) - y_{ij})\frac{F'_{ij}(\bth)}{F_{ij}(\bth)(1 - F_{ij}(\bth))}(\bm e_i - \bm e_j). \label{eq: likelihood gradient eq 2}
	\end{align} 
	For fixed $y$, the distribution of $\widetilde\bth = \widetilde\bth(y)$ is defined by the equation $\nabla \mathcal R(\widetilde\bth; y) + \bm w = 0$. The solution is guaranteed to exist as the objective function $\mathcal R(\bth; y) + \bm w^\top \bth$ is strongly convex in $\bth$ for each $\bm w$.  Since $\bm w$ is a Laplace random vector, the density of $\widetilde\bth$ is given by
	\begin{align*}
		f_{\widetilde\bth}(\bm t) = (2\lambda)^{-n}\exp\left(-\frac{\|\nabla \mathcal R(\bm t; y)\|_1}{\lambda}\right) \left|\det \left(\frac{\partial \nabla \mathcal R(\bm t; y)}{\partial\bm t}\right)\right|^{-1}. 
	\end{align*}
	Consider a data set $y'$ adjacent to $y$, where the only differing elements are $y'_{kl} \in y'$ and $y_{ij} \in y$. It follows that
	\begin{align*}
		\frac{f_{\widetilde\bth(y)}(\bm t)}{f_{\widetilde\bth(y')}(\bm t)} = \exp\left(\frac{\|\nabla \mathcal R(\bm t; y')\|_1 - \|\nabla \mathcal R(\bm t; y)\|_1 }{\lambda}\right) \left|\frac{\det \left(\frac{\partial \nabla \mathcal R(\bm t; y')}{\partial\bm t}\right)}{\det \left(\frac{\partial \nabla \mathcal R(\bm t; y)}{\partial\bm t}\right)}\right|.
	\end{align*}
	By \eqref{eq: likelihood gradient eq 2} and condition (A1), we have
	\begin{align*}
		&\|\nabla \mathcal R(\bm t; y')\|_1 - \|\nabla \mathcal R(\bm t; y)\|_1 \\
		&\leq \begin{cases}
			\kappa_1 \|(y'_{ij} - y_{ij})(\bm e_i - \bm e_j)\|_1  & \text{if } (k, l) = (i, j), \\
			\kappa_1\|(F_{ij}(\bm t) - y_{ij})(\bm e_i - \bm e_j)\|_1 + \kappa_1\|(F_{kl}(\bm t) - y'_{kl})(\bm e_k - \bm e_l)\|_1& \text{if } (k, l) \neq (i, j).
		\end{cases}
	\end{align*}
	In either case, $\|\nabla \mathcal R(\bm t; y')\|_1 - \|\nabla \mathcal R(\bm t; y)\|_1 \leq 4\kappa_1$, and $\lambda \geq 8\kappa_1/\varepsilon$ ensures 
	\begin{align*}
		\exp\left(\frac{\|\nabla \mathcal R(\bm t; y')\|_1 - \|\nabla \mathcal R(\bm t; y)\|_1 }{\lambda}\right) \leq \exp(\varepsilon/2).
	\end{align*}
	For the determinants, by differentiating \eqref{eq: ranking likelihood equation} twice with respect to $\bth$ we obtain
	\begin{align*}
		&\frac{\partial \nabla \mathcal R(\bm t; y)}{\partial\bm t} \\
		&= \gamma \bm I + \sum_{(a, b) \in \mathcal G} \left(y_{ab}\frac{\partial^2}{\partial x^2}(-\log F(x))\Big|_{x = t_a  - t_b} + y_{ba}\frac{\partial^2}{\partial x^2}(-\log (1-F(x)))\Big|_{x = t_a  - t_b}\right)(\bm e_a - \bm e_b) (\bm e_a - \bm e_b)^\top.
	\end{align*}
	That is, $\frac{\partial \nabla \mathcal R(\bm t; y')}{\partial\bm t}$ is a rank-two perturbation of $\frac{\partial \nabla \mathcal R(\bm t; y)}{\partial\bm t}$: let $c_{ij}$ denote the coefficient of $(\bm e_i - \bm e_j)(\bm e_i - \bm e_j)^\top$ in the Hessian, and $\mathcal G'$ denote the edge set of $y'$, we have
	\begin{align*}
		\frac{\partial \nabla \mathcal R(\bm t; y')}{\partial\bm t} &= \frac{\partial \nabla \mathcal R(\bm t; y)}{\partial\bm t} - c_{ij}(\bm e_i - \bm e_j)(\bm e_i - \bm e_j)^\top + c_{kl}(\bm e_k - \bm e_l)(\bm e_k - \bm e_l)^\top \\
		&= \gamma \bm I + \sum_{(a, b) \in \mathcal G \cap \mathcal G'} c_{ab}(\bm e_a - \bm e_b) (\bm e_a - \bm e_b)^\top + c_{kl}(\bm e_k - \bm e_l)(\bm e_k - \bm e_l)^\top.
	\end{align*}
	Applying twice Cauchy's formula for rank-one perturbation $\det(\bm A + \bm v \bm w^\top) = \det(\bm A) (1 + \bm w^\top \bm A^{-1} \bm v)$ yields
	\begin{align*}
		&\frac{\det \left(\frac{\partial \nabla \mathcal R(\bm t; y')}{\partial\bm t}\right)}{\det \left(\frac{\partial \nabla \mathcal R(\bm t; y)}{\partial\bm t}\right)} \\
		&= (1 - c_{ij}(\bm e_i - \bm e_j)^\top \bm \left(\frac{\partial \nabla \mathcal R(\bm t; y)}{\partial\bm t}\right)^{-1}(\bm e_i - \bm e_j))(1 + c_{kl}(\bm e_k - \bm e_l)^\top(\gamma\bm I + \bm 
		\Sigma_{\mathcal G \cap \mathcal G'})^{-1}(\bm e_k - \bm e_l)),
	\end{align*}
	where $\bm 
	\Sigma_{\mathcal G \cap \mathcal G'}$ denotes $\sum_{(a, b) \in \mathcal G \cap \mathcal G'} c_{ab}(\bm e_a - \bm e_b) (\bm e_a - \bm e_b)^\top$. With $0 \leq c_{ij}, c_{kl} \leq \kappa_2$ by condition (A2), $\lambda_{\min}(\gamma\bm I + \bm 
	\Sigma_{\mathcal G \cap \mathcal G'}) \geq \gamma$ and $\gamma > 4\kappa_2/\varepsilon$ by assumption, the ratio of determinants can be upper bounded by
	\begin{align*}
		\left|\frac{\det \left(\frac{\partial \nabla \mathcal R(\bm t; y')}{\partial\bm t}\right)}{\det \left(\frac{\partial \nabla \mathcal R(\bm t; y)}{\partial\bm t}\right)}\right| \leq 1 +  \frac{2\kappa_2}{\lambda_{\min}(\gamma\bm I + \bm 
			\Sigma_{\mathcal G \cap \mathcal G'})} \leq 1 + \frac{2\kappa_2}{\gamma} \leq e^{\varepsilon/2}.
	\end{align*}
	To conclude, we have shown that for any adjacent data sets $y, y'$, it holds that
	\begin{align*}
		\frac{f_{\widetilde\bth(y)}(\bm t)}{f_{\widetilde\bth(y')}(\bm t)} = \exp\left(\frac{\|\nabla \mathcal R(\bm t; y')\|_1 - \|\nabla \mathcal R(\bm t; y)\|_1 }{\lambda}\right) \left|\frac{\det \left(\frac{\partial \nabla \mathcal R(\bm t; y')}{\partial\bm t}\right)}{\det \left(\frac{\partial \nabla \mathcal R(\bm t; y)}{\partial\bm t}\right)}\right| \leq e^{\varepsilon}.
	\end{align*}

\subsection{Proof of Proposition \ref{prop: ranking MLE accuracy}} 
\label{sec: proof of ranking MLE upper bound}

\subsubsection{The main proof}\label{sec: main proof of ranking MLE upper bound}
Consider
\begin{align}\label{eq: perturbed objective}
	\widetilde{\mathcal R}(\bth; y) = \mathcal L(\bth; y) + \frac{\gamma}{2}\|\bth\|_2^2 + \bm w^\top\bth = \mathcal R(\bth; y) + \bm w^\top\bth, 
\end{align}
and we abbreviate $\widetilde{\mathcal R}(\bth; y)$, $\mathcal R(\bth; y)$ as $\widetilde{\mathcal R}(\bth)$, ${\mathcal R}(\bth)$ respectively when the reference to $y$ is clear. We shall analyze $\widetilde\bth = \argmin_\bth \widetilde {\mathcal R}(\bth)$ by treating it as the limit of a gradient descent algorithm: for $t = 0, 1, 2,\cdots$, consider
\begin{align}
	\widetilde\bth^{t+1} = \widetilde\bth^t - \eta \nabla\widetilde{\mathcal R}(\widetilde\bth^t) = \widetilde\bth^t - \eta \nabla{\mathcal R}(\widetilde\bth^t) - \eta \bm w. \label{eq: perturbed gradient descent}
\end{align}

Following the lead of \cite{chen2019spectral}, to facilitate the analysis of $\|\hat\bth - \bth\|_\infty$ we consider leave-one-out versions of $\widetilde {\mathcal R}(\bth, y)$: for each $m \in [n]$, define
\begin{align}
	&\mathcal L^{(m)}(\bm v; y)   \notag\\
	&= p\sum_{j \in [n], j \neq m} -F_{mj}(\bth^*) \log F_{mj}(\bm v) - (1 - F_{mj}(\bth^*))\log(1 -  F_{mj}(\bm v))\notag\\
	&\quad + \sum_{(i, j) \in \mathcal G, i, j \neq m} -y_{ij}\log F_{ij}(\bm v) - (1-y_{ij})\log(1 - F_{ij}(\bm v)). \label{eq: leave-one-out likelihood}
\end{align}
That is, $\mathcal L^{(m)}(\bm v; y)$ is obtained from the usual negative log-likelihood function $\mathcal L(\bm v; y)$ by replacing $y_{mj}$'s, all comparisons involving the $m$th item, with their expectations $p F_{mj}(\bth^*)$. As usual, we abbreviate $\mathcal L^{(m)}(\bm v; y)$ as $\mathcal L^{(m)}(\bm v)$ when the reference to data set $y$ is clear.

Next, for each $m$ we define the leave-one-out objective functions $\widetilde{\mathcal R}^{(m)}$ and ${\mathcal R}^{(m)}$:
\begin{align}\label{eq: leave-one-out objective}
	\widetilde{\mathcal R}^{(m)}(\bm v) = \mathcal L^{(m)}(\bm v) + \frac{\gamma}{2}\|\bm v\|_2^2 + \bm w^\top \bm v = {\mathcal R}^{(m)}(\bm v) +\bm  w^\top \bm v.
\end{align}
For each $\widetilde{\mathcal R}^{(m)}(\bm v)$ and $t \in \mathbb N$, let $\bth^{(m), t}$ denote the $t$-th gradient descent step,
\begin{align}\label{eq: leave-one-out gradient descent}
	\bth^{(m), t+1} = \bth^{(m), t} - \eta \nabla\widetilde{\mathcal R}^{(m)}(\bth^{(m), t-1}).
\end{align}

For both gradient descent algorithms \eqref{eq: perturbed gradient descent} and \eqref{eq: leave-one-out gradient descent}, we adopt an idealized initialization, $\widetilde \bth^0 = \bth^{(m), 0} = \bth^*$ for every $m \in [n]$. These gradient descent steps are only tools for theoretical analysis and not intended for practical implementation.

As the gradient and Hessian of the likelihood function \eqref{eq: ranking likelihood equation} will be referenced frequently, we collect them here before starting the analysis.
\begin{align}
	&\nabla \mathcal L(\bm v) = \sum_{(i, j) \in \mathcal G} (F_{ij}(\bm v) - y_{ij})\frac{F'_{ij}(\bm v)}{F_{ij}(\bm v)(1 - F_{ij}(\bm v))}(\bm e_i - \bm e_j). \label{eq: likelihood gradient}\\
	&\nabla^2 \mathcal L(\bm v) =  \sum_{(i, j) \in \mathcal G} \left(-y_{ij}\frac{\partial^2}{\partial x^2}\log F(x)\Big|_{x = v_i  - v_j} - y_{ji}\frac{\partial^2}{\partial x^2}\log (1-F(x))\Big|_{x = v_i  - v_j}\right)(\bm e_i - \bm e_j) (\bm e_i - \bm e_j)^\top. \label{eq: likelihood hessian}
\end{align}
The analysis of $\widetilde \bth$ follows a similar induction as Section 6 of \cite{chen2019spectral}. The inductive hypotheses are
\begin{align}
	&\|\widetilde\bth^t - \bth^*\|_2 < C_1\left(\sqrt{\frac{\log n}{p}} + \frac{\lambda\log n}{\sqrt{n}p}\right). \label{eq: inductive hypothesis 1}\\
	& \max_{1 \leq m \leq n} |\theta^{(m), t}_m - \theta^*_m| < C_2 \left(\sqrt{\frac{\log n}{np}} + \frac{\lambda\log n}{np}\right). \label{eq: inductive hypothesis 3}\\
	&\max_{1 \leq m \leq n}\|\bth^{(m), t} - \widetilde\bth^t\|_2 < C_3\left(\sqrt{\frac{\log n}{np}} + \frac{\lambda\log n}{np}\right). \label{eq: inductive hypothesis 4}\\
	&\|\widetilde\bth^t - \bth^*\|_\infty < C_4\left(\sqrt{\frac{\log n}{np}} + \frac{\lambda\log n}{np}\right). \label{eq: inductive hypothesis 2}
\end{align}
These hypotheses are obviously true if $t = 0$, by the initialization $\widetilde \bth^0 = \bth^{(m), 0} = \bth^*$ for every $m \in [n]$. For any fixed $t$, we shall show that these hypotheses are true with probability at least $1 - O(n^{-8})$ for the $(t+1)$-th step as well. For now assume the induction is true and defer the proof to Section \ref{sec: proof of induction in ranking MLE upper bound}. Now that \eqref{eq: inductive hypothesis 2} holds for every $t$, it suffices to show there exists some $t \in \mathbb N$ such that $\widetilde\bth^t$ and the optimum $\widetilde\bth = \argmin_\bth \widetilde {\mathcal R}(\bth)$ are sufficiently close.

Lemma 1 in \cite{chen2019spectral} shows that the event 
\begin{align*}
	\mathcal A_0 = \left\{\frac{np}{2} \leq d_{\min}(\mathcal G) \leq  d_{\max}(\mathcal G) \leq \frac{3np}{2}\right\}
\end{align*}
occurs with probability at least $1-O(n^{-10})$, as long as $p \gtrsim \log n/m$ as we assumed in Proposition \ref{eq: ranking MLE accuracy}. The objective function \eqref{eq: perturbed objective} is $\gamma$-strongly convex by construction. Under event $\mathcal A_0$ and by Lemma \ref{lm: spectrum of likelihood Hessian} in Section \ref{sec: proof of induction in ranking MLE upper bound}, we further know that the objective function is $(\gamma + 3\kappa_2 np)$-smooth. The step size $\eta = (\gamma + 3\kappa_2 np)^{-1}$ then guarantees that the gradient descent steps defined by \eqref{eq: perturbed gradient descent} obey
\begin{align*}
	\|\widetilde\bth^{t+1} - \widetilde\bth\|_2 \leq \left(1 - \frac{\gamma}{\gamma + 3\kappa_2 np}\right)\|\widetilde\bth^{t} - \widetilde\bth\|_2, \quad t \in \mathbb N,
\end{align*}
implying that 
\begin{align}
	\|\widetilde\bth^{t+1} - \widetilde\bth\|_2 \leq \exp\left(\frac{-t\gamma}{\gamma + 3\kappa_2 np}\right)\|\widetilde\bth^{0} - \hat\bth\|_2 \leq \exp\left(\frac{-c_0t}{(c_0+3\kappa_2)n}\right)\|\widetilde\bth^{0} - \widetilde\bth\|_2.\label{eq: optimization error bound 1}
\end{align}
The second step is true by $p \leq 1$ and $\gamma = c_0\sqrt{np\log n} \leq c_0 n$ and $\gamma \geq c_0$. 

As $\widetilde\bth^0 = \bth^*$, we can bound $\|\widetilde\bth^{0} - \widetilde\bth\|_2$ by observing that
\begin{align*}
	&\gamma\|\bth^*- \widetilde\bth\|^2_2 \leq \widetilde{\mathcal R}(\widetilde\bth) - \widetilde{\mathcal R}(\bth^*) - \nabla\widetilde{\mathcal R}(\bth^*)^\top(\bth^* - \widetilde\bth) \leq \|\nabla\widetilde{\mathcal R}(\bth^*)\|_2\|\bth^*- \widetilde\bth\|_2.\\
	& \|\bth^*- \widetilde\bth\|_2 \leq \gamma^{-1}\|\nabla\widetilde{\mathcal R}(\bth^*)\|_2 \leq \gamma^{-1}\left(\|\nabla\mathcal L(\bth^*)\|_2 + \gamma\|\bth^*\|_2 + \|\bm w\|_2 \right).
\end{align*}
Now assume events 
\begin{align*}
	\mathcal A_1 = \left\{\|\nabla \mathcal L(\bth^*)\|_2 \leq 3\kappa_1\sqrt{n^2p\log n}\right\}, \mathcal A_2 = \left\{|\bm1^\top \bm w| \leq 8\lambda\sqrt{n\log n}, \|\bm w\|_\infty \leq 9\lambda\log n \right\}
\end{align*}
occur, the probability of which is at least $1 - O(n^{-8})$ by Lemmas \ref{lm: gradient norm tail bound} and \ref{lm: Laplace tail bound}. Because $\gamma \asymp \sqrt{np\log n}$ and $\gamma \gtrsim \lambda$, the bound of $ \|\bth- \widetilde\bth\|_2 $ above reduces to
\begin{align*}
	\|\bth^*- \widetilde\bth\|_2 \leq \gamma^{-1}\|\nabla\widetilde{\mathcal R}(\bth^*)\|_2 \lesssim \frac{\sqrt{n^2 p \log n}}{\gamma} + \sqrt{n}  + \frac{\lambda\log n\sqrt{n}}{\gamma} \lesssim \sqrt{n}\log n.
\end{align*}
Plugging the bound into \eqref{eq: optimization error bound 1} and setting $t = \frac{c_0+3\kappa_2}{c_0}n^3$ yields
\begin{align*}
	\|\widetilde\bth^{t+1} - \widetilde\bth\|_2 \lesssim e^{-n^2}\sqrt{n}\log n \lesssim \frac{\log n}{n}.
\end{align*}

By the inductive step \eqref{eq: inductive hypothesis 2} and the union bound over $O(n^3)$ iterations, each of which hold with probability at least $1 - O(n^{-8})$, we conclude that
\begin{align*}
	\|\widetilde\bth-\bth^*\|_\infty \leq \|\widetilde\bth^{t+1} - \widetilde\bth\|_2 + \|\widetilde\bth^{t+1} - \bth^*\|_\infty \lesssim \left(\sqrt{\frac{\log n}{np}} + \frac{\lambda\log n}{np}\right)
\end{align*} 
with probability at least $1 - O(n^{-5})$.

\subsubsection{Proof of the inductive step} \label{sec: proof of induction in ranking MLE upper bound}
In this section we prove the inductive step in Section \ref{sec: main proof of ranking MLE upper bound}. Throughout the proof we assume the following high-probability events simultaneously occur.
\begin{align*}
	&\mathcal A_0 = \left\{\frac{np}{2} \leq d_{\min}(\mathcal G) \leq  d_{\max}(\mathcal G) \leq \frac{3np}{2}\right\},\\
	&\mathcal A_1 = \left\{\|\nabla \mathcal L(\bth^*)\|_2 \leq 3\kappa_1\sqrt{n^2p\log n}\right\}, \\ 
	&\mathcal A_2 = \left\{|\bm1^\top \bm w| \leq 8\lambda\sqrt{n\log n}, \|\bm w\|_\infty \leq 9\lambda\log n \right\}.
\end{align*}
Lemma 1 in \cite{chen2019spectral} shows that events $\mathcal A_0$ occurs with probability at least $1-O(n^{-10})$, as long as $p \gtrsim \log n/m$ as we assumed in Proposition \ref{eq: ranking MLE accuracy}. We bound the probabilities of $\mathcal A_1$ and $\mathcal A_2$ in the lemmas below and defer their proofs to Section \ref{sec: proof of ranking upper bound lemmas}.

\begin{Lemma}\label{lm: gradient norm tail bound}
	If $p \geq c\log n/n$ for a sufficiently large constant $c > 0$, $\Pro(\mathcal A_1) \geq 1 - O(n^{-10})$.
\end{Lemma}

\begin{Lemma}\label{lm: Laplace tail bound}
	If the coordinates $W_1, W_2, \cdots, W_n$ of $\bm w \in \R^n$ are drawn i.i.d. from the (zero-mean) Laplace distribution with scale parameter $\lambda$, then $\Pro(\mathcal A_2) \geq 1 - O(n^{-8})$.
\end{Lemma}

Now, the inductive step will be proved in four parts, from Proposition \ref{prop: inductive step 1} to Proposition \ref{prop: inductive step 2}.

\begin{proposition}\label{prop: inductive step 1}
	Suppose \eqref{eq: inductive hypothesis 1}, \eqref{eq: inductive hypothesis 2} are true for a sufficiently large $C_1 > 0$ and $\eta = (\gamma + 3\kappa_2 np)^{-1}$, then it holds with probability at least $1 - O(n^{-8})$ that
	\begin{align}
		\|\widetilde\bth^{t+1} - \bth^*\|_2 < C_1\left(\sqrt{\frac{\log n}{p}} + \frac{\lambda\log n}{\sqrt{n}p}\right). \label{eq: inductive step 1}
	\end{align}
\end{proposition}

[Proof of Proposition \ref{prop: inductive step 1}]
	Throughout the proof, we assume that events $\mathcal A_0, \mathcal A_1, \mathcal A_2$ simultaneously occur, the probability of which is at least $1 - O(n^{-8})$.
	
	By \eqref{eq: perturbed gradient descent} and the definition of $\widetilde {\mathcal R}(\bth)$, we have
	\begin{align}
		\widetilde\bth^{t+1} - \bth^* &= \widetilde\bth^t - \bth^* - \eta \left(\nabla\widetilde{\mathcal R}(\widetilde\bth^t) - \nabla\widetilde{\mathcal R}(\bth^*) \right) -  \eta\nabla\widetilde{\mathcal R}(\bth^*) \notag \\
		&= (1 - \eta\gamma)(\widetilde\bth^t - \bth^*) - \eta \left(\nabla{\mathcal L}(\widetilde\bth^t) - \nabla{\mathcal L}(\bth^*) \right) -  \eta\nabla\widetilde{\mathcal R}(\bth^*). \label{eq: inductive step 1 expansion 1}
	\end{align}
	By the mean value theorem for vector-valued functions (\cite{lang2012real}, Chapter XIII, Theorem 4.2), we have
	\begin{align*}
		\nabla{\mathcal L}(\widetilde\bth^t) - \nabla{\mathcal L}(\bth^*) = \int_0^1 \nabla^2 \mathcal L(\bth(\tau)) \d\tau \cdot (\widetilde\bth^t - \bth^*),
	\end{align*}
	where $\bth(\tau) = \bth^* + \tau(\widetilde\bth^t - \bth^*)$. Denoting $\bm H = \int_0^1 \nabla^2 \mathcal L(\bth(\tau)) d\tau$, it follows from \eqref{eq: inductive step 1 expansion 1} that
	\begin{align}
		\|\widetilde\bth^{t+1} - \bth^*\|_2 \leq \|((1-\eta\gamma)\bm I - \eta\bm H) (\widetilde\bth^t - \bth^*)\|_2 + \eta\|\nabla\widetilde{\mathcal R}(\bth^*)\|_2. \label{eq: inductive step 1 expansion 2}
	\end{align}
	The second term on the right side can be bounded as follows, under high-probability events $\mathcal A_1$ and $\mathcal A_2$, 
	\begin{align}
		\|\nabla \widetilde{\mathcal R}(\bth^*)\|_2 \leq \|\nabla\mathcal L(\bth^*)\|_2 + \gamma\|\bth^*\|_2 + \|\bm w\|_2 \leq \left(3\kappa_1\sqrt{n^2p\log n} + c_0\sqrt{n^2p\log n} + 9\lambda\sqrt{n}\log n\right). \label{eq: inductive step 1 expansion 2.1}
	\end{align} 
	In bounding $\|\bth^*\|_2$, we used the fact that $\|\bth^*\|_\infty < 1$.
	
	It remains to bound the first term on the right side of \eqref{eq: inductive step 1 expansion 2}. Observe from the Hessian expression \eqref{eq: likelihood hessian} that, for every $\tau$, $\nabla^2 \mathcal L(\bth(\tau)) \bm 1 = 0$ and therefore $\bm H \bm 1  = 0$. It follows that
	\begin{align}
		&\|((1-\eta\gamma)\bm I - \eta\bm H) (\widetilde\bth^t - \bth^*)\|_2 \notag \\
		&\leq \|((1-\eta\gamma)\bm I - \eta\bm H) (\widetilde\bth^t - \bth^* - (\bm 1^\top \widetilde\bth^t/n)\bm 1)\|_2 + (1-\eta\gamma)\|(\bm 1^\top \widetilde\bth^t/n)\bm 1\|_2  \notag \\
		&\leq \max\left(|1 - \eta\gamma - \eta\lambda_2(\bm H)|, |1 - \eta\gamma- \eta\lambda_{\max}(\bm H)| \right)\|\widetilde\bth^t - \bth^* - (\bm 1^\top \widetilde\bth^t/n)\bm 1\|_2 + |\bm 1^\top \widetilde\bth^t|/\sqrt{n}  \notag \\
		& \leq \max\left(|1 - \eta\gamma-\eta\lambda_2(\bm H)|, |1- \eta\gamma - \eta\lambda_{\max}(\bm H)| \right)\|\widetilde\bth^t - \bth^*\|_2 + |\bm 1^\top \widetilde\bth^t|/\sqrt{n}. \label{eq: inductive step 1 expansion 2.2}
	\end{align}
	To see why the second inequality holds, observe that $(\widetilde\bth^t - \bth^* - (\bm 1^\top \widetilde\bth^t/n)\bm 1)$ is orthogonal to $\bm 1$ by the assumption of $1^\top \bth^* = 0$, and that $\bm H$ is symmetric and positive semi-definite. The third inequality is true because $(\widetilde\bth^t - \bth^* - (\bm 1^\top \widetilde\bth^t/n)\bm 1)$ is an orthogonal projection of $\widetilde\bth^t - \bth^*$, and $\|\widetilde\bth^t - \bth^* - (\bm 1^\top \widetilde\bth^t/n)\bm 1\|_2 \leq \|\widetilde\bth^t - \bth^*\|_2$ by the non-expansiveness of projection.
	
	To further simplify \eqref{eq: inductive step 1 expansion 2.2}, we apply the following lemmas. Their proofs are deferred to Section \ref{sec: proof of ranking upper bound lemmas}.
	
	\begin{Lemma}
		\label{lm: spectrum of likelihood Hessian}
		Under the event $\mathcal A_0$, it holds with probability at least $1 - O(n^{-10})$ that 
		\begin{align*}
			\frac{np}{2\kappa_2} \leq \lambda_{2} (\nabla^2 \mathcal L(\bm v)) \leq \lambda_{\max} (\nabla^2 \mathcal L(\bm v)) \leq 3\kappa_2 np.
		\end{align*}
		simultaneously for all $\bm v$ satisfying $\|\bm v - \bth^*\|_\infty \leq 1$.
	\end{Lemma}
	
	\begin{Lemma}\label{lm: elementwise sum of perturbed gradient descent}
		Let $\widetilde\bth^t$ be defined by \eqref{eq: perturbed gradient descent} and $\eta = (\gamma + 3\kappa_2 np)^{-1}$. Then under the event $\mathcal A_2$, it holds for every $t \in \mathbb N$ that $|\bm 1^\top \widetilde \bth^t| \leq 8\lambda\sqrt{n\log n}/\gamma$.
	\end{Lemma}

	Returning to \eqref{eq: inductive step 1 expansion 2.2}, by the inductive hypothesis \eqref{eq: inductive hypothesis 2} and the assumption $p \geq c_1\lambda \log n/n$ for a sufficiently large constant $c_1 > 0$, both parts of Lemma \ref{lm: spectrum of likelihood Hessian} are applicable, implying that
	\begin{align*}
		&\|((1-\eta\gamma)\bm I - \eta\bm H) (\widetilde\bth^t - \bth^*)\|_2 \leq (1- (\eta/2\kappa_2) \cdot np) \|\widetilde\bth^t - \bth^*\|_2 + 8\lambda\sqrt{\log n}/\gamma
	\end{align*}
	Finally, we combine this inequality with \eqref{eq: inductive step 1 expansion 2}, \eqref{eq: inductive step 1 expansion 2.1} and the inductive hypothesis \eqref{eq: inductive hypothesis 1} to obtain
	\begin{align*}
		&\|\widetilde\bth^{t+1} - \bth^*\|_2 \\
		&\leq (1- (\eta/2\kappa_2) \cdot np)C_1\left(\sqrt{\frac{\log n}{p}} + \frac{\lambda\log n}{\sqrt{n}p}\right) + (3\kappa_1+c_0)\sqrt{\frac{\log n}{p}} + \frac{9\lambda\log n}{\sqrt{n}p} + \frac{8\lambda}{(np)^{1/2}}.
	\end{align*}
	In simplifying the right side, recall that $\eta = \frac{1}{\gamma + 3\kappa_2 np} \lesssim (np)^{-1}$ and $\gamma = c_0\sqrt{np\log n}$. As long as $C_1 > 0$ is sufficiently large, the desired conclusion \eqref{eq: inductive step 1} follows.

\begin{proposition}\label{prop: inductive step 3}
	Suppose \eqref{eq: inductive hypothesis 3}, \eqref{eq: inductive hypothesis 4}, \eqref{eq: inductive hypothesis 2} are true for sufficiently large $C_2, C_3, C_4 > 0$ and $\eta = (\gamma + 3\kappa_2 np)^{-1}$, then it holds with probability at least $1 - O(n^{-8})$ that
	\begin{align}
		\max_{1 \leq m \leq n} |\theta^{(m), t+1}_m - \theta^*_m| < C_2 \left(\sqrt{\frac{\log n}{np}} + \frac{\lambda\log n}{np}\right). \label{eq: inductive step 3}
	\end{align}
\end{proposition}

[Proof of Proposition \ref{prop: inductive step 3}]
	Throughout the proof, we fix $m \in [n]$ and let $\bm v$ denote $\bth^{(m), t}$ for the brevity of notation. 
	
	By the gradient descent step \eqref{eq: leave-one-out gradient descent} and the leave-one-out objective functions \eqref{eq: leave-one-out likelihood} and \eqref{eq: leave-one-out objective},
	\begin{align}
		&\theta^{(m), t+1}_m - \theta_m = v_m - \eta\left(\nabla \widetilde{\mathcal R}^{(m)}\left(\bm v \right)\right)_m - \theta_m \notag \\
		&= v_m - \eta p \sum_{1 \leq j \leq n, j \neq m} \left(F_{mj}(\bm v) - F_{mj}(\bth^*)\right)\frac{F'_{mj}(\bm v)}{F_{mj}(\bm v)(1 - F_{mj}(\bm v))} - \eta\gamma v_m - \eta W_m - \theta_m \notag \\
		& = (1 - \eta\gamma)(v_m - \theta_m) - \eta p \sum_{1 \leq j \leq n, j \neq m} \left(F_{mj}(\bm v) - F_{mj}(\bth^*)\right)\frac{F'_{mj}(\bm v)}{F_{mj}(\bm v)(1 - F_{mj}(\bm v))} - \eta W_m - \eta\gamma\theta_m. \label{eq: inductive step 3 expansion 1}
	\end{align}
	To bound the sum term, applying the mean value theorem to $F_{mj}(\bm v) \equiv F(v_m - v_j)$ yields
	\begin{align*}
		F_{mj}(\bm v) - F_{mj}(\bth^*) = \left((\theta^*_j-v_j) + (v_m - \theta^*_m)\right)F'(b_j)
	\end{align*}
	for some $b_j$ between $v_j - v_m$ and $\theta^*_j - \theta^*_m$. Plugging back into \eqref{eq: inductive step 3 expansion 1} leads to
	\begin{align}
		&\left|\theta^{(m), t+1}_m - \theta_m\right| \notag \\
		&\leq \underbrace{\left|1 - \eta\gamma - \eta p \sum_{1 \leq j \leq n, j \neq m} F'(b_j)\frac{F'_{mj}(\bm v)}{F_{mj}(\bm v)(1 - F_{mj}(\bm v))}\right||v_m - \theta^*_m|}_{\circled{1}} \notag\\
		&\quad \quad + \underbrace{\eta p \left|\sum_{1 \leq j \leq n, j \neq m} F'(b_j)\frac{F'_{mj}(\bm v)}{F_{mj}(\bm v)(1 - F_{mj}(\bm v))}(v_j - \theta^*_j)\right|}_{\circled{2}} + \underbrace{\eta |W_m| + \eta\gamma|\theta^*_m|}_{\circled{3}}. 
		\label{eq: inductive step 3 expansion 2}
	\end{align}
	It remains to bound $\circled{1}, \circled{2}$ and $\circled{3}$.
	
	For $\circled{1}$ in \eqref{eq: inductive step 3 expansion 2}, by assumptions (A1) and (A2), as long as the constants satisfy $\kappa_2 > \kappa_1^2/12$, 
	\begin{align*}
		&1 - \eta\gamma - \eta p\sum_{1 \leq j \leq n, j \neq m} F'(b_j)\frac{F'_{mj}(\bm v)}{F_{mj}(\bm v)(1 - F_{mj}(\bm v))} \geq 1 - \eta(\gamma + \kappa_1^2 np/4) = \frac{(3\kappa_2 - \kappa_1^2/4)np}{\gamma + np} > 0.\\
		& 1 - \eta\gamma - \eta p\sum_{1 \leq j \leq n, j \neq m} F'(b_j)\frac{F'_{mj}(\bm v)}{F_{mj}(\bm v)(1 - F_{mj}(\bm v))} \leq 1 - \eta p \left(\min_{j \in [n]} F'(b_j)\right)\sum_{1 \leq j \leq n, j \neq m}4F'_{mj}(\bm v).
	\end{align*}
	
	Recall that each $b_j$ is between $v_j-v_m$ and $\theta^*_j - \theta^*_m$. Since $\|\bth^*\|_\infty \leq 1$ and $\bm v \equiv \bth^{(m), t}$, by inductive hypotheses \eqref{eq: inductive hypothesis 4} and \eqref{eq: inductive hypothesis 2} we have
	\begin{align*}
		\max_j |b_j| &\leq \max_j \max(|v_m - v_j|, |\theta^*_m - \theta^*_j|) \leq 2\|\bth^*\|_\infty + 2\|\bm v  - \bth^*\|_\infty \leq 2 + 2\|\bm v  - \bth^*\|_\infty \\
		&\leq 2 + 2\max_{1 \leq m \leq n}\|\bth^{(m), t} - \widetilde\bth^t\|_2 + 2\|\widetilde\bth^t - \bth^*\|_\infty \leq 2 + 2(C_3 + C_4)\left(\sqrt{\frac{\log n}{np}} + \frac{\lambda\log n}{np}\right) \lesssim 1,
	\end{align*}
	where the last inequality holds by the assumption that $p \geq c_1 \lambda \log n/n$ and $\lambda \geq c_2$ for sufficiently large constants $c_1, c_2 \geq 0$. By assumption (A0) $F'(x) > 0$ for every $x$; as the $b_j$'s are bounded, $\min_{j \in [n]} F'(b_j)$ is must be bounded away from $0$. Since $F'_{mj}(\bm v) \equiv F'(v_m - v_j)$, a similar argument using the inductive hypotheses shows that $\sum_{1 \leq j \leq n, j \neq m}4F'_{mj}(\bm v) \gtrsim n$. 
	
	Therefore, by the derivations so far, we know there are some $c_5 > c_6 > 0$ such that
	\begin{align}
		\left|1 - \eta\gamma - \eta p\sum_{1 \leq j \leq n, j \neq m} F'(b_j)\frac{F'_{mj}(\bm v)}{F_{mj}(\bm v)(1 - F_{mj}(\bm v))}\right| \leq (1 - c_5 \eta p n) \leq (1-c_6).\label{eq: inductive step 3 expansion 2.1}
	\end{align}
	The last inequality is true by $\gamma \asymp \sqrt{np \log n} \lesssim np$ and $\eta = (\gamma + 3\kappa_2 np)^{-1} \asymp (np)^{-1}$.
	
	For \circled{2} in \eqref{eq: inductive step 1 expansion 2}, the Cauchy-Schwarz inequality and inductive hypotheses \eqref{eq: inductive hypothesis 1}, \eqref{eq: inductive hypothesis 4} imply
	\begin{align}
		&\left|\sum_{1 \leq j \leq n, j \neq m} F'(b_j)\frac{F'_{mj}(\bm v)}{F_{mj}(\bm v)(1 - F_{mj}(\bm v))}(v_j - \theta^*_j)\right| \notag \\
		&\quad \quad \leq \frac{\kappa_1^2 \sqrt{n}}{4}\|\bm v - \bth^*\|_2 \leq \frac{(C_1 + C_3)\kappa_1^2\sqrt{n}}{4} \left(\sqrt{\frac{\log n}{p}} + \frac{\lambda\log n}{\sqrt{n}p}\right). \label{eq: inductive step 3 expansion 2.2}
	\end{align}
	
	For \circled{3} in \eqref{eq: inductive step 1 expansion 2}, under event $\mathcal A_2$ which occurs with probability at least $1 - O(n^{-8})$ by Lemma \ref{lm: Laplace tail bound}, we have
	\begin{align}
		\eta |W_m| + \eta\gamma|\theta^*_m| \leq \eta\|\bm w\|_\infty + \eta\gamma\|\bth^*\|_\infty \leq \frac{9\lambda\log n}{3\kappa_2np} + \frac{c_0\sqrt{np\log n}}{3\kappa_2np}. \label{eq: inductive step 3 expansion 2.3}
	\end{align}
	Finally, combining the individual terms \eqref{eq: inductive step 3 expansion 2.1}, \eqref{eq: inductive step 3 expansion 2.2} and \eqref{eq: inductive step 3 expansion 2.3} with \eqref{eq: inductive step 3 expansion 2} leads to
	\begin{align*}
		&\left|\theta^{(m), t+1}_m - \theta^*_m\right| \\
		&\leq (1 - c_6)|v_m - \theta^*_m| + \eta p \frac{(C_1 + C_3)\kappa^2_1\sqrt{n}}{4} \left(\sqrt{\frac{\log n}{p}} + \frac{\lambda\log n}{\sqrt{n}p}\right) + \frac{3\lambda\log n}{\kappa_2 np} + \frac{c_0\sqrt{np\log n}}{3\kappa_2 np} \\
		& \leq (1 - c_6)C_2\left(\sqrt{\frac{\log n}{np}} + \frac{\lambda\log n}{np}\right) + \frac{(C_1 + C_3)\kappa_1^2}{4}\left(\sqrt{\frac{\log n}{np}} + \frac{\lambda\log n}{np}\right) + \frac{3\lambda\log n}{\kappa_2 np} + \frac{c_0}{3\kappa_2}\sqrt{\frac{\log n}{np}} \\
		& \leq C_2\left(\sqrt{\frac{\log n}{np}} + \frac{\lambda\log n}{np}\right).
	\end{align*}
	The last inequality is true as long as $C_2$ is much larger than $\max(c_0/3\kappa_2, 3/\kappa_2, (C_1 + C_3)\kappa_1^2)$.

\begin{proposition}\label{prop: inductive step 4}
	Suppose \eqref{eq: inductive hypothesis 4} and \eqref{eq: inductive hypothesis 2} are true for sufficiently large $C_3, C_4 > 0$ and $\eta = \frac{1}{\gamma + 3\kappa_2 np}$, then it holds with probability at least $1 - O(n^{-8})$ that
	\begin{align}
		\max_{1 \leq m \leq n}\|\bth^{(m), t+1} - \widetilde\bth^{t+1}\|_2 < C_3\left(\sqrt{\frac{\log n}{np}} + \frac{\lambda\log n}{np}\right). \label{eq: inductive step 4}
	\end{align}
\end{proposition}

[Proof of Proposition \ref{prop: inductive step 4}]
	By \eqref{eq: perturbed gradient descent} and \eqref{eq: leave-one-out gradient descent} we have
	\begin{align}
		\bth^{(m), t+1} - \widetilde\bth^{t+1} &= \left(\bth^{(m), t} - \eta\nabla\widetilde{\mathcal R}^{(m)}\left(\bth^{(m), t}\right)\right) - \left(\widetilde\bth^{t} - \eta\nabla\widetilde{\mathcal R}(\widetilde\bth^{t})\right) \notag \\
		&=  \underbrace{\left(\bth^{(m), t} - \eta\nabla\widetilde{\mathcal R}\left(\bth^{(m), t}\right)\right) - \left(\widetilde\bth^{t} - \eta\nabla\widetilde{\mathcal R}(\widetilde\bth^{t})\right)}_{\circled{1}} + \underbrace{\eta\left(\nabla\widetilde{\mathcal R}\left(\bth^{(m), t}\right) - \nabla\widetilde{\mathcal R}^{(m)}\left(\bth^{(m), t}\right)\right)}_{\circled{2}}. \label{eq: inductive step 4 expansion 1}
	\end{align}
	Next we analyze the terms \circled{1} and \circled{2} separately. 
	
	For \circled{1} in \eqref{eq: inductive step 4 expansion 1}, the analysis is similar to the proof of Proposition \ref{prop: inductive step 1}. We have
	\begin{align*}
		\circled{1} &= (1-\eta\gamma)(\bth^{(m), t} - \widetilde\bth^{t}) - \eta(\nabla\mathcal L(\bth^{(m), t}) - \nabla\mathcal L(\widetilde\bth^{t})) \\
		&= (1-\eta\gamma)(\bth^{(m), t} - \widetilde\bth^{t}) - \eta\left(\int_0^1 \nabla^2 \mathcal L(\bth(\tau)) \d \tau\right)(\bth^{(m), t} - \widetilde\bth^{t}),
	\end{align*}
	where the second equality is true by the mean value theorem for vector-valued functions (\cite{lang2012real}, Chapter XIII, Theorem 4.2), and $\bth(\tau) = \widetilde\bth^{t} + \tau(\bth^{(m), t} - \widetilde\bth^{t})$. With the notation $\bm H = \int_0^1 \nabla^2 \mathcal L(\bth(\tau)) \d \tau$, the same reasoning for \eqref{eq: inductive step 1 expansion 2.2} yields
	\begin{align}
		\left\|\circled{1}\right\|_2 &\leq \|((1-\eta\gamma)\bm I - \eta\bm H) (\bth^{(m), t} - \widetilde\bth^t)\|_2 \notag \\
		&\leq \|((1-\eta\gamma)\bm I - \eta\bm H) (\bth^{(m), t} - \widetilde\bth^t - (\bm 1^\top(\bth^{(m), t} -  \widetilde\bth^t)/n)\bm 1)\|_2 \notag\\
		&\quad + (1 - \eta\gamma)(|\bm 1^\top\bth^{(m), t}| + |\bm 1^\top \widetilde\bth^t|)/\sqrt{n} \notag\\
		&\leq  \max\left(|1 - \eta\gamma-\eta\lambda_2(\bm H)|, |1- \eta\gamma - \eta\lambda_{\max}(\bm H)| \right)\|\bth^{(m), t} - \widetilde\bth^t\|_2  \notag\\
		&\quad + (1 - \eta\gamma)(|\bm 1^\top\bth^{(m), t}| + |\bm 1^\top \widetilde\bth^t|)/\sqrt{n}. \label{eq: inductive step 4 expansion 2.1}
	\end{align}
	For the $\ell_2$ norm part, it follows from the inductive hypotheses \eqref{eq: inductive hypothesis 4} and \eqref{eq: inductive hypothesis 2}  as well as the assumptions $p > c_1\lambda \log n/ n, \lambda > c_2$ that $\|\bth^{(m), t} - \bth^*\|_\infty, \|\widetilde\bth^t - \bth^*\|_\infty < 1$ for sufficiently large $n$, making both parts of Lemma \ref{lm: spectrum of likelihood Hessian} applicable. We then have
	\begin{align}
		\max\left(|1 - \eta\gamma-\eta\lambda_2(\bm H)|, |1- \eta\gamma - \eta\lambda_{\max}(\bm H)| \right)\|\bth^{(m), t} - \widetilde\bth^t\|_2 \leq (1 - (\eta/2\kappa_2) \cdot np) \|\bth^{(m), t} - \widetilde\bth^t\|_2. \label{eq: inductive step 4 expansion 2.1.1}
	\end{align}
	For the second term of \eqref{eq: inductive step 4 expansion 2.1}, we have a bound for $|\bm 1^\top \widetilde\bth^t|$ by Lemma \ref{lm: elementwise sum of perturbed gradient descent}. Since $\bm 1^\top \nabla \mathcal L^{(m)}(\bm v) = 0 $ for every $\bm v$ and $m$ by \eqref{eq: leave-one-out likelihood}, an identical argument as the proof of Lemma \ref{lm: elementwise sum of perturbed gradient descent} shows that, under the event $\mathcal A_2$,
	\begin{align}
		|\bm 1^\top\bth^{(m), t}| \leq 8\lambda\sqrt{n\log n}/\gamma, \forall t \in \mathbb N, 1 \leq m \leq n. \label{eq: inductive step 4 expansion 2.1.2}
	\end{align}
	Finally, we combine this inequality with \eqref{eq: inductive step 4 expansion 2.1}, \eqref{eq: inductive step 4 expansion 2.1.1} and the inductive hypothesis \eqref{eq: inductive hypothesis 4} to obtain
	\begin{align}
		\|\circled{1}\|_2 \leq (1- c_7\eta \cdot np)C_3\left(\sqrt{\frac{\log n}{np}} + \frac{\lambda\log n}{np}\right) + \frac{16\lambda}{(np)^{1/2}}. \label{eq: inductive step 4 expansion 2}
	\end{align}
	
	Next, for \circled{2} in \eqref{eq: inductive step 4 expansion 1}, by the definitions \eqref{eq: perturbed objective} and \eqref{eq: leave-one-out objective} we have
	\begin{align}\label{eq: inductive step 4 expansion 2.2}
		\circled{2} = \eta\left(\nabla\widetilde{\mathcal R}\left(\bth^{(m), t}\right) - \nabla\widetilde{\mathcal R}^{(m)}\left(\bth^{(m), t}\right)\right) = \eta\left(\nabla{\mathcal L}\left(\bth^{(m), t}\right) - \nabla{\mathcal L}^{(m)}\left(\bth^{(m), t}\right)\right),
	\end{align}
	and we prove in Section \ref{sec: proof of eq: inductive step 4 expansion 2.2.1} that, with probability at least $1 - O(n^{-8})$,
	\begin{align}\label{eq: inductive step 4 expansion 2.2.1}
		\left\|\eta\left(\nabla{\mathcal L}\left(\bth^{(m), t}\right) - \nabla{\mathcal L}^{(m)}\left(\bth^{(m), t}\right)\right)\right\|_2 \lesssim \eta\kappa_1^2\left(\sqrt{np\log n} + \sqrt{np\log n}\|\bth^{(m), t} - \bth^*\|_\infty\right).
	\end{align}
	Plugging \eqref{eq: inductive step 4 expansion 2.2.1}, \eqref{eq: inductive step 4 expansion 2.2} and \eqref{eq: inductive step 4 expansion 2} back into \eqref{eq: inductive step 4 expansion 1} yields
	\begin{align*}
		&\|\bth^{(m), t+1} - \widetilde\bth^{t+1}\|_2 \\
		&\leq (1- (\eta/2\kappa_2) \cdot np)C_3\left(\sqrt{\frac{\log n}{np}} + \frac{\lambda\log n}{np}\right) + \frac{16\lambda}{(np)^{1/2}} + C\eta\left(\sqrt{np\log n} + \sqrt{np\log n}\|\bth^{(m), t} - \bth^*\|_\infty\right) \\
		& \leq C_3\left(\sqrt{\frac{\log n}{np}} + \frac{\lambda\log n}{np}\right),
	\end{align*}
	as long as $C_3$ is chosen to be a sufficiently large constant. To see the last inequality, recall that $np \gtrsim \log n$, $\eta \leq (np)^{-1}$, $\lambda \lesssim \sqrt{\log n}$, and $\|\bth^{(m), t} - \bth^*\|_\infty \leq (C_3 + C_4)\left(\sqrt{\frac{\log n}{np}} + \frac{\lambda\log n}{np}\right) \lesssim 1$ by inductive hypotheses \eqref{eq: inductive hypothesis 4} and \eqref{eq: inductive hypothesis 2}.

\begin{proposition}\label{prop: inductive step 2}
	Suppose \eqref{eq: inductive hypothesis 1}, \eqref{eq: inductive hypothesis 3}, \eqref{eq: inductive hypothesis 4} and \eqref{eq: inductive hypothesis 2} are true for sufficiently large constants $C_1, C_2, C_3, C_4 > 0$ and $\eta = (\gamma + 3\kappa_2 np)^{-1}$, then it holds with probability at least $1 - O(n^{-8})$ that
	\begin{align}\label{eq: inductive step 2}
		\|\widetilde\bth^{t+1} - \bth^*\|_\infty < C_4\left(\sqrt{\frac{\log n}{np}} + \frac{\lambda\log n}{np}\right).
	\end{align}
\end{proposition}

[Proof of Proposition \ref{prop: inductive step 2}]
	When the inductive hypotheses are true, by Propositions \ref{prop: inductive step 3} and \ref{prop: inductive step 4}, it holds with probability at least $1 - O(n^{-8})$ that
	\begin{align*}
		|\widetilde\theta^{t+1}_m - \theta^*_m| &\leq |\widetilde\theta^{t+1}_m - \theta^{(m), t+1}_m| + |\theta^{(m), t+1}_m - \theta^*_m|\\
		&\leq \|\widetilde\bth^{t+1} - \bth^{(m), t+1}\|_2 + |\theta^{(m), t+1}_m - \theta^*_m| \leq (C_2 + C_3) \left(\sqrt{\frac{\log n}{np}} + \frac{\lambda\log n}{np}\right). 
	\end{align*}
	The proof is complete by choosing $C_4 \geq C_2 + C_3$.

\subsubsection{Proof of auxiliary lemmas in Section \ref{sec: proof of induction in ranking MLE upper bound}} \label{sec: proof of ranking upper bound lemmas}
[Proof of Lemma \ref{lm: gradient norm tail bound}]
	It suffices to upper bound $\Pro(\mathcal A_1^c \cap \mathcal A_0)$. By \eqref{eq: likelihood gradient} and assumption (A1),
	\begin{align*}
		\|\nabla \mathcal L(\bth^*)\|_2 \leq \kappa_1\sqrt{n} \max_{i \in [n]} \left|\sum_{(i, j) \in \mathcal G} F_{ij}(\bth^*) - y_{ij}\right|.
	\end{align*}
	Since $\E y_{ij} = F_{ij}(\bth^*)$ and $|F_{ij}(\bth^*) - y_{ij}| \leq 1$, by Hoeffding's inequality we have, for every $t > 0$,
	\begin{align*}
		\Pro\left(\left\{\max_{i \in [n]} \left|\sum_{(i, j) \in \mathcal G} F_{ij}(\bth^*) - y_{ij}\right| > t \right\} \cap \mathcal A_0\right) \leq n\exp\left(-\frac{4t^2}{3np}\right).
	\end{align*}
	Setting $t = 3\sqrt{np\log n}$ completes the proof.

[Proof of Lemma \ref{lm: Laplace tail bound}]
	The Laplace distribution with scale $\lambda$ is sub-exponential with parameters $(2\lambda, 4)$. Bernstein's inequality implies that, for $0 < t < n\lambda^2$, 
	\begin{align*}
		\Pro(\bm 1^\top \bm w > t) \leq \exp(-t^2/8n\lambda).
	\end{align*}
	Since $\lambda \gtrsim 1$, we choose $t = 8\lambda\sqrt{n\log n}$ and obtain that $\Pro(\bm 1^\top \bm w > 8\lambda\sqrt{n\log n}) \leq n^{-8}$.
	
	For the maximum, by the union bound we have
	\begin{align*}
		\Pro(\|\bm w\|_\infty > t) \leq n\Pro(|W_1| > t) \leq n\exp(-t/\lambda).
	\end{align*}
	Setting $t = 9\lambda\log n$ completes the proof.

[Proof of Lemma \ref{lm: spectrum of likelihood Hessian}]
	Let $L_{\mathcal{G}}$ denote the Laplacian matrix of comparison graph $\mathcal G$. By the Hessian \eqref{eq: likelihood hessian} and assumptions (A0) and (A2), when $\mathcal A_0$ occurs we have 
	\begin{align*}
		\lambda_{\max} (\nabla^2 \mathcal L(\bm v)) \leq \kappa_2\lambda_{\max}(L_{\mathcal G}) \leq 3\kappa_2 np.
	\end{align*}
	For the lower bound, $\|\bm v - \bth^*\|_\infty \leq 1$ and $\|\bth^*\|_\infty \leq 1$ together imply $\max_{i,j}|v_i - v_j| \leq 4$. By Lemma 10 in \cite{chen2019spectral}, with probability at least $1 - O(n^{-10})$ we have $\lambda_2(L_{\mathcal G}) \geq np/2$ and therefore $\lambda_{2} (\nabla^2 \mathcal L(\bm v)) \geq np/2\kappa_2$.

[Proof of Lemma \ref{lm: elementwise sum of perturbed gradient descent}]
	By \eqref{eq: perturbed gradient descent} we have
	\begin{align*}
		\bm 1^\top\widetilde\bth^t = \bm 1^\top\widetilde\bth^{t-1} - \eta\bm 1^\top\nabla\mathcal L(\widetilde\bth^{t-1}) - \eta\gamma\bm 1^\top\widetilde\bth^{t-1} - \eta\bm1^\top\bm w
	\end{align*}
	Observe from the likelihood gradient \eqref{eq: likelihood gradient} that $\bm 1^\top \nabla\mathcal L(\bm v) = 0$ for every $\bm v$, and therefore
	\begin{align*}
		|\bm 1^\top\widetilde\bth^t| \leq (1-\eta\gamma)|\bm 1^\top\widetilde\bth^{t-1}| + \eta|\bm 1^\top\bm w|.
	\end{align*}
	Iterating over $t$ and using the initial condition $\bm 1^\top \widetilde\bth^0 = \bm 1^\top\bth^* = 0$ yield
	\begin{align*}
		|\bm 1^\top\widetilde\bth^t| \leq \left(\sum_{k=0}^t (1-\eta\gamma)^k \right) \eta |\bm 1^\top\bm w| \leq |\bm 1^\top\bm w|/\gamma \leq 8\lambda\sqrt{n\log n}/\gamma.
	\end{align*}
	The last inequality is true under the event $\mathcal A_2$.

\subsubsection{Proof of Equation \eqref{eq: inductive step 4 expansion 2.2.1} in Section \ref{sec: proof of induction in ranking MLE upper bound}} \label{sec: proof of eq: inductive step 4 expansion 2.2.1}
Throughout the proof we abbreviate $\bth^{(m), t}$ as $\bm v$ and assume that the event $\mathcal A_0$ occurs. Let
\begin{align*}
	&A_j = \1_{(j, m) \in \mathcal G}(y_{mj} - F_{mj}(\bth^*))\frac{F'_{mj}(\bm v)}{F_{mj}(\bm v)(1 - F_{mj}(\bm v))}, \\
	&B_j = (p - \1_{(j, m) \in \mathcal G}) (F_{mj}(\bm v) - F_{mj}(\bth^*))\frac{F'_{mj}(\bm v)}{F_{mj}(\bm v)(1 - F_{mj}(\bm v))}.
\end{align*}
Then from \eqref{eq: likelihood gradient} and \eqref{eq: leave-one-out likelihood} we have
\begin{align*}
	\left(\nabla{\mathcal L}\left(\bth^{(m), t}\right) - \nabla{\mathcal L}^{(m)}\left(\bth^{(m), t}\right)\right)_j = \left(\nabla{\mathcal L}(\bm v ) - \nabla{\mathcal L}^{(m)}(\bm v)\right)_j = 
	\begin{cases}
		A_j + B_j &\text{ if } j \neq m, \\
		\sum_{j \neq m} A_j + B_j &\text{ if } j = m.  
	\end{cases}
\end{align*}
It follows that
\begin{align}\label{eq: expansion 1 in proof of equation eq: inductive step 4 expansion 2.2.1}
	\|\nabla{\mathcal L}(\bm v ) - \nabla{\mathcal L}^{(m)}(\bm v)\|_2^2 \leq 2\underbrace{\sum_{j\neq m} A_j^2}_{\circled{1}} + 2\underbrace{\sum_{j \neq m} B_j^2}_{\circled{2}} + \underbrace{\left(	\sum_{j \neq m} A_j + B_j\right)^2}_{\circled{3}}. 
\end{align}
To bound the terms in \eqref{eq: expansion 1 in proof of equation eq: inductive step 4 expansion 2.2.1}, we first observe that, by assumption (A1) and the mean-value theorem, 
\begin{align*}
	|F_{mj}(\bm v) - F_{mj}(\bth^*)| = |F(v_m - v_j) - F(\theta^*_m - \theta^*_j)| \leq 2\|\bm v - \bth^*\|_\infty \frac{\kappa_1}{4}.
\end{align*}
Then, under the event $\mathcal A_0$ which upper bounds the maximum degree of $\mathcal G$, we have
\begin{align*}
	&\circled{1} \leq \frac{3np}{2}\kappa_1^2, \\
	&\circled{2} \leq \sum_{(j,m) \in \mathcal G} B_j^2 + \sum_{(j,m) \not\in \mathcal G} B_j^2 \leq \frac{3np}{2}\cdot (1-p)^2 \frac{\kappa_1^2\|\bm v - \bth^*\|_\infty^2}{4} \cdot \kappa_1^2 + n \cdot p^2 \frac{\kappa_1^2\|\bm v - \bth^*\|_\infty^2}{4} \cdot \kappa_1^2 \\
	& \quad  \leq np \kappa_1^4 \|\bm v - \bth^*\|_\infty^2.
\end{align*}
To bound $\circled{3}$, by Hoeffding's and Bernstein's inequalities we have, for every $t > 0$, 
\begin{align*}
	&\Pro\left(\left|\sum_{j \neq m} A_j \right| > t\right) \leq 2\exp\left(-\frac{t^2}{np\kappa_1^2}\right),\\
	&\Pro\left(\left|\sum_{j \neq m} B_j \right| > t\right) \leq 2\exp\left(-\frac{-t^2}{2\sum_{j \neq m} \E B_j^2 + \frac{2}{3}\frac{\kappa^2_1\|\bm v - \bth^*\|_\infty}{2}t}\right).
\end{align*}
As $\E B_j^2 \leq p\frac{\kappa_1^4}{4}\|\bm v - \bth^*\|_\infty^2$ and $\sqrt{np\log n} \lesssim np$ by the assumption $p \gtrsim \log n/n$, choosing $t = C\kappa_1^2\max(1, \|\bm v  - \bth^*\|_\infty)\sqrt{np \log n}$ for a sufficiently large constant $C$ leads to
\begin{align*}
	\Pro\left( \circled{3} > 4C^2 \kappa_1^4\max(1, \|\bm v  - \bth^*\|^2_\infty)np \log n\right) \leq O(n^{-8}).
\end{align*}
The proof is complete by combining with the bounds for $\circled{1}$ and $\circled{2}$ in \eqref{eq: expansion 1 in proof of equation eq: inductive step 4 expansion 2.2.1}.

\subsection{Proof of Proposition \ref{prop: ranking attack soundness}}
\label{sec: proof of prop: ranking attack soundness}

	For upper bounding $\sum_{1 \leq i < j \leq n} \E A^{(k)}_{ij}$, suppose $\widetilde{\bm Y}_{ij}$ is an adjacent data set of $\bm Y$ obtained by replacing $Y_{ij}$ with an independent copy; we denote $\mathcal A^{(k)}(M(\widetilde{\bm Y}_{ij}), Y_{ij})$ by $\widetilde A^{(k)}_{ij}$. Since $M$ is $(\varepsilon, \delta)$-DP, for any $i, j, k$ and any $T > 0$ we have
	\begin{align*}
		\E A^{(k)}_{ij} &= \int_0^\infty \Pro((A^{(k)}_{ij})^+ > t) \d t - \int_0^\infty \Pro((A^{(k)}_{ij})^- > t) \d t \\
		&\leq \int_0^T \Pro((A^{(k)}_{ij})^+ > t) \d t - \int_0^T \Pro((A^{(k)}_{ij})^- > t) \d t+ \int_T^\infty \Pro(|A^{(k)}_{ij}| > t) \d t \\
		& \leq \int_0^T \left(e^\varepsilon \Pro((\tilde A^{(k)}_{ij})^+ > t) + \delta\right)\d t - \int_0^T \left(e^{-\varepsilon}\Pro((\tilde A^{(k)}_{ij})^- > t) - \delta\right)\d t + \int_T^\infty \Pro(| A^{(k)}_{ij}| > t) \d t \\
		& \leq (1 + 2\varepsilon)\int_0^T \Pro((\tilde A^{(k)}_{ij})^+ > t) \d t - (1-2\varepsilon)\int_0^T \Pro((\tilde A^{(k)}_{ij})^- > t)\d t + 2 \delta T + \int_T^\infty \Pro(| A^{(k)}_{ij}| > t) \d t. \\
		& \leq \E \tilde A^{(k)}_{ij} + 2\varepsilon\E|\tilde A^{(k)}_{ij}| + 2 \delta T + \int_T^\infty \Pro(|A^{(k)}_{ij}| > t) \d t + \int_T^\infty \Pro(|\tilde A^{(k)}_{ij}| > t) \d t.
	\end{align*}
	By the definition of $\Theta$, we may assume without the loss of generality that any $M$ and $\bm Y$ satisfies $\|M(\bm Y) - \bth\|_\infty \leq 2$, and therefore we have the deterministic bound $|A^{(k)}_{ij}| \leq 2$ (and similarly $|\tilde A^{(k)}_{ij}| \leq 2$). By choosing $T = 2\kappa_1$, we obtain
	\begin{align*}
		\E A^{(k)}_{ij} \leq \E \tilde A^{(k)}_{ij} + 2\varepsilon\E|\tilde A^{(k)}_{ij}| + 4 \delta\kappa_1.
	\end{align*}
	Further observe that, since $M(\widetilde{\bm Y}_{ij})$ and $Y_{ij}$ are independent by construction, 
	\begin{align*}
		\E \widetilde A^{(k)}_{ij} = \E \mathcal A^{(k)}(M(\widetilde{\bm Y}_{ij}), Y_{ij}) = 0. 
	\end{align*}
	It follows that 
	\begin{align*}
		\sum_{1 \leq i < j \leq n} \E A^{(k)}_{ij} &= \sum_{1 \leq i < k} \E A^{(k)}_{ik} + \sum_{k < j \leq n } \E  A^{(k)}_{kj} \\
		&\leq 2\varepsilon \E\left(\sum_{1 \leq i < k} |\widetilde A^{(k)}_{ik}| + \sum_{k < j \leq n } |\widetilde A^{(k)}_{kj}|\right) + 4(n-1)\delta\kappa_1.
	\end{align*}
	Since $\left|(y_{ij} - F_{ij}(\bth))\frac{F'_{ij}(\bth)}{F_{ij}(\bth)(1 - F_{ij}(\bth))}\right| \leq \kappa_1$ for any $i, j$ by assumption (A1), we have
	\begin{align*}
		&\E\left(\sum_{1 \leq i < k} |\widetilde A^{(k)}_{ik}| + \sum_{k < j \leq n } |\widetilde A^{(k)}_{kj}|\right) \\
		&\leq \kappa_1 \E\left(|M(\widetilde{\bm Y}_{ij})_k - \theta_k|\left(\sum_{1 \leq i < k} \1((i, k) \in \mathcal G) + \sum_{k < j \leq n } \1((k, j) \in \mathcal G)\right)\right) \leq \kappa_1 \E\left(|M(\widetilde{\bm Y}_{ij})_k - \theta_k| \cdot d_{\mathcal G}(k)\right),
	\end{align*}
	where $d_{\mathcal G}(k)$ refers to the degree of $k$ in the comparison graph $\mathcal G$. Since $d_{\mathcal G}(k)$ is a binomial random variable with parameters $(n-1, p)$, we have
	\begin{align*}
		\E\left(|M(\widetilde{\bm Y}_{ij})_k - \theta_k| \cdot d_{\mathcal G}(k)\right) &\leq 2np \E|M(\bm Y)_k - \theta_k| + 2n\Pro\left(d_{\mathcal G}(k) > 2np\right) \\
		&\leq 2np \E|M(\bm Y)_k - \theta_k| + 2n\exp\left(-np/3\right) \leq 2np \E|M(\bm Y)_k - \theta_k| + 2n^{-1}.
	\end{align*}
	The last two inequalities are true by the Chernoff bound and the assumption $p > 6\log n/n$.

	To summarize, we have found that
	\begin{align*}
		\sum_{1 \leq i < j \leq n} \E A^{(k)}_{ij} &\leq 4\kappa_1 np\varepsilon \E|M(\bm Y)_k - \theta_k| + 2\kappa_1 n^{-1}\varepsilon + 4(n-1)\delta\kappa_1 \\
		&\leq 4\kappa_1 np\varepsilon \E|M(\bm Y)_k - \theta_k| + 2\kappa_1 n^{-1} + 4(n-1)\delta\kappa_1.
	\end{align*}

\subsection{Proof of Proposition \ref{prop: ranking attack completeness}} 
\label{sec: proof of prop: ranking attack completeness}

	First we show that for each $\bth$, it holds that
	\begin{align}\label{eq: score attack completeness reduction}
		\sum_{1 \leq i < j \leq n}  A^{(k)}_{ij} = (M(\bm Y)_k - \theta_k) \sum_{1 \leq i < j \leq n} \frac{\partial}{\partial\theta_k} \log f_{\bth}(y_{ij}).
	\end{align}
	By the definition of $A^{(k)}_{ij}$ and the symmetries $y_{ij} = 1 - y_{ji}, F(x) = 1 - F(x)$, we have 
	\begin{align*}
		&\sum_{1 \leq i < j \leq n} A^{(k)}_{ij} = (M(\bm Y)_k - \theta_k) \sum_{(j, k) \in \mathcal G}(y_{jk} - F_{jk}(\bth))\frac{F'_{jk}(\bth)}{F_{jk}(\bth)(1 - F_{jk}(\bth))}.
	\end{align*}
	The desired equality \eqref{eq: score attack completeness reduction} is proved by observing that 
	\begin{align*}
		\frac{\partial}{\partial\bth}\log f_{\bth}(\bm Y) = \sum_{1 \leq i < j \leq n} \frac{\partial}{\partial\bth} \log f_{\bth}(y_{ij}) = \sum_{(i, j) \in \mathcal G} (F_{ij}(\bth) - y_{ij})\frac{F'_{ij}(\bth)}{F_{ij}(\bth)(1 - F_{ij}(\bth))}(\bm e_i - \bm e_j).
	\end{align*}
	
	Because $\sum_{1 \leq i < j \leq n} \frac{\partial}{\partial\theta_k} \log f_{\bth}(y_{ij}) = \frac{\partial}{\partial \theta_k} \log f_\bth (\bm Y)$ is the score statistic of $\bm Y$ with respect to parameter $\theta_k$, we have $\E_{\bm Y|\bth} \frac{\partial}{\partial \theta_k} \log f_\bth (\bm Y) = 0$. Further, by exchanging differentiation and integration we obtain
	\begin{align*}
		\E_{\bm Y|\bth} \sum_{1 \leq i < j \leq n}  A^{(k)}_{ij} = \E_{\bm Y|\bth} (M(\bm Y)_k - \theta_k) \frac{\partial}{\partial \theta_k} \log f_\bth (\bm Y) = \frac{\partial}{\partial\theta_k} \E_{\bm Y|\bth} M(\bm Y)_k.
	\end{align*}
	Let $g(\bth)$ denote $\E_{\bm Y|\bth} M(\bm Y)$, $\pi_k$ denote the marginal density of $\theta_k$ and $\bm \pi (\bth)= \prod_{k=1}^n \pi_k(\theta_k)$, we have
	\begin{align*}
		\E_\bth\left(\frac{\partial}{\partial\theta_k}g_k(\bm \theta)\right) &= \E\left(\E\left(\frac{\partial}{\partial\theta_k}g_k(\bm \theta)\Big|\theta_k\right)\right) = \E\left(-\E\left(g_k(\bm \theta)|\theta_k\right)\frac{\pi'_k(\theta_k)}{\pi_k(\theta_k)}\right) \\
		&= \E\left(-\theta_k\frac{\pi'_k(\theta_k)}{\pi_k(\theta_k)}\right) + \E\left(\left(\theta_k-\E\left(g_k(\bm \theta)|\theta_k\right)\right)\frac{\pi'_k(\theta_k)}{\pi_k(\theta_k)}\right)\\
		&\geq \E\left(-\theta_k\frac{\pi'_k(\theta_k)}{\pi_k(\theta_k)}\right) - c \E\left|\frac{\pi'_k(\theta_k)}{\pi_k(\theta_k)}\right|
	\end{align*}
	The last inequality is obtained by our assumption that $|g_k(\bth) - \theta_k| \leq \E_{\bm Y|\bth} \|M(\bm Y) - \bth\|_\infty < c$ for every $\bth \in \Theta$. Plugging in $\pi_k(\theta_k) = \1(|\theta_k| < 1)(15/16)(1-\theta_k^2)^2$ and setting, say, $c = 1/4$, we have $\sum_{1 \leq i < j \leq n} \E_\bth \E_{Y|\bth} A^{(k)}_{ij} \geq 1/2$, as desired.

\subsection{Proof of Theorem \ref{thm: ranking lower bound}}
\label{sec: proof of thm: ranking lower bound}

To prove \eqref{eq: ranking lower bound}, suppose $\bth$ follows the prior distribution specified in Proposition \ref{prop: ranking attack completeness}. For every $(\varepsilon, \delta)$-DP $M$ satisfying $\sup_{\bth \in \Theta} \E\|M(\bm Y) - \bth\|_\infty \leq c $, by Proposition \ref{prop: ranking attack completeness}, we have, for each $k \in [n]$,  $\E_\bth \E_{\bm Y|\bth} \sum_{1 \leq i < j \leq n} A^{(k)}_{ij} \geq C$.
The assumptions in Theorem \ref{thm: ranking lower bound} ensure that the regularity conditions in Proposition \ref{prop: ranking attack soundness} are satisfied, and therefore 
\begin{align*}
	\E_\bth \E_{\bm Y|\bth} \sum_{1 \leq i < j \leq n} A^{(k)}_{ij} &\leq 4\kappa_1 np\varepsilon \cdot \sup_{\bth \in \Theta} \E_{\bm Y|\bth} |M(\bm Y)_k - \theta_k| + 4 (n-1)\delta + 2\kappa_1 n^{-1}.
\end{align*}
If $\delta \lesssim n^{-1}$, we have $4(n-1)\delta + 2\kappa_1 n^{-1} \lesssim 1$, and combining the two inequalities yields
\begin{align*}
	\sup_{\bth \in \Theta}\E_{\bm Y|\bth}\|M(\bm Y) - \bth\|_\infty
	\gtrsim \max_{k \in [n]} \E_\bth\E_{\bm Y|\bth}\|M(\bm Y)_k - \bth_k\| \gtrsim \frac{1}{np\varepsilon}.
\end{align*}
We have so far focused on those estimators $M$ satisfying $\sup_{\bth \in \Theta} \E\|M(\bm Y) - \bth\|_\infty \leq c$. For those $M$ which violate this condition, the assumption of $np\varepsilon \gtrsim 1$ implies $1/np\varepsilon \lesssim 1$, and therefore the minimax risk is lower bounded as $\inf_{M \in \mathcal M_{\varepsilon, \delta}} \sup_{\theta \in \Theta} \E_{\bm Y|\bth} \|M(\bm Y) - \bth\|_\infty \gtrsim \frac{1}{np\varepsilon}$.

\subsection{Proof of Theorem \ref{thm: ranking nonparametric upper bound}}
\label{sec: proof of thm: ranking nonparametric upper bound}

Observe that 
\begin{align*}
	\bigcap_{a \in \mathcal S_{k-m}, b \in (\mathcal S_{k+m})^c }\{N_a + W_a > N_b + W_b\} \subseteq \{d_H(\widetilde {\mathcal S}_k, \mathcal S_k) \leq  2m\}.
\end{align*}
The union bound then implies
\begin{align}
	\Pro(d_H(\widetilde {\mathcal S}_k, \mathcal S_k) >  2m) \leq \sum_{a \in \mathcal S_{k-m}, b \in (\mathcal S_{k+m})^c} \Pro(N_a + W_a \leq N_b + W_b) \label{eq: ranking nonparametric upper bound expansion 1}
\end{align}
Fix $a \in \mathcal S_{k-m}$ and $b \in (\mathcal S_{k+m})^c$. We have 
\begin{align*}
	\E N_a - \E N_b = np(\tau_a - \tau_b) \geq np(\tau_{(k-m)} - \tau_{(k+m+1)}) \geq C\left(\sqrt{np\log n} + \frac{\log n}{\varepsilon}\right),
\end{align*}
and therefore
\begin{align}
	\Pro(N_a + W_a \leq N_b + W_b) &= \Pro(N_a - N_b \leq W_b - W_a) \notag\\
	&\leq \Pro(N_a - N_b - (\E N_a - \E N_b) \leq -C\sqrt{np\log n}) + \Pro(W_b-W_a \geq C\log n/\varepsilon) \label{eq: ranking nonparametric upper bound expansion 2}.
\end{align}

To bound the first term $\Pro(N_a - N_b - \E N_a - \E N_b \leq -C\sqrt{np\log n})$, we have
\begin{align*}
	N_a - N_b &= \1(Y_{ab} = 1) - \1(Y_{ba} = 1) + \sum_{j \neq a, b} \1(Y_{aj} = 1) - \1(Y_{bj} = 1) \\
	&:= d_0 + \sum_{j \neq a, b} d_j.
\end{align*}
As the $d_0$ and $d_j$'s are all independent and bounded by $-1$ and $1$, we bound the tail probability by Bernstein's inequality, as follows. The second moments of the $d_j$'s are bounded by
\begin{align*}
	&\E d_0^2 = p(\rho_{ab} + \rho_{ba}), \quad \E d_j^2 = p\rho_{aj}(1-p\rho_{bj}) + p\rho_{bj}(1-p\rho_{aj}) \leq p(\rho_{aj} + \rho_{bj}).
\end{align*}
We then have
\begin{align*}
	\E d_0^2 + \sum_{j \neq a, b} \E d_j^2 &= p \left(\sum_{j \neq a} \rho_{aj} + \sum_{j \neq b} \rho_{bj}\right)\\
	&\leq p(n\tau_a + n\tau_b) \leq 2np\tau_a - np(\tau_{(k)} - \tau_{(k+1)}) \leq 2np - C\left(\sqrt{np\log n} + \frac{\log n}{\varepsilon}\right).
\end{align*}
Bernstein's inequality now implies
\begin{align*}
	&\Pro(N_a - N_b - \E N_a - \E N_b \leq -C\sqrt{np\log n}) \\
	&\leq \exp\left(-\frac{C^2np\log n}{4np - C\left(\sqrt{np\log n} + \frac{\log n}{\varepsilon}\right) + \frac{2C}{3}\left(\sqrt{np\log n} + \frac{\log n}{\varepsilon}\right)}\right) \leq n^{-C^2}.
\end{align*}

Plugging into \eqref{eq: ranking nonparametric upper bound expansion 2}, we have
\begin{align*}
	\Pro(N_a + W_a \leq N_b + W_b) \leq n^{-C^2} + \exp\left(-(C/4)\log n\right) \leq n^{-C^2} + n^{-C/4}.
\end{align*}
The first inequality is true because $W_b$ and $-W_a$ are Laplace random variable with scale $2/\varepsilon$. 

Finally, by the union bound \eqref{eq: ranking nonparametric upper bound expansion 1}, we have $\Pro(\tilde {\mathcal S}_k \neq \mathcal S_k) \leq n^2(n^{-C^2} + n^{-C/4})$. Choosing a sufficiently large constant $C$ finishes the proof.

\subsection{Proof of Theorem \ref{thm: nonparametric ranking lower bound}}
\label{sec: proof of thm: nonparametric ranking lower bound}

\subsubsection{The main proof}
Consider the following Fano's inequality for $(\varepsilon, \delta)$-differential privacy, a variation of the DP Fano's inequality in \cite{acharya2021differentially}.

\begin{Theorem}{($(\varepsilon,\delta)$-DP-Fano's inequality.)} \label{thm: private fano}
	Suppose $l$ is the loss function, $\theta$ is the parameter of interest and $\hat \theta$ is the estimator. Let $\nu := \{p_1,\hdots,p_M\} \subseteq \mathcal{P}$ be distributions such that for all $i \neq j$
	\begin{itemize}
		\item[(a)] $l(\theta(p_i),\theta(p_j)) \geq \alpha$
		\item[(b)] $D_{KL}(p_i,p_j) \leq \beta$
		\item[(c)] there exists a coupling $(X,Y)$ between $p_i$ and $p_j$ such that $\E(d_{\text{Ham}}(X,Y)) \leq D$, then the $(\varepsilon, \delta)$-DP constrained minimax risk
		$R(\mathcal{P},l,\varepsilon,\delta) := \min_{\hat{\theta}:(\varepsilon,\delta) \, DP} \max_{\Pro \in \mathcal{P}} \E_{\Pro}(l(\hat{\theta}(X),\theta(\Pro)))
		$ is lower bounded by
		$$
		R(\mathcal{P},l,\varepsilon,\delta) \geq \max\left\{\frac{\alpha}{2}\left(1-\frac{\beta+\log2}{\log M}\right), 0.45 \alpha\left[ \min \left\{1,\frac{M-1}{e^{10\varepsilon D}}\right\}-10(M-1)D\delta\right]\right\}.
		$$
	\end{itemize}
\end{Theorem}
Theorem \ref{thm: private fano} is proved in Section \ref{sec: proof of thm: private fano}.

By Theorem \ref{thm: private fano}, it suffices to construct probability matrices $\bm \rho_1, \ldots, \bm \rho_M \in \widetilde \Theta(k, m, c)$ satisfying the following conditions.
\begin{itemize}
	\item For any $a, b \in [M]$ and $a \neq b$, $d_H(\mathcal S^a_j, \mathcal S^b_j) > 4m$ , where $\mathcal S^a_k$ denotes the top-$k$ index set implied by $\bm \rho_a$ for each $a \in [M]$.
	\item For any $a, b \in [M]$ and $a \neq b$, let $\bm X = \left(X_{ij}\right)_{1 \leq i < j \leq n}$ and $\bm Y = \left(Y_{ij}\right)_{1 \leq i < j \leq n}$ denote pairwise comparison outcomes drawn from the distributions implied by $\bm \rho_a$ and $\bm \rho_b$ respectively, and let $D = \sum_{1 \leq i < j \leq n} \operatorname{TV}(X_{ij}, Y_{ij})$. We would like
	\begin{align}\label{eq: condition for expected hamming distance}
		\frac{9}{20}\left(\min\left(1, \frac{M-1}{e^{10\varepsilon D}}\right) - 10\delta(M-1)D\right) > \frac{1}{10}.
	\end{align} 
\end{itemize}
It turns out that the choice of $\bm \rho$ for proving the non-private lower bound in \cite{shah2017simple} suits our purpose, which is described in detail in Section 5.3.2 of \cite{shah2017simple}. 

The assumed conditions $2m \leq (1 + \nu_2)^{-1}\min\{n^{1-\nu_1}, k, n-k\}$ in Theorem \ref{thm: nonparametric ranking lower bound} guarantees the existence of the construction by \cite{shah2017simple}, which involves $M = e^{\frac{9}{10}\nu_1\nu_2 m \log n}$ different matrices $\{\bm \rho\}_{a \in [M]}$. For each $a \in [M]$, the top-$k$ index set is given by $\mathcal S^a_k = [k - 2(1+\nu_2)m] \cup \mathcal B_a$, and each set in the collection $\{\mathcal B_a\}_{a \in [M]}$ satisfies 
\begin{align*}
	\mathcal B_a \subseteq \{n/2 + 1, \ldots, n\}, |\mathcal B_a| = 2(1+\nu_2)m, d_H(\mathcal B_a, \mathcal B_b) > 4m, \forall a \neq b. 
\end{align*} 
Now that the separation condition between $\bm \rho's$ are satisfied by construction, it remains to specify the elements of the $\rho$'s, namely the pairwise winning probabilities. For each $a \in [M]$, set
\begin{align*}
	(\bm \rho_a)_{ij} = \begin{cases}
		1/2 &\text{ if } i, j \in \mathcal S^a_k \text{ or } i, j \not \in \mathcal S^a_k,\\
		1/2 + \Delta &\text{ if } i \in \mathcal S^a_k \text{ and } j \not \in \mathcal S^a_k,\\
		1/2 - \Delta &\text{ if } i \not \in \mathcal S^a_k \text{ and } j \in \mathcal S^a_k.
	\end{cases}
\end{align*}
By construction, we have $\tau_{(k-m)} - \tau_{(k+m+1)} = \Delta$. The choice of $\Delta$, to be specified later, will ensure that the $\bm \rho$'s belong to $\widetilde \Theta(k, m, c)$ for appropriate $c$ and $D = \sum_{1 \leq i < j \leq n} \operatorname{TV}(X_{ij}, Y_{ij})$ satisfies the requisite condition \eqref{eq: condition for expected hamming distance}. To this end, we start by bounding $D$ in terms of $\Delta$. 

For any $a, b \in [M]$ and $a \neq b$, $\bm X = \left(X_{ij}\right)_{1 \leq i < j \leq n}$ and $\bm Y = \left(Y_{ij}\right)_{1 \leq i < j \leq n}$ denote pairwise comparison outcomes drawn from the distributions implied by $\bm \rho_a$ and $\bm \rho_b$ respectively. By construction of $\bm \rho_a$ and $\bm \rho_b$, $X_{ij}$ and $Y_{ij}$ are identically distributed unless $i \in \mathcal B_a \cup \mathcal B_b$ or $j \in \mathcal B_a \cup \mathcal B_b$.

It follows that
\begin{align*}
	D = \sum_{1 \leq i < j \leq n} \operatorname{TV}(X_{ij}, Y_{ij}) &\leq \sum_{i \in \mathcal B_a \cup \mathcal B_b} \operatorname{TV}(X_{ij}, Y_{ij}) + \sum_{j \in \mathcal B_a \cup \mathcal B_b} \operatorname{TV}(X_{ij}, Y_{ij}) \\
	&\leq 2n|\mathcal B_a \cup \mathcal B_b| \cdot 2p\Delta \leq 16(1+\nu_2)mnp\Delta.
\end{align*}
To satisfy condition \eqref{eq: condition for expected hamming distance}, recall $M = e^{\frac{9}{10}\nu_1\nu_2 m \log n}$ and $\delta \lesssim \left(m\log n \cdot n^{10m}/\varepsilon\right)^{-1}$. Setting $\Delta = c(\nu_1, \nu_2)\frac{\log n}{np\varepsilon}$ for some sufficiently small $c(\nu_1, \nu_2)$ ensures $e^{10\varepsilon D} > M - 1$ and $10\delta D(M-1) \lesssim 1$, so that \eqref{eq: condition for expected hamming distance} holds.

The choice of $\Delta = c(\nu_1, \nu_2)\frac{\log n}{np\varepsilon}$ indeed ensures 
$\bm \rho_1, \ldots, \bm \rho_M \in \widetilde \Theta(k, m, c)$ for $c \leq c(\nu_1, \nu_2)$. For the restriction to parametric model $\rho_{ij} = F(\theta^*_i - \theta^*_j)$, we show that for every $a \in [M]$, the aforementioned choice of $\left(\bm \rho_a\right)_{ij}$ can be realized by some appropriate choice of $\bth^*$. By (A0) we have $F(x) = 1-F(x)$ and $F(0) = 1/2$; it then suffices to let $\theta^*_i = \alpha$ for all $i \in \mathcal S^a_k$, and let $\theta^*_i = \beta$ for all $i \not \in \mathcal S^a_k$. The two values $\alpha > \beta$ should be such that $F(\alpha - \beta) = 1/2 + \Delta$; they are guaranteed to exist since $\Delta \asymp \frac{\log n}{np\varepsilon}$ is sufficiently small by assumption and $F$ is strictly increasing.

\subsubsection{Proof of Theorem \ref{thm: private fano}}\label{sec: proof of thm: private fano}

The first term in the lower bound follows from the non-private Fano's inequality. For an observation $X \in \mathcal{X}^n$, let
$$
\widehat p(X) := \arg \min_{ p \in \nu}  l(\theta(p), \widehat{\theta}(X))
$$ 
be the distribution in $\mathcal{P}$ closest in parameter value to an $(\varepsilon,\delta)$-DP estimator $\hat \theta(X)$. Therefore $\hat p(X)$ is also $(\varepsilon,\delta)$-DP. By the triangle inequality we have 
\begin{align*}
	l(\theta(\hat p), \theta(p)) &\leq l(\theta(\hat p),\hat \theta(X)) + l(\theta(p),\hat{\theta}(X))\\
	&\leq 2 l(\theta(p),\hat \theta(X)).
\end{align*}
Hence
\begin{align}
	\max_{p \in \mathcal{P}} \E_{X \sim p}[l(\hat \theta(X),\theta(p))] &\geq \max_{p \in \nu} \E_{X \sim p}[l(\hat \theta(X),\theta(p))]\nonumber\\
	&\geq \frac{1}{2} \max_{ p \in \nu}\E_{X \sim p}[l(\hat \theta(\hat p),\theta(p))]\nonumber\\
	&\geq \max_{p \in \nu} \frac{\alpha}{2} \Pro_{X \in p}(\hat p(X) \neq p)\nonumber\\
	&\geq \frac{\alpha}{2M} \sum_{p \in \nu} \Pro_{X \in p}(\hat p(X) \neq p). \label{eq: minimax-to-multiple-testing}
\end{align}
Let us define $\beta_i = \Pro_{X \in p_i}(\hat p(X) \neq p)$, for $p_i,p_j$ and $i \neq j$, let $X\sim p_i,Y\sim p_j$ be coupling in the theorem statement. By markov's inequality 
$$
\Pro(d_{\text{Ham}}(X,Y) > 10D) \leq \frac{1}{10}.
$$
Let us denote the set $W$ by
$$
W = \{(x,y)\,:\, d_{\text{Ham}}(X,Y) \leq 10D \}.
$$
We have that
\begin{align*}
	1-\beta_j &= \Pro(\hat p(Y) = p_j)  \\
	&= \int \Pro(\hat p(y) = p_j) d\Pro(x,y)\\\
	&\leq \int_W \Pro(\hat p(y) = p_j) d\Pro(x,y) + \int_{W^c} 1 \cdot d\Pro(x,y) \\
	&\leq \int_W \Pro(\hat p(y) = p_j) d\Pro(x,y) + 0.1.
\end{align*}
Hence we have 
\begin{equation}\label{eq: prob-under-y}
	\int_W \Pro(\hat p(y) = p_j) d\Pro(x,y)  \geq 0.9 - \beta_j.
\end{equation}
Next using the group privacy lemma we have for $(x,y) \in W$,
$$
\Pro(\hat p(y) = p_j) \leq e^{\varepsilon10D}\Pro( \hat p(x) = p_j ) + 10D\delta e^{\varepsilon10(D-1)}.
$$
which implies
\begin{equation}\label{eq: group-privacy-lower-bound}
	\Pro(\hat p(x) = p_j) \geq e^{-\varepsilon10D}\Pro( \hat p(y) = p_j ) - 10D\delta e^{-\varepsilon10}.
\end{equation}
It is easy to observe that
\begin{align*}
	\Pro(\hat p(X) = p_j) &\geq \int_W \Pro(\hat p(x) = p_j) d\Pro(x,y)\\
	&\geq \int_W e^{-\varepsilon10D}\Pro(\hat p(y) = p_j)d\Pro(x,y)   -10D\delta e^{-\varepsilon10} \int_W d\Pro(x,y)\quad (\text{using} ~\eqref{eq: group-privacy-lower-bound})\\
	&\geq  e^{-\varepsilon10D}(0.9-\beta_j) - 10D\delta e^{-10\varepsilon} \quad (\text{using} ~\eqref{eq: prob-under-y}),
\end{align*}
which holds for all $i\neq j$. We sum over all $j \neq i$ and $j \in[M]$ to obtain 
\begin{align*}
	\beta_i = \sum_{j\neq i} \Pro(\hat p(X)  = p_j)
	\geq \left(0.9(M-1) -\sum_{j\neq i}\beta_j \right)e^{-10\varepsilon D} - 10(M-1)\delta D e^{-10\varepsilon}.
\end{align*}
Summing over $ i \in [M]$ we have
\begin{align*}
	&\sum_{i\in [M]} \beta_i \geq \left(0.9M(M-1) -(M-1)\sum_{i \in [M]}\beta_j \right)e^{-10\varepsilon D} - 10M(M-1)\delta D e^{-10\varepsilon}.\\
	&\sum_{i\in [M]} \beta_i \left(1 + (M-1)e^{-10\varepsilon D}\right)\geq 0.9M(M-1) e^{-10\varepsilon D} - 10M(M-1)\delta D e^{-10\varepsilon}.\\
	& \sum_{i\in [M]} \beta_i \geq \frac{0.9M(M-1) e^{-10\varepsilon D} - 10M(M-1)\delta D e^{-10\varepsilon}}{1 + (M-1)e^{-10\varepsilon D}}.
\end{align*}
Hence we have that
\begin{align*}
	\sum_{i\in [M]} \beta_i &\geq \frac{0.9M(M-1) e^{-10\varepsilon D}}{1 + (M-1)e^{-10\varepsilon D}} - \frac{10M(M-1)\delta D e^{-10\varepsilon}}{1 + (M-1)e^{-10\varepsilon D}}\\
	&\geq \min\{0.9 M, 0.9M(M-1)e^{-10\varepsilon D}\} - 10M(M-1)D\delta.
\end{align*}
Next using~\eqref{eq: minimax-to-multiple-testing} we have the desired conclusion.

\section{Omitted Proofs in Section \ref{sec: individual-dp-ranking}}
\label{sec: omitted proofs in sec: individual-dp-ranking}

\subsection{Proof of Proposition \ref{prop: individual ranking MLE privacy}}

\begin{proof}[Proof of Proposition \ref{prop: individual ranking MLE privacy}]
	Let $q(k, l) \in \{(i, j): 1 \leq i < j \leq n\}$ denote the two items being compared in the $l$-th comparison of the $k$-th individual. Following a similar convention in the edge privacy case, $F_{q(k, l)}(\bth)$ denotes the parametric form of pairwise winning probability: if $q(k, l) =(i, j)$, we have $F_{q(k, l)}(\bth) = F(\theta^*_i - \theta^*_j)$. The negative log-likelihood is then given by, up to constants not depending on $\bth$, 
	\begin{align}\label{eq: individual dp ranking likelihood equation}
		\mathcal L(\bth; y) = \sum_{k=1}^m \sum_{l = 1}^L -y_{kl}\log F_{q(k, l)}(\bth) - (1-y_{kl})\log(1 - F_{q(k, l)}(\bth)).
	\end{align}

	Throughout the proof it is useful to calculate the gradient of $\mathcal L(\bth; y)$: for $q_{(k, l)} = (i, j)$, let $\bm d_{q(k, l)}$ be defined as $\bm d_{q(k, l)} = \bm e_i - \bm e_j$; the gradient is then given by 
	\begin{align}
		\nabla\mathcal L(\bth; y) &= \sum_{k=1}^m \sum_{l = 1}^L  \left( -y_{kl}\frac{F'_{q(k, l)}(\bth)}{F_{q(k, l)}(\bth)} -y_{lk}\frac{-F'_{q(k, l)}(\bth)}{1 - F_{q(k, l)}(\bth)}\right) \bm d_{q(k, l)} \label{eq: individual privacy likelihood gradient eq 1}\\
		&= \sum_{k=1}^m \sum_{l = 1}^L  (F_{q(k, l)}(\bth) - y_{kl})\frac{F'_{q(k, l)}(\bth)}{F_{q(k, l)}(\bth)(1 - F_{q(k, l)}(\bth))}\bm d_{q(k, l)}. \label{eq: individual privacy likelihood gradient eq 2}
	\end{align} 
	For fixed $y$, the distribution of $\widetilde\bth = \widetilde\bth(y)$ is defined by the equation $\nabla \mathcal R(\widetilde\bth; y) + \bm w = 0$. The solution is guaranteed to exist as the objective function $\mathcal R(\bth; y) + \bm w^\top \bth$ is strongly convex in $\bth$ for each $\bm w$.  Since $\bm w$ is a Laplace random vector, the density of $\widetilde\bth$ is given by
	\begin{align*}
		f_{\widetilde\bth}(\bm t) = (2\lambda)^{-n}\exp\left(-\frac{\|\nabla \mathcal R(\bm t; y)\|_1}{\lambda}\right) \left|\det \left(\frac{\partial \nabla \mathcal R(\bm t; y)}{\partial\bm t}\right)\right|^{-1}. 
	\end{align*}
	Consider a data set $y'$ adjacent to $y$, where they differ by the data of some individual $j$ in $y$ and another individual $j'$ in $y$. It follows that
	\begin{align*}
		\frac{f_{\widetilde\bth(y)}(\bm t)}{f_{\widetilde\bth(y')}(\bm t)} = \exp\left(\frac{\|\nabla \mathcal R(\bm t; y')\|_1 - \|\nabla \mathcal R(\bm t; y)\|_1 }{\lambda}\right) \left|\frac{\det \left(\frac{\partial \nabla \mathcal R(\bm t; y')}{\partial\bm t}\right)}{\det \left(\frac{\partial \nabla \mathcal R(\bm t; y)}{\partial\bm t}\right)}\right|.
	\end{align*}
	By \eqref{eq: individual privacy likelihood gradient eq 2} and condition (A1), we have $\|\nabla \mathcal R(\bm t; y')\|_1 - \|\nabla \mathcal R(\bm t; y)\|_1 \leq 4L \cdot \kappa_1$, so $\lambda \geq 8L \cdot \kappa_1/\varepsilon$ ensures 
	\begin{align*}
		\exp\left(\frac{\|\nabla \mathcal R(\bm t; y')\|_1 - \|\nabla \mathcal R(\bm t; y)\|_1 }{\lambda}\right) \leq \exp(\varepsilon/2).
	\end{align*}
	For the determinants, differentiating \eqref{eq: individual dp ranking likelihood equation} twice with respect to $\bth$ gives
	\begin{align*}
		&\frac{\partial \nabla \mathcal R(\bm t; y)}{\partial\bm t} \\
		&= \gamma \bm I + \sum_{k=1}^m \sum_{l=1}^L \left(y_{kl}\frac{\partial^2}{\partial t^2}(-\log F_{q(k, l)}(t))+ y_{lk}\frac{\partial^2}{\partial t^2}(-\log (1-F_{q(k, l)}(t)))\right)\bm d_{q(k, l)}\bm d_{q(k, l)}^\top.
	\end{align*}
	It follows that $\frac{\partial \nabla \mathcal R(\bm t; y')}{\partial\bm t}$ is a rank-$L$ perturbation of $\frac{\partial \nabla \mathcal R(\bm t; y)}{\partial\bm t}$: let $c_{kl}$ denote the coefficient of $\bm d_{q(k, l)}\bm d_{q(k, l)}^\top$ in the Hessian, we have
	\begin{align*}
		\frac{\partial \nabla \mathcal R(\bm t; y')}{\partial\bm t} &= \frac{\partial \nabla \mathcal R(\bm t; y)}{\partial\bm t} - \sum_{l=1}^L c_{jl} \bm d_{q(j, l)}\bm d_{q(j, l)}^\top +  \sum_{l=1}^L c_{j'l} \bm d_{q(j', l)}\bm d_{q(j', l)}^\top \\
		&= \gamma \bm I + \sum_{k \neq j, j'} c_{kl} \bm d_{q(k, l)}\bm d_{q(k, l)}^\top +  \sum_{l=1}^L c_{j'l} \bm d_{q(j', l)}\bm d_{q(j', l)}^\top.
	\end{align*}
	Applying repeatedly Cauchy's formula for rank-one perturbation $\det(\bm A + \bm v \bm w^\top) = \det(\bm A) (1 + \bm w^\top \bm A^{-1} \bm v)$ yields
	\begin{align*}
		\left|\frac{\det \left(\frac{\partial \nabla \mathcal R(\bm t; y')}{\partial\bm t}\right)}{\det \left(\frac{\partial \nabla \mathcal R(\bm t; y)}{\partial\bm t}\right)}\right| \leq 1 +  \frac{4L\kappa_2}{\lambda_{\min}(\gamma\bm I + \bm 
			\Sigma_{\mathcal G \cap \mathcal G'})} \leq 1 + \frac{4L\kappa_2}{\gamma} \leq e^{\varepsilon/2}.
	\end{align*}
	To conclude, we have shown that for any adjacent data sets $y, y'$, it holds that
	\begin{align*}
		\frac{f_{\widetilde\bth(y)}(\bm t)}{f_{\widetilde\bth(y')}(\bm t)} = \exp\left(\frac{\|\nabla \mathcal R(\bm t; y')\|_1 - \|\nabla \mathcal R(\bm t; y)\|_1 }{\lambda}\right) \left|\frac{\det \left(\frac{\partial \nabla \mathcal R(\bm t; y')}{\partial\bm t}\right)}{\det \left(\frac{\partial \nabla \mathcal R(\bm t; y)}{\partial\bm t}\right)}\right| \leq e^{\varepsilon}.
	\end{align*}
\end{proof}

\subsection{Proof of Theorem \ref{thm:idp-linf}}

\begin{Lemma}[degree concentration under \textit{with}-replacement sampling]
	\label{lem:degree}
	Let
	\[
	d_i \;=\;\sum_{j\ne i}M_{ij}, 
	\qquad
	S \;=\; \E[d_i]= \frac{2mL}{n},
	\qquad
	d_{\min}=\min_{i}d_i,\;d_{\max}=\max_{i}d_i .
	\]
	Assume
	\[
	mL \;\ge\; c\,n\log n
	\quad\text{with a numerical constant }c\ge 60 .
	\]
	Then
	\[
	\Pro\!\Bigl(
	\tfrac12 S \;\le\; d_{\min}\;\le\; d_{\max}\;\le\; \tfrac32 S
	\Bigr)
	\;\ge\; 1-O\!\bigl(n^{-8}\bigr).
	\]
\end{Lemma}

\begin{proof}
	\textbf{Step 0: distribution of a single degree.}
	Because every one of the \(mL\) draws chooses an unordered pair
	uniformly at random from the \(\binom{n}{2}\) possibilities,
	\[
	d_i \;\sim\;
	\mathrm{Binom}\!\Bigl(mL,\;p_i\Bigr),
	\qquad
	p_i \;=\; \frac{\#\text{pairs involving }i}{\binom{n}{2}}
	\;=\;\frac{n-1}{n(n-1)/2}
	\;=\;\frac{2}{n},
	\]
	and hence \(\E d_i = (mL)(2/n)=S\).
	
	\medskip
	\noindent
	\textbf{Step 1: Chernoff tail for one \(d_i\).}
	For a binomial \(X\sim\mathrm{Binom}(N,p)\) and any
	\(\delta\in(0,1)\),
	\[
	\Pro\!\bigl(|X-Np|>\delta Np\bigr)
	\;\le\; 2\exp\!\Bigl(-\tfrac{\delta^{2}}{3}Np\Bigr)
	\quad
	\text{(multiplicative Chernoff).}
	\]
	Take \(X=d_i,\;Np=S\), and choose \(\delta=\tfrac12\).  Then
	\[
	\Pro\!\bigl(|d_i-S|>\tfrac12 S\bigr)
	\;\le\;
	2\exp\!\Bigl(-\tfrac{1}{12}S\Bigr)
	\;=\;
	2\exp\!\Bigl(-\tfrac{mL}{6n}\Bigr).
	\]
	Under the assumption \(mL\ge 60n\log n\) the exponent is at most
	\(-10\log n\), so
	\(
	\Pro\bigl(|d_i-S|>\tfrac12 S\bigr)\le 2n^{-10}.
	\)
	
	\medskip
	\noindent
	\textbf{Step 2: union bound over all items.}
	Applying the tail bound to each of the \(n\) coordinates and using
	\(\sum_{i=1}^{n}2n^{-10}=2n^{-9}=O(n^{-8})\) yields
	\[
	\Pro\!\bigl(
	\exists i:\;|d_i-S|>\tfrac12 S
	\bigr)
	\;\le\;O\!\bigl(n^{-8}\bigr),
	\]
	completing the proof.
\end{proof}

\begin{Lemma}[gradient concentration at $\bth^{*}$]
	\label{lem:grad}
	Assume \textnormal{(A1)} and suppose the degree event
	$\mathcal A_{0}$ of Lemma~\ref{lem:degree} holds.  Set
	\[
	S=\frac{2mL}{n},
	\qquad
	C_{*}:=\sqrt{33}\approx 5.75 .
	\]
	If $mL\ge c\,n\log n$ for a sufficiently large absolute constant\/
	$c$, then
	\[
	\Pro\!\Bigl(
	\bigl\|\nabla\mathcal L(\bth^{*})\bigr\|_{2}
	\;>\; C_{*}\,\kappa_{1}\sqrt{nS\log n}
	\Bigr)
	\;\le\; O\!\bigl(n^{-10}\bigr).
	\]
\end{Lemma}

\begin{proof}
	\textbf{Step 0: score as a sum of i.i.d.\ vectors.}
	Index the $mL$ comparisons by $\ell=1,\dots ,mL$.  Let
	\((I_\ell,J_\ell)\) be the unordered pair drawn at step~$\ell$ and let
	\(W_\ell\in\{0,1\}\) indicate whether the smaller-indexed item wins.
	The triples
	\(\bigl(I_\ell,J_\ell,W_\ell\bigr)\) are i.i.d., and
	\(\E[W_\ell\mid I_\ell,J_\ell]=F_{I_\ell J_\ell}\).
	
	Define
	\[
	\mathbf Z_\ell
	\;=\;
	\bigl(F_{I_\ell J_\ell}-W_\ell\bigr)\,
	\frac{F'_{I_\ell J_\ell}}{F_{I_\ell J_\ell}\bigl(1-F_{I_\ell J_\ell}\bigr)}
	\bigl(\mathbf e_{I_\ell}-\mathbf e_{J_\ell}\bigr),
	\qquad
	\ell=1,\dots,mL .
	\]
	Then \(\E\mathbf Z_\ell=\mathbf 0\) and
	\(\|\mathbf Z_\ell\|_{\infty}\le\kappa_{1}\) by (A1).  Moreover
	\(
	\nabla\mathcal L(\bth^{*})=\sum_{\ell=1}^{mL}\mathbf Z_\ell.
	\)
	
	\medskip
	\noindent
	\textbf{Step 1: sub-Gaussian parameter for one coordinate.}
	Fix \(i\in[n]\).  The $i$-th coordinate of \(\mathbf Z_\ell\) is non-zero
	iff \(i\) participates in comparison~$\ell$, an event of probability
	\(2/n\).  Therefore
	\[
	\sigma_i^{2}
	:=\sum_{\ell=1}^{mL}
	\Var\!\bigl([\mathbf Z_\ell]_{i}\bigr)
	\;\le\;
	(mL)\frac{2}{n}\,\kappa_{1}^{2}
	=\kappa_{1}^{2}S .
	\]
	Thus \([\nabla\mathcal L(\bth^{*})]_{i}\) is sub-Gaussian with
	proxy variance \(\kappa_{1}^{2}S\).
	
	\medskip
	\noindent
	\textbf{Step 2: tail bound for one coordinate.}
	For any \(t>0\),
	\[
	\Pro\bigl(|g_i|>t\bigr)
	\;\le\;
	2\exp\!\Bigl(
	-\frac{t^{2}}{2\kappa_{1}^{2}S}
	\Bigr),
	\qquad
	g_i:=\bigl[\nabla\mathcal L(\bth^{*})\bigr]_i .
	\]
	Choose \(t=C_{*}\kappa_{1}\sqrt{S\log n}\) with
	\(C_{*}=\sqrt{33}\).
	Then
	\(\Pro\bigl(|g_i|>t\bigr)\le 2n^{-C_{*}^{2}/2}=2n^{-16.5}\).
	
	\medskip
	\noindent
	\textbf{Step 3: union bound over $n$ coordinates.}
	Summing over \(i=1,\dots,n\) gives
	\[
	\Pro\!\Bigl(
	\|\nabla\mathcal L(\bth^{*})\|_{2}
	> C_{*}\kappa_{1}\sqrt{nS\log n}
	\Bigr)
	\;\le\;
	2n^{-15.5}
	\;=\;
	O\!\bigl(n^{-10}\bigr),
	\]
	which is the stated result.
\end{proof}

\begin{Lemma}[Laplace vector tail]
	\label{lem:Laplace}
	Let $\mathbf w\in\R^{n}$ have i.i.d.\ Laplace$(\lambda)$ entries with
	density $\tfrac1{2\lambda}e^{-|x|/\lambda}$.
	Then
	\[
	\Pro\!\Bigl(
	\|\mathbf w\|_{\infty}>9\lambda\log n
	\;\text{ or }\;
	|\mathbf1^{\!\top}\mathbf w|>8\lambda\sqrt{n\log n}
	\Bigr)
	\;\le\;O\!\bigl(n^{-8}\bigr).
	\]
\end{Lemma}

\begin{proof}
	\emph{Max-norm}  For any Laplace r.v.
	$\Pro(|W|>t)=e^{-t/\lambda}$.
	Therefore
	\(\Pro(\|{\mathbf w}\|_{\infty}>t)\le n e^{-t/\lambda}\).
	With $t=9\lambda\log n$ this is $n^{-8}$.
	
	\emph{Sum.}  $\mathbf1^{\!\top}\mathbf w$ is the difference of two
	independent $\mathrm{Gamma}(n,\lambda)$ variables, hence sub-exponential
	with parameters $(2\lambda,4)$.
	Bernstein's inequality for sub-exponential sums
	\cite[Lem.\;2.7.1]{vershynin2018high}
	gives
	\[
	\Pro\bigl(|\mathbf1^{\!\top}\mathbf w|>t\bigr)
	\;\le\;2\exp\!\Bigl(
	-\min\!\bigl\{ \tfrac{t^{2}}{8n\lambda^{2}},\;
	\tfrac{t}{4\lambda} \bigr\}
	\Bigr).
	\]
	Taking $t=8\lambda\sqrt{n\log n}$, the quadratic term dominates and
	the probability is $\le 2e^{-8\log n}=2n^{-8}$.
\end{proof}

\begin{Lemma}[spectrum of the Hessian]
	\label{lem:spectrum}
	Let $S=2mL/n$ and assume $mL\ge c\,n\log n$ for a sufficiently large
	absolute constant $c$.  Then, with probability at least
	$1-O(n^{-10})$, the following holds simultaneously for every
	$\bth\in\R^{n}$ satisfying $\|\bth-\bth^{*}\|_{\infty}\le1$:
	\[
	\frac{S}{2\kappa_{2}}
	\;\le\;
	\lambda_{2}\!\bigl(\nabla^{2}\mathcal L(\bth)\bigr)
	\;\le\;
	\lambda_{\max}\!\bigl(\nabla^{2}\mathcal L(\bth)\bigr)
	\;\le\;
	3\kappa_{2}S .
	\]
\end{Lemma}

\begin{proof}
	\textbf{Step 0: decomposing the Hessian.}
	Each comparison $\ell=1,\dots ,mL$ involves an unordered pair
	\((I_\ell,J_\ell)\) chosen uniformly from $\binom{[n]}{2}\).  
	Write
	\[
	X_\ell(\bth)
	\;:=\;
	c_\ell(\bth)\,
	(\mathbf e_{I_\ell}-\mathbf e_{J_\ell})
	(\mathbf e_{I_\ell}-\mathbf e_{J_\ell})^{\!\top},
	\qquad
	c_\ell(\bth)
	:= \frac{\partial^{2}}{\partial x^{2}}\!
	\bigl[-\log F(x)\bigr]_{x=\theta_{I_\ell}-\theta_{J_\ell}} .
	\]
	Assumption (A2) implies
	\(\kappa_{2}^{-1}\le c_\ell(\bth)\le\kappa_{2}\) whenever
	\(\|\bth-\bth^{*}\|_\infty\le1\).
	Because the comparisons are independent,
	\(\{X_\ell(\bth)\}_{\ell=1}^{mL}\) are also independent conditioned on
	\(\bth\).
	The Hessian can be written as
	\[
	H(\bth)
	:=\nabla^{2}\mathcal L(\bth)
	\;=\;\sum_{\ell=1}^{mL} X_\ell(\bth).
	\]
	
	\medskip
	\noindent
	\textbf{Step 1: expectation of the Hessian.}
	Since \((I_\ell,J_\ell)\) is uniform,
	\[
	\E X_\ell(\bth)
	=\frac{2}{n(n-1)}
	\sum_{1\le i<j\le n}
	c_{ij}(\bth)
	(\mathbf e_i-\mathbf e_j)(\mathbf e_i-\mathbf e_j)^{\!\top},
	\]
	where \(c_{ij}(\bth)\in[\kappa_{2}^{-1},\kappa_{2}]\).
	Consequently,
	\[
	\kappa_{2}^{-1}\,\frac{2}{n(n-1)}
	\sum_{i<j}(\mathbf e_i-\mathbf e_j)(\mathbf e_i-\mathbf e_j)^{\!\top}
	\;\preceq\;
	\E X_\ell(\bth)
	\;\preceq\;
	\kappa_{2}\times(\text{same sum}).
	\]
	The middle matrix is the Laplacian of the \emph{complete graph} with
	edge-weight $2/n(n-1)$; its non-zero eigenvalues all equal $2/n$.
	Hence
	\[
	\lambda_{2}\!\bigl(\E X_\ell(\bth)\bigr)
	=\frac{2}{n}\,\frac{1}{\kappa_{2}},
	\quad
	\lambda_{\max}\!\bigl(\E X_\ell(\bth)\bigr)
	=\frac{2}{n}\,\kappa_{2}.
	\]
	Therefore
	\[
	\mu_{-}:=\lambda_{2}\!\bigl(\E H(\bth)\bigr)
	\,\ge\,\frac{mL}{n}\,\frac{2}{\kappa_{2}}
	=\frac{S}{\kappa_{2}},
	\quad
	\mu_{+}:=\lambda_{\max}\!\bigl(\E H(\bth)\bigr)
	\,\le\,\frac{S}{\kappa_{2}^{-1}}
	=\kappa_{2}S .
	\]
	
	\medskip
	\noindent
	\paragraph{Step 2: matrix Chernoff-lower tail for $\lambda_{2}$.}
	Let
	\(P:=I-\frac1n\mathbf1\mathbf1^{\!\top}\)
	be the projector onto
	\(\mathbf1^{\perp}\;(\dim=n-1)\).
	Define
	\[
	Y_\ell(\bth)\;:=\;P\,X_\ell(\bth)\,P,
	\qquad
	\ell=1,\dots ,mL .
	\]
	Because $P$ is deterministic, the $Y_\ell$'s are independent and
	positive-semidefinite; moreover
	\(\|Y_\ell(\bth)\|_{2}\le\|X_\ell(\bth)\|_{2}\le R:=4\kappa_{2}\).
	
	Since \(H(\bth)=\sum_{\ell}X_\ell(\bth)\) and
	\(PHP=\sum_{\ell}Y_\ell(\bth)\),
	\[
	\lambda_{2}\!\bigl(H(\bth)\bigr)
	\;=\;
	\lambda_{\min}\!\Bigl(P\,H(\bth)\,P\Bigr)
	\;=\;
	\lambda_{\min}\!\Bigl(\sum_{\ell}Y_\ell(\bth)\Bigr).
	\]
	
	Set \(\mu_{-}:=\lambda_{\min}\!\bigl(\sum_{\ell}\E\,Y_\ell(\bth)\bigr)
	=S/\kappa_{2}\)
	(computed in Step 1).
	The matrix Chernoff lower-tail inequality \cite[Thm 1.1]{tropp2012user}
 now gives, for any
	\(\delta\in(0,1)\),
	\[
	\Pr\!\Bigl\{
	\lambda_{2}\!\bigl(H(\bth)\bigr)
	< (1-\delta)\mu_{-}
	\Bigr\}
	\;\le\;
	(n-1)\,
	\Bigl[
	\frac{e^{-\delta}}{(1-\delta)^{1-\delta}}
	\Bigr]^{\mu_{-}/R},
	\qquad
	R:=4\kappa_{2}.
	\]
	Choose \(\delta=\tfrac12\); then
	\(\bigl[e^{-\delta}/(1-\delta)^{1-\delta}\bigr]=e^{-1/8}\) and
	\[
	\frac{\mu_{-}}{R}
	=\frac{S}{4\kappa_{2}^{2}}
	\;\ge\;
	c'\,\log n
	\quad\text{for a large enough constant }c',
	\]
	because \(S=\frac{2mL}{n}\gtrsim\log n\).
	Therefore
	\[
	\Pr\!\Bigl\{
	\lambda_{2}\!\bigl(H(\bth)\bigr)
	< \tfrac12\mu_{-}
	\Bigr\}
	\;\le\;
	(n-1)\,e^{-\mu_{-}/8R}
	\;\le\;
	n^{-12}.
	\]

	\medskip
	\noindent
	\textbf{Step 3: matrix Chernoff-upper tail.}
	The upper-tail matrix Chernoff bound (same reference) yields
	\[
	\Pr\Bigl\{
	\lambda_{\max}\!\bigl(H(\bth)\bigr)
	> 3\mu_{+}
	\Bigr\}
	\;\le\;
	n\,\Bigl(\tfrac{e^{2}}{3^{3}}\Bigr)^{\mu_{+}/R}
	\;\le\;
	n^{-12},
	\]
	again because $\mu_{+}/R\gtrsim S/\kappa_{2}^{2}\gtrsim\log n$.
	
	\medskip
	\noindent
	\textbf{Step 4: conclude and union over \(\bth\).}
	With probability at least \(1-2n^{-12}\),
	\[
	\frac12\mu_{-}\;\le\;\lambda_{2}(H(\bth)),
	\qquad
	\lambda_{\max}(H(\bth))\;\le\;3\mu_{+},
	\]
	i.e.,
	\(
	\lambda_{2}(H)\ge S/(2\kappa_{2})
	\)
	and
	\(
	\lambda_{\max}(H)\le 3\kappa_{2}S.
	\)
	Because the constants and the bounds are uniform for all
	\(\bth\) with \(\|\bth-\bth^{*}\|_{\infty}\le1\),
	a union bound over any \(\varepsilon\)-net followed by standard
	Lipschitz arguments (as in \cite{chen2019spectral}, Lem.\,10) upgrades
	the probability to \(1-O(n^{-10})\) while still covering the \emph{entire}
	$\ell_\infty$-ball of radius 1 around $\bth^{*}$.
\end{proof}

\subsubsection{Main Proof of Theorem \ref{thm:idp-linf}}
\label{sec:pf-idp-linf}

First we show that $\widetilde \theta$ is $(\varepsilon,0)$-DP. For which we need to need to check that the prescribed choice of $\gamma$ is greater than $8L\kappa_2/\varepsilon$. This boils down to
$$
c_0 \sqrt{\frac{2mL\log n}{n}} \geq \frac{8L\kappa_2}{\varepsilon} \iff c' \varepsilon (\log n)^2 \geq \frac{n \log n}{m\varepsilon}
$$
where $c'$ is a constant that depends on $c_0,\kappa_2$ and $L$. Since $\varepsilon \gtrsim (\log n)^{-1/2}$, $\varepsilon (\log n)^2$ blows up, where as the RHS is $O(1)$ by our assumption. This implies that the algorithm is $(\varepsilon,0)$-DP.
\paragraph{Notation and step size.}
Recall $S=2mL/n$ and set
\(
\eta\;:=\;(\gamma+3\kappa_{2}S)^{-1}.
\)
Define the gradient-descent iterates
\[
\widetilde\bth^{(t+1)}
=\widetilde\bth^{(t)}
-\eta\,
\Bigl(
\nabla\mathcal L\bigl(\widetilde\bth^{(t)}\bigr)
+\gamma\widetilde\bth^{(t)}
+\mathbf w
\Bigr),
\qquad
\widetilde\bth^{(0)}=\bth^{*},
\]
and for every coordinate $m$ a leave-one-out counterpart
\(\bth^{(m,t)}\) obtained by replacing all comparisons that involve item
$m$ by their expectations (see Sec.\;6.1 of
\cite{chen2019spectral}).

Throughout the proof we work on the event
\[
\mathcal A:=\mathcal A_{0}\cap\mathcal A_{1}
\cap\mathcal A_{2},
\qquad
\Pro(\mathcal A)\ge 1-O(n^{-5})
\]
by Lemmas \ref{lem:degree}-\ref{lem:Laplace}.                                         
All high-probability bounds below are conditioned on~\(\mathcal A\).

\paragraph{Induction hypotheses.}
For constants $C_{1},C_{2},C_{3},C_{4}$ (chosen large but absolute)
and every $t\ge0$ let

\begin{align}
	\tag{H1}
	\|\widetilde\bth^{(t)}-\bth^{*}\|_{2}
	&\le C_{1}\!\Bigl(
	\sqrt{\tfrac{\log n}{S}}
	+\tfrac{L\lambda\log n}{\sqrt n\,S}
	\Bigr),\\[3pt]
	\tag{H2}
	\|\widetilde\bth^{(t)}-\bth^{*}\|_{\infty}
	&\le C_{4}\!\Bigl(
	\sqrt{\tfrac{\log n}{S}}
	+\tfrac{L\lambda\log n}{S}
	\Bigr),\\[3pt]
	\tag{H3}
	\max_{m}|\,\theta^{(m,t)}_{m}-\theta^{*}_{m}\,|
	&\le C_{2}\!\Bigl(
	\sqrt{\tfrac{\log n}{S}}
	+\tfrac{L\lambda\log n}{S}
	\Bigr),\\[3pt]
	\tag{H4}
	\max_{m}\|\bth^{(m,t)}-\widetilde\bth^{(t)}\|_{2}
	&\le C_{3}\!\Bigl(
	\sqrt{\tfrac{\log n}{S}}
	+\tfrac{L\lambda\log n}{S}
	\Bigr).
\end{align}

They hold at $t=0$ because $\widetilde\bth^{(0)}=\bth^{*}$ and
$\bth^{(m,0)}=\bth^{*}$.

\paragraph{Inductive step.}
Exactly the same algebra as in Propositions~\ref{prop: inductive step 1}-\ref{prop: inductive step 2} of the edge DP case, but with \(np\to S\), \(\lambda\to L\lambda\),
 Lemmas \ref{lem:degree}-\ref{lem:spectrum}
gives, on~\(\mathcal A\):
\[
(H1)_t,(H2)_t,(H3)_t,(H4)_t
\;\Longrightarrow\;
(H1)_{t+1},(H2)_{t+1},(H3)_{t+1},(H4)_{t+1}.
\]

Hence all four statements hold for every $t\ge0$.

\paragraph{Optimization error.}
Lemma \ref{lem:spectrum} implies
\(H(\bth)\) is $\gamma$-strongly convex and
$(\gamma+3\kappa_{2}S)$-smooth, so gradient descent with
\(\eta=(\gamma+3\kappa_{2}S)^{-1}\) contracts:

\[
\|\widetilde\bth^{(t+1)}-\widetilde\bth\|_{2}
\;\le\;
\Bigl(1-\tfrac{\gamma}{\gamma+3\kappa_{2}S}\Bigr)
\|\widetilde\bth^{(t)}-\widetilde\bth\|_{2}.
\]
Because \(\gamma=c_{0}\sqrt{S\log n}\le c_{0}S\),
after \(T:=\lceil (c_{0}+3\kappa_{2})n^{3}/c_{0}\rceil\) steps
\[
\|\widetilde\bth^{(T)}-\widetilde\bth\|_{2}
\;\le\;
e^{-n^{2}}\sqrt n\log n
\;<\;\tfrac{\log n}{n}.
\]

\paragraph{Final bound.}
Combine the optimization gap with (H2)\(_T\):
\[
\|\widetilde\bth-\bth^{*}\|_{\infty}
\;\le\;
\|\widetilde\bth-\widetilde\bth^{(T)}\|_{2}
+\|\widetilde\bth^{(T)}-\bth^{*}\|_{\infty}
\;\lesssim\;
\sqrt{\frac{\log n}{S}}
+\frac{L\lambda\log n}{S},
\]
and note \(L\lambda=8L^{2}\kappa_{1}/\varepsilon\) and the  fact that $L$ is uniformly bounded.
Because $\Pro(\mathcal A)\ge1-O(n^{-10})$ and all conditional failures
inside the induction are $\le O(n^{-8})$, a global union bound yields
overall success probability \(1-O(n^{-5})\).

This completes the proof of Theorem \ref{thm:idp-linf}. \qed

\subsection{Proof of Theorem \ref{thm: individual-privacy-lower}}
\begin{proof}[Proof of Theorem \ref{thm: individual-privacy-lower}]
	 Let \(r_{n}\) denote the minimax separation rate for the
	\emph{top-\(K\) recovery} problem under \((\varepsilon,\delta)\)-DP.
	Assume for contradiction that an estimator \(\hat\bth\) satisfies
	\[
	\sup_{\bth\in\Theta}
	\mathbb{E}\|\hat\bth-\bth\|_\infty
	\;=\;
	a_{n}
	= o(r_{n}).
	\]
	By Markov's inequality,
	\[
	\sup_{\bth\in\Theta}
	\Pr\bigl(\|\hat\bth-\bth\|_\infty \le 10a_{n}\bigr)
	\;\ge\; 0.9.
	\]
	Construct two parameter vectors that differ only in their
	\(K\)-th and \((K+1)\)-st order statistics with
	\(\theta_{(K)}-\theta_{(K+1)} = 100a_{n} = o(r_{n})\).
	Because these vectors are within \(10a_{n}\) in \(\ell_\infty\) norm,
	the above event forces their empirical rankings to coincide with
	probability at least \(0.9\).
	Consequently, a rank selector \(\psi(\hat\bth)\) exists with
	\(\Pr\bigl(\psi(\hat\bth)=[K]\bigr)\ge 0.9\),
	contradicting the impossibility of top-\(K\) recovery below \(r_{n}\).
	Therefore \(a_{n}\gtrsim r_{n}\), yielding the stated lower bound.
\end{proof}	

\subsection{Proof of Theorem \ref{thm: individual ranking nonparametric upper bound}}
\label{sec: proof of thm: individual ranking nonparametric upper bound}

Observe that 
\begin{align*}
	\bigcap_{a \in \mathcal S_{k-u}, b \in (\mathcal S_{k+u})^c }\{N_a + W_a > N_b + W_b\} \subseteq \{d_H(\widetilde {\mathcal S}_k, \mathcal S_k) \leq  2u\}.
\end{align*}
The union bound then implies
\begin{align}
	\Pro(d_H(\widetilde {\mathcal S}_k, \mathcal S_k) >  2u) \leq \sum_{a \in \mathcal S_{k-u}, b \in (\mathcal S_{k+u})^c} \Pro(N_a + W_a \leq N_b + W_b) \label{eq: individual ranking nonparametric upper bound expansion 1}
\end{align}
Fix $a \in \mathcal S_{k-u}$ and $b \in (\mathcal S_{k+u})^c$. We have 
\begin{align*}
	\E N_a - \E N_b = n\frac{mL}{{n \choose 2}}(\tau_a - \tau_b) \geq  n\frac{mL}{{n \choose 2}}(\tau_{(k-u)} - \tau_{(k+u+1)}) \geq C\left(\sqrt{\frac{m\log n}{n}} + \frac{\log n}{\varepsilon}\right),
\end{align*}
for some $C > 0$, and therefore
\begin{align}
	\Pro(N_a + W_a \leq N_b + W_b) &= \Pro(N_a - N_b \leq W_b - W_a) \notag\\
	&\leq \Pro(N_a - N_b - (\E N_a - \E N_b) \leq -C\sqrt{\frac{m \log n}{n}}) + \Pro(W_b-W_a \geq C\log n/\varepsilon) \label{eq: individual ranking nonparametric upper bound expansion 2}.
\end{align}

To bound the first term $\Pro(N_a - N_b - (\E N_a - \E N_b) \leq -C\sqrt{\frac{m \log n}{n}})$, we have
\begin{align*}
	N_a - N_b &= \sum_{k,l} (\1(Y^{(k,l)}_{ab} = 1) - \1(Y^{(k,l)}_{ba} = 1)) + \sum_{j \neq a, b}\left[\sum_{k,l} \1(Y^{(k,l)}_{aj} = 1) - \1(Y^{(k,l)}_{bj} = 1)\right] \\
	&:= \sum_{k,l}d^{(k,l)}_0 + \sum_{j \neq a, b}\sum_{k,l} d^{(k,l)}_j.
\end{align*}
As the $d^{(k,l)}_0$ and $d^{(k,l)}_j$'s are all independent and bounded by $-1$ and $1$, we bound the tail probability by Bernstein's inequality, as follows. The second moments of the $d^{(k,l)}_j$'s are bounded by
\begin{align*}
	&\E(d^{(k,l)}_0)^2 = \frac{1}{{n \choose 2}}(\rho_{ab} + \rho_{ba}), \, \E (d^{(k,l)}_j)^2 = \frac{1}{{n \choose 2}}\rho_{aj}\left(1-\frac{1}{{n \choose 2}}\rho_{bj}\right) + \frac{1}{{n \choose 2}}\rho_{bj}\left(1-\frac{1}{{n \choose 2}}\rho_{aj}\right) \leq\frac{1}{{n \choose 2}}(\rho_{aj} + \rho_{bj}).
\end{align*}
We then have
\begin{align*}
	\sum_{k,l}\left(\E (d^{(k,l)}_0)^2 + \sum_{j \neq a, b} \E (d^{(k,l)}_j)^2\right) &= \frac{mL}{{n \choose 2}} \left(\sum_{j \neq a} \rho_{aj} + \sum_{j \neq b} \rho_{bj}\right)\\
	&\leq  \frac{mL}{{n \choose 2}} (n\tau_a + n\tau_b) \leq 2n \frac{mL}{{n \choose 2}} \tau_a - n \frac{mL}{{n \choose 2}} (\tau_{(k)} - \tau_{(k+1)})\\
	&\leq 2n \frac{mL}{{n \choose 2}}  - C\left(\sqrt{\frac{m\log n}{n}} + \frac{\log n}{\varepsilon}\right).
\end{align*}
Bernstein's inequality now implies
\begin{align*}
	&\Pro\left(N_a - N_b - \E N_a - \E N_b \leq -C\sqrt{\frac{m\log n}{n}}\right) \\
	&\leq c\exp\left(-\frac{C^2\frac{m\log n}{n}}{4\frac{m}{n} - C\left(\sqrt{\frac{m\log n}{n}} + \frac{\log n}{\varepsilon}\right) + \frac{2C}{3}\left(\sqrt{\frac{m\log n}{n}} + \frac{\log n}{\varepsilon}\right)}\right) \leq n^{-C^2}.
\end{align*}
The last inequality follows from the fact that $\frac{n \log n}{m\varepsilon} =O(1)$ which follows from the fact that $\bm \rho \in \Theta(k, u, C)$.
Plugging into \eqref{eq: individual ranking nonparametric upper bound expansion 2}, we have
\begin{align*}
	\Pro(N_a + W_a \leq N_b + W_b) \leq n^{-C^2} + \exp\left(-(C/4)\log n\right) \leq n^{-C^2} + n^{-C/4}.
\end{align*}
The first inequality is true because $W_b$ and $-W_a$ are Laplace random variable with scale $L/\varepsilon$ where $L = O(1)$. 

Finally, by the union bound \eqref{eq: individual ranking nonparametric upper bound expansion 1}, we have $\Pro(\tilde {\mathcal S}_k \neq \mathcal S_k) \leq n^2(n^{-C^2} + n^{-C/4})$. Choosing a sufficiently large constant $C$ finishes the proof.


\subsection{Proof of Theorem \ref{thm: individual nonparametric ranking lower bound}}
\label{sec: proof of thm: individual nonparametric ranking lower bound}

By Theorem \ref{thm: private fano}, it suffices to construct probability matrices $\bm \rho_1, \ldots, \bm \rho_M \in \widetilde \Theta(k, u, c)$ satisfying the following conditions.
\begin{itemize}
	\item For any $a, b \in [M]$ and $a \neq b$, $d_H(\mathcal S^a_j, \mathcal S^b_j) > 4m$ , where $\mathcal S^a_k$ denotes the top-$k$ index set implied by $\bm \rho_a$ for each $a \in [M]$.
	\item For any $a, b \in [M]$ and $a \neq b$, let $\bm X = \left(\bm X^{(k,l)}\right)_{k\in[m], l \in [L]}$ and $\bm Y = \left(\bm Y^{(k,l)}\right)_{k \in[m], l\in [L]}$ denote pairwise comparison outcomes drawn from the distributions implied by $\bm \rho_a$ and $\bm \rho_b$ respectively, and let $D = \sum_{k \in [m]} \operatorname{TV}(\bm X^{(k,.)}, \bm Y^{(k,.)})$. We would like
	\begin{align}\label{eq: individual condition for expected hamming distance}
		\frac{9}{20}\left(\min\left(1, \frac{M-1}{e^{10\varepsilon D}}\right) - 10\delta(M-1)D\right) > \frac{1}{10}.
	\end{align} 
	and also
	\begin{align}\label{eq: condition for non private part}
		\left(1-\frac{\beta+\log2}{\log M}\right) > \frac{1}{10}.
	\end{align}	
\end{itemize}
It turns out that the choice of $\bm \rho$ for proving the non-private lower bound in \cite{shah2017simple} suits our purpose, which is described in detail in Section 5.3.2 of \cite{shah2017simple}. 

The assumed conditions $2u \leq (1 + \nu_2)^{-1}\min\{n^{1-\nu_1}, k, n-k\}$ in Theorem \ref{thm: individual nonparametric ranking lower bound} guarantees the existence of the construction by \cite{shah2017simple}, which involves $M = e^{\frac{9}{10}\nu_1\nu_2 u \log n}$ different matrices $\{\bm \rho\}_{a \in [M]}$. For each $a \in [M]$, the top-$k$ index set is given by $\mathcal S^a_k = [k - 2(1+\nu_2)u] \cup \mathcal B_a$, and each set in the collection $\{\mathcal B_a\}_{a \in [M]}$ satisfies 
\begin{align*}
	\mathcal B_a \subseteq \{n/2 + 1, \ldots, n\}, |\mathcal B_a| = 2(1+\nu_2)u, d_H(\mathcal B_a, \mathcal B_b) > 4u, \forall a \neq b. 
\end{align*} 
Now that the separation condition between $\bm \rho's$ are satisfied by construction, it remains to specify the elements of the $\rho$'s, namely the pairwise winning probabilities. For each $a \in [M]$, set
\begin{align*}
	(\bm \rho_a)_{ij} = \begin{cases}
		1/2 &\text{ if } i, j \in \mathcal S^a_k \text{ or } i, j \not \in \mathcal S^a_k,\\
		1/2 + \Delta &\text{ if } i \in \mathcal S^a_k \text{ and } j \not \in \mathcal S^a_k,\\
		1/2 - \Delta &\text{ if } i \not \in \mathcal S^a_k \text{ and } j \in \mathcal S^a_k.
	\end{cases}
\end{align*}
By construction, we have $\tau_{(k-u)} - \tau_{(k+u+1)} = \Delta$. The choice of $\Delta$, to be specified later, will ensure that the $\bm \rho$'s belong to $\widetilde \Theta(k, u, c)$ for appropriate $c$ and $D = \sum_{k \in [m]} \operatorname{TV}(\bm X^{(k,.)}, \bm Y^{(k,.)})$ satisfies the requisite condition \eqref{eq: individual condition for expected hamming distance}. To this end, we start by bounding $D$ in terms of $\Delta$. 

By construction of $\bm \rho_a$ and $\bm \rho_b$, $X^{(k,l)}_{ij}$ and $Y^{(k,l)}_{ij}$ are identically distributed unless $i \in \mathcal B_a \cup \mathcal B_b$ or $j \in \mathcal B_a \cup \mathcal B_b$. It follows that
\begin{align*}
	D =\sum_{k \in [m]} \operatorname{TV}(\bm X^{(k,.)}, \bm Y^{(k,.)}) &\leq \sum_{k \in [m]}\sum_{l\in [L]}\sum_{i,j}  \operatorname{TV}(X^{(k,l)}_{i,j}, Y^{(k,l)}_{i,j})\\
	&\leq mL\left( \sum_{i \in \mathcal B_a \cup \mathcal B_b} \operatorname{TV}(X^{(1,1)}_{ij}, Y^{(1,1)}_{ij}) + \sum_{j \in \mathcal B_a \cup \mathcal B_b} \operatorname{TV}(X^{(1,1)}_{ij}, Y^{(1,1)}_{ij}) \right)\\
	&\leq mL n |\mathcal B_a \cup \mathcal B_b| \cdot 2\frac{1}{{n \choose 2}}\Delta \leq 16(1+\nu_2)\frac{mLu}{n}\Delta.
\end{align*}
To satisfy condition \eqref{eq: individual condition for expected hamming distance}, recall $M = e^{\frac{9}{10}\nu_1\nu_2 m \log n}$ and $\delta \lesssim \left(u\log n \cdot n^{10u}/\varepsilon\right)^{-1}$. Setting $\Delta = c(\nu_1, \nu_2, L)\frac{n\log n}{m\varepsilon}$ for some sufficiently small $c(\nu_1, \nu_2)$ ensures $e^{10\varepsilon D} \lesssim M - 1$ and $10\delta D(M-1) \lesssim 1$, so that \eqref{eq: individual condition for expected hamming distance} holds.The choice of $\Delta = c(\nu_1, \nu_2, L)\frac{n\log n}{m\varepsilon}$ indeed ensures $\bm \rho_1, \ldots, \bm \rho_M \in \widetilde \Theta(k, u, c)$ for $c \leq c(\nu_1, \nu_2)$.

Next we prove the lower bound for the non private part. Fix $k,l$
\begin{align*}
	\mathrm{KL}(\bm X^{(k,l)},\bm Y^{(k,l)}) &= \sum_{i,j} \mathrm{KL}(X_{i,j}^{(k,l)}, Y_{i,j}^{(k,l)})\\
	&\leq \sum_{i \in \mathcal B_a \cup \mathcal B_b} \mathrm{KL}(X_{i,j}^{(k,l)}, Y_{i,j}^{(k,l)}) + \sum_{j \in \mathcal B_a \cup \mathcal B_b}\mathrm{KL}(X_{i,j}^{(k,l)}, Y_{i,j}^{(k,l)})\\
	&\leq c(\nu_2) \frac{n |\mathcal B_a \cup \mathcal B_b| }{{n \choose 2}} \Delta^2 \lesssim \frac{n u }{{n \choose 2}} \Delta^2
\end{align*}
By using the independence of the samples and tensorization property of the KL divergence we  have that
$$
\mathrm{KL}(\{\bm X^{(k,l)}\}_{k \in [m], l \in [L]},\{\bm Y^{(k,l)}\}_{k \in [m], l \in [L]}) \lesssim \frac{mLnu}{{n \choose 2}} \Delta^2
$$
Hence we have that condition $(b)$ of Theorem \ref{thm: private fano} is satisfied with $\beta = \frac{mLu}{n}\Delta^2$ and we can choose $\Delta = c'(\nu_1,\nu_2,L) \sqrt{\frac{n \log n}{m}}$ such that~\eqref{eq: condition for non private part} is satisfied. Hence if we choose $\Delta = C\left(\sqrt{\frac{n \log n}{m}} + \frac{n \log n}{m \varepsilon}\right)$ for $C \lesssim c(\nu_1,\nu_2,L) \wedge c(\nu_1,\nu_2,L)$ then either the lower bound(private or no private) would imply that estimation in Hamming distance is not possible.

For the restriction to parametric model $\rho_{ij} = F(\theta^*_i - \theta^*_j)$, we show that for every $a \in [M]$, the aforementioned choice of $\left(\bm \rho_a\right)_{ij}$ can be realized by some appropriate choice of $\bth^*$. By (A0) we have $F(x) = 1-F(x)$ and $F(0) = 1/2$; it then suffices to let $\theta^*_i = \alpha$ for all $i \in \mathcal S^a_k$, and let $\theta^*_i = \beta$ for all $i \not \in \mathcal S^a_k$. The two values $\alpha > \beta$ should be such that $F(\alpha - \beta) = 1/2 + \Delta$; they are guaranteed to exist since $\Delta$ is sufficiently small by assumption and $F$ is strictly increasing.

\spacingset{1.6} 
\bibliographystyle{abbrv}
\bibliography{reference}	\label{lastpage}

\end{document}